\documentclass[11pt]{amsart}
 \usepackage{latexsym,amsfonts,amssymb}
 \usepackage{amsmath}

 \usepackage[top=30truemm,bottom=30truemm,left=25truemm,right=25truemm]{geometry}

 \usepackage{amsthm}
 \usepackage{indentfirst}
 \usepackage{enumerate}
 \usepackage{braket}
 \usepackage{verbatim}
 \usepackage{color}
 \usepackage{colordvi,multicol}
 \usepackage[marginal]{footmisc}
 \usepackage{bbm}
 \usepackage[backref]{hyperref}
 \usepackage{lineno}

\usepackage{fancyhdr}

 \allowdisplaybreaks[4]

 \makeatletter
 
 \newcommand{\Rmnum}[1]{\expandafter\@slowromancap\romannumeral #1@}
 \makeatother

 \def \4l{\ \ \ \  }
 \def \dpa{    \overline{\partial}    }
 \def \pa{    {\partial}    }
 \def \Supp{   {\mathrm{Supp}}\    }
 \def \cal{\mathcal}

 \def \C{ {  \mathbb{C}  } }
 \def \R{ {  \mathbb{R}  } }
 \def \im{ { \rm  i  } }
 \def \Psh{ \mathrm{Psh}  }
 \def \Spsh{ \mathrm{Spsh} }

 \def \llangle{  {  \langle\!\langle}  }
 \def \rrangle{  {  \rangle\!\rangle}  }
 \def \loc{ {\rm loc}  }
 \def \vp{ {\varphi}  }
 \def \ve{  {\varepsilon}  }
 \def \Nak{ {\rm Nak}   }

\def \reg{{\rm reg}}

 \def \Exp{{\rm Exp}}

 \newtheorem{thm}{Theorem}
 \newtheorem*{theorem}{Theorem}
 \newtheorem{cor}[thm]{Corollary}
 \newtheorem{lem}[thm]{Lemma}
 \newtheorem{pro}[thm]{Proposition}
 \newtheorem{defi}[thm]{Definition}
 \newtheorem{rem}[thm]{Remark}
 \newtheorem{exa}[thm]{Example}

 \numberwithin{equation}{section}
 \numberwithin{thm}{section}
 \date{}

\begin{document}
    \title{\bf Multiplier Submodule Sheaves and a problem of Lempert}

    \author[Z. Liu]{Zhuo Liu}
    \address{Zhuo Liu:     Beijing Institute of Mathematical Sciences and Applications, Beijing 101408, China; Department of Mathematics and Yau Mathematical Sciences Center, Tsinghua University, Beijing 100084, China.}
    \email{liuzhuo@amss.ac.cn}

    \author[B. Xiao]{Bo Xiao}
    \address{Bo Xiao: School of Mathematical Sciences, Peking University, Beijing, 100871, China.}
    \email{xb1994@amss.ac.cn}

    \author[H. Yang]{Hui Yang}
    \address{Hui Yang: School of Mathematical Sciences, Peking University; Institute of Mathematics\\Academy of Mathematics and Systems Science\\Chinese Academy of Sciences\\Beijing 100190, China}
    \email{yanghui@amss.ac.cn}

    \author[X. Zhou]{Xiangyu Zhou}
    \address{Xiangyu Zhou: Institute of Mathematics\\Academy of Mathematics and Systems Science\\Chinese Academy of Sciences\\Beijing 100190, China}
    \email{xyzhou@math.ac.cn}



\thanks{The fourth author is supported by National Key R\&D Program of China (2021YFA1003100) and the NSFC grant (no. 12288201).}

\begin{abstract}
    In this article, we establish an  $L^2$ extension theorem for Nakano semi-positive singular Hermitian metrics on holomorphic vector bundles, and the strong openness and stability properties of  the multiplier submodule sheaves associated to Nakano semi-positive singular Hermitian metrics on holomorphic vector bundles.
    We solve affirmatively a question of Lempert on the preservation of Nakano semi-positivity under limit of an increasing metrics based on Deng-Ning-Wang-Zhou's    characterization of Nakano positivity.
\end{abstract}

 \maketitle

\tableofcontents

\section{Introduction}
    \subsection{Multiplier ideal sheaves and strong openness conjecture}

  Let $\varphi$ be a plurisubharmonic function on a complex manifold $X$. The multiplier ideal sheaf $\cal{I}(\varphi)$ is the sheaf of germs of holomorphic functions $f$ such that $|f|^2e^{-\varphi}$ is locally integrable. The main idea actually goes back to the works of Bombieri \cite{Bom70} and Skoda \cite{Skoda72}.

 As an invariant of the singularities of plurisubharmonic functions, multiplier ideal sheaves play an important role in relating several complex variables to complex algebraic geometry. It is well-known that multiplier ideal sheaves possess fundamental properties such as coherence, torsion-freeness and integral closedness. Furthermore, there have been intriguing conjectures regarding further properties of multiplier ideal sheaves, such as the openness conjecture and the strong openness conjecture proposed in \cite{DK2001, Demailly-note2000}.

    \textbf{Openness conjecture(OC)}: If $\cal{I}(\varphi)=\cal{O}_X$, then
    $$ \cal{I}_+(\varphi)=\cal{I}(\varphi);$$

    \textbf{Strong openness conjecture(SOC)}:
    $$  \cal{I}_+(\varphi)=\cal{I}(\varphi).$$
    Here $\cal{I}_+(\varphi):=\bigcup_{\varepsilon>0}\cal{I}((1+\varepsilon)\varphi)$.

 In the case of $\dim X=2$, {OC} was proved by Favre-Jonsson in \cite{FM05j}, while {SOC} was proved by Jonsson-Musta\c{t}\u{a} in \cite{JonssonMustata2012}, both employing an algebraic approach. For general dimensions, Berndtsson proved {OC} and established an effectiveness result of OC in \cite{Berndtsson2013openness}, using the complex Brunn-Minkowski theory developed by himself as a key tool.

  {SOC} was proved by Guan-Zhou in \cite{GuanZhouSOC2013} by using unusually the Ohsawa-Takegoshi $L^2$ extension theorem together with other observations. Lempert discussed SOC in the vector bundle setting in \cite{Lempert17}.
 
  Moreover, Guan-Zhou proved an effectiveness result of {SOC} in \cite{GZ15}, implying some generalized versions of {OC} and {SOC}, namely the Demailly-Koll\'ar conjecture and the Jonsson-Musta\c{t}\u{a} conjecture.

	Furthermore, in \cite{GLZ2016} Guan-Li-Zhou found and proved the following stability property of the multiplier ideal sheaves:

\begin{theorem}
    Let $\{\varphi_j\}_j$ be a sequence of negative plurisubharmonic functions
    on $D\ni o$ which is convergent to $\varphi$ locally in measure and $\cal{I}(\varphi_j)_o\subset\cal{I}(\varphi)_o$. Let
    $\{F_j\}$ be a sequence of holomorphic functions on $D$ with $(F_j , o) \in \cal{I}(\varphi)_o$, which
    is compactly convergent to a holomorphic function $F$. Then $|F_j|^2e^{-\varphi_j}$ converges to $|F|^2e^{-\varphi}$ in $L^1_{\rm{loc}}$ near $o$.
\end{theorem}

\subsection{$L^2$ extension theorem and strong openness on vector bundles}

 The above results on strong openness and stability can be reformulated within the context of holomorphic line bundles with singular Hermitian metrics. Let $(L,h)$ be a pseudoeffective line bundle, i.e., the curvature current of $h$ is semi-positive, then the associated multiplier ideal sheaf $\cal{I}(h)$ satisfies the strong openness.

It is natural to ask whether the analogous multiplier submodule sheaves, associated to singular Hermitian metrics on holomorphic vector bundles, satisfy  similar properties. Lempert initiated studying such a question in \cite{Lempert17}. In the present paper, motivated by Lempert's work, we address this question by two different methods and obtain some additional properties concerning multiplier submodule sheaves. During our research process, we observed an affirmative answer to a question of Lempert on the preservation of Nakano semi-positivity under limit of an increasing metrics.

 Let $X$ be a complex manifold, $E$ be a holomorphic vector bundle of rank $r$ over $X$.
    A singular Hermitian metric on $E$ is a measurable map $h$ from the base
    manifold $X$ to the set of Hermitian forms on the fibers of $E$, which assigns to almost every
    point $x \in X$ a Hermitian form $h(x)$ on $E_x$ satisfying $0 < \det h(x) < +\infty$.

\begin{defi}[\cite{Cataldo98}]
    The multiplier submodule sheaf $\mathcal{E}(h)$ of $\mathcal{O}(E)$ associated to $(E,h)$  is the sheaf of the germs $s_x \in \mathcal{O}(E)_x$
    such that $|s_x|^2_h$ is integrable in some  neighborhood of x.
\end{defi}

 We need to impose some positivity assumption on $h$. It is well-known that, when $h$ is smooth, $h$ is Griffiths semi-positive if and only if $\log|u|^2_{h^*}$ is plurisubharmonic for any local holomorphic section $u$ of the dual bundle. This characterization naturally leads to a definition of Griffiths semi-positivity for singular Hermitian metrics (see \cite{BP08,Rau15,PT18}).

	 In order to obtain further results on multiplier submodule sheaves, we need to consider Nakano positivity. Notably, when the base manifold is one-dimensional, Nakano positivity coincides with Griffiths positivity for smooth Hermitian metrics. Naturally, it comes to mind that one might use the Chern curvature current to define the Nakano positivity for singular Hermitian metrics. However, in \cite{Rau15}, Raufi pointed out that the Chern curvature current of a Griffiths semi-positive singular Hermitian metric may not have measure coefficients.

To overcome this difficulty,
 Lempert \cite{Lempert17} introduced a new positivity concept for singular Hermitian metrics $h$ on holomorphic Hilbert bundles: the Chern curvature of $h$ is said to dominate 0 in the sense of Nakano if there exists a sequence of $C^2$-smooth Nakano semi-positive Hermitian metrics that converges increasingly to $h$.
  He further established an H\"ormander-Skoda type $L^2$-estimation, an Ohsawa-Takegoshi type $L^2$-extension theorem, and a strong openness property for a singular Hermitian metric with the Chern curvature dominating 0 in the sense of Nakano.

  In the same paper, Lempert posed a fundamental problem as follows.

$\bullet$ \textbf{Problem} \cite{Lempert17}
   Is a $C^2$ Hermitian metric whose Chern curvature dominates 0 in the sense of Nakano still Nakano semi-positive in the usual sense?

   This is a natural and beautiful question.

  Recently, Deng-Ning-Wang-Zhou in \cite{DNWZ20} introduced the notion of ``optimal $L^2$-estimate condition" and presented a characterization of Nakano positivity for smooth Hermitian metrics in terms of this notion. We observe that this characterization actually leads to an affirmative answer to Lempert's problem (see section 2.3).

\begin{theorem}[\cite{DNWZ20}]
	Let $h$ be a $C^2$-smooth Hermitian metric on holomorphic Hilbert bundle  $E$, then $h$
	is Nakano semi-positive if and only if $h$ is called to be $L^2$ optimal, that is,
	for any Stein open subset $U$ such that $E|_U$ is trivial, any K\"ahler form $\omega_U$ on $U$, any $\psi \in {\rm Spsh}(U)\cap C^\infty(U)$ and any $\overline{\partial}$-closed f $\in C^\infty_{(n,1)} (U,E|_U)$, there exists  $u\in  L^2_{(n,0)} (U,\omega_U,E|_U,he^{-\psi})$ such that $\overline{\partial}u = f$ and
\begin{equation}\label{def Hormander estimate}
	\int_{U}^{} |u|^2_{\omega_U,h} e^{-\psi} dV_{\omega_U}
	\le \int_{U}^{} \langle B^{-1}_{\omega_U,\psi} f,f \rangle_{\omega_U,h} e^{-\psi} dV_{\omega_U},
	\end{equation}
	provided that the right hand side is finite, where $B_{\omega_U,\psi}=[\im\partial\overline{\partial}\psi \otimes {\rm Id}_E, \Lambda_{\omega_U}]$.
\end{theorem}
 The ``only if" part is due to H\"ormander, Skoda, Demailly, and Lempert for the case of Hilbert bundles.\\

Recently, based on the above characterization, Inayama \cite{In20} defined the Nakano semi-positivity for singular Hermitian metrics and showed that the multiplier submodule sheaf $\mathcal{E}(h)$ associated to a Nakano semi-positive singular Hermitian metric $h$ is coherent. The reader is referred to Section \ref{ssect GSP} and \ref{ssect NSP} for more details.

\begin{defi}[\cite{In20}]\label{def NSP}
Let $h$ be a singular Hermitian metric on a holomorphic vector bundle $E$ over a complex manifold $X$.
  \begin{enumerate}
    \item[$(1)$]  $h$ is said to be Nakano semi-positive if $h$ is Griffiths semi-positive and is $L^2$ optimal.
    \item[$(2)$] $h$ is said to be locally Nakano semi-positive if  for any point $x\in X$, there exists an open neighborhood $U$ of $x$ such that $(E|_U,h|_U)$ is Nakano semi-positive .
    \item[$(3)$] $h$ is said to be (locally) Nakano positive if locally there exists a smooth strictly plurisubharmonic function $\psi$ such that $he^{\psi}$ is  (locally) Nakano semi-positive.
  \end{enumerate}	
\end{defi}

 These definitions also make sense for singular Hermitian metrics on holomorphic Hilbert bundles. However, with the exception of the solution to Lempert's problem (Proposition \ref{prop he^}), our other results heavily rely on the existence of $(\det E,\det h)$. Consequently, it is difficult to extend our method to holomorphic vector bundles of infinite rank. It is noteworthy that the importance of determinant metrics in the study of Griffiths semi-positive metrics was first noticed by Raufi \cite{Rau15}.

Now let us present our main theorems in this paper.
Firstly, we obtain an Ohsawa-Takegoshi type $L^2$ extension theorem for a holomorphic vector bundle with a Nakano semi-positive singular Hermitian metric.

Let $(X,\omega)$ be an $n$-dimensional Stein manifold possessing a K\"ahler metric $\omega$, and  $s=(s_1,\cdots,s_k)$, where $s_i(1\leq i\leq k)$ are global holomorphic functions on $X$.
Assume that $s$ is transverse to the zero section $$S:=\{z\in X : s(z)=0\},$$ and let
$\dim S=n-k$.
Let $\phi$ be a plurisubharmonic function on $X$ and  $e^{-\phi} |_S\in L^1_{\loc}(S)$.
Denote $r=|s|^2=\sum |s_i|^2$.
Given such an $r$, any continuous volume form $d\mu$ on $X$ induces a volume form $d\mu_r$ on $Y$ as follows.
Suppose $\theta\in C(X)$, then $d\mu_r$ will satisfy
\begin{equation}
\int_S \theta d\mu_r=\lim_{\varepsilon\to 0}\varepsilon^{-2k}
\int_{\{z\in X\colon r(z) < \varepsilon^2\}} \theta d\mu.
\end{equation}
Locally $d\mu_r$ can be computed if we introduce local coordinates $z_1,\cdots,z_n$ on $X$ so that $s_i= z_i$.
If $d\mu=\im^n\alpha dz_1\wedge d\bar z_1\cdots \wedge dz_n\wedge d\bar z_n$ and $v_k$ is the volume of the unit ball in $\C^k$,
then $d\mu_r= \im^{n-k}v_k\alpha dz_{k+1}\wedge d\bar z_{k+1}\cdots \wedge dz_n\wedge d\bar z_n$.
Direct computation gives that $d\mu_r$ induced by $dV_{X,\omega}$ is $dV_{S,\omega}/|\Lambda^k ds|^2$.
\begin{thm}[Theorem \ref{main1}]\label{thm l2ext}
Let $X,\omega,s, S, r$ be as above. Assume that $r\leq e^{-2}$ on $X$.
Let $E\to X$ be a holomorphic vector bundle with a Nakano semi-positive singular Hermitian metric $h$. Assume that $\varphi\not\equiv-\infty$ on every connected component of $S$, where $\varphi=- \log \det  h$ is the local weight.
	Let $f$ be a holomorphic section of $K_X\otimes E|_S$ over $S$ and
$$
\int_S|f|^2_hd\mu_r<+\infty,
$$
then there is a holomorphic section $F$ of $K_X\otimes E$ such that $F|_S=f$ on $S$ and
\begin{equation*}
\int_X  \frac{|F|^2_h}{ r^k(\log r)^2} dV_{X,\omega}\leq C \int_S|f|^2_h d\mu_r,
\end{equation*}
where $C$ is a constant only depending on the codimension of $S$.
\end{thm}

 In \cite{Hor1965}, H\"ormander obtained an $L^2$ extension theorem from a single point for Lipschitz continuous plurisubharmonic weight by directly solving the $\dpa$-equation. Recently, Hosono-Inayama \cite{HI19} obtained an $L^2$ extension theorem from a single point for locally H\"older continuous Hermitian metrics on holomorphic vector bundles and Deng-Ning-Wang-Zhou \cite{DNWZ20}
 obtained an $L^2$ extension theorem from a single point for continuous Hermitian metrics on holomorphic vector bundles. In their cases, the continuity assumptions are used to control the integrals over tubular neighborhoods.

 In our case, using the minimal solution, we derive an $L^2$ extension theorem of openness type \cite{Jen76,chenboyong,XZ22} from the $\dpa$-equations, rather than the twisted $\dpa$-equation. Subsequently, we apply a generalized Siu's lemma established by Zhou-Zhu \cite{Siu Lemma} to control the integrals over tubular neighborhoods of the submanifold.
Moreover, using the generalized Siu's lemma, we also show that the restriction of a locally Nakano semi-positive singular Hermitian metric onto a ``generic" complex submanifold is also locally Nakano semi-positive (see Proposition \ref{pro nak res nak}).

 The Ohsawa-Takegoshi $L^2$ extension theorem \cite{OT87} has been deeply developed and becomes one of the most powerful tools in complex analytic and algebraic geometry (see \cite{O18,O20,Z15}). Recently, Guan-Zhou \cite{GZ15} solved the strong openness conjecture by using the $L^2$ extension theorem from a movable subvarieties.
 Subsequently, Lempert \cite{Lempert17} overcame the decisive difficulties in extending Guan-Zhou's method to holomorphic Hilbert bundles. Building upon Lempert's work, it follows from Theorem \ref{thm l2ext} and Proposition \ref{pro nak res nak} that the multiplier submodule sheaves associated to Nakano semi-positive singular Hermitian metrics on holomorphic vector bundles satisfy the strong openness property.

\begin{thm}[Theorem \ref{soc}]\label{mainthm soc}
     Let $E \rightarrow X$ be a holomorphic vector bundle and $h_1 \ge h_2\ge \cdots$ be Nakano semi-positive singular  Hermitian metrics on $E$. Suppose that $h = \lim_j h_j$
    is bounded below by a continuous Hermitian metric, then $\bigcup_j \mathcal{E}(h_j) = \cal{E}(h)$.
\end{thm}

 By Proposition \ref{prop he^}, a singular Hermitian metric on $E$ whose Chern curvature dominates $0$ in the sense of Nakano must be Nakano semi-positive in the sense of Definition \ref{def NSP}. Therefore, Theorem \ref{mainthm soc} generalizes Lempert's result in the case of holomorphic vector bundles.

Moreover, together with Lemma \ref{lem e^(phi-phi_j) bounded},  we obtain the following strong openness property in a broader setting.

\begin{thm}[Corollary~\ref{main thm2,usual case}]\label{thm 2}
    Let $X$ be a complex manifold, $E$ be a holomorphic vector bundle on $X$ with  a  locally Nakano semi-positive singular Hermitian metric $h$ and Griffiths semi-positive singular Hermitian metrics $h_j,\ j\ge1$. If  $h_j\ge h$ and
    $-\log \det h_j$ converges to $-\log\det h$ locally in measure, then  $ \bigcup_{j} \mathcal{E}(h_j)=\mathcal{E}(h)$.
\end{thm}

\begin{rem}\label{rmk ge lempert}
        If  $h_j\ge h$ and $-\log \det h_j$ converges to $-\log\det h$ locally in measure, then $h_j$ converges to $h$ locally in measure.
        Moreover,
        we can take a subsequence $\{h_{j_k}\}$ such that $h_{j_k}$ converges to $h$ almost everywhere.
\end{rem}

      Secondly, similar to \cite{GZ15} and \cite{Guan19}, we define several invariants for singular Hermitian metrics on holomorphic vector bundles and subsequently derive an effectiveness result of the strong openness property for (locally) Nakano semi-positive singular Hermitian metrics.

\begin{defi}\label{def C_F}
    Let $o\in D$ $\subset \mathbb{C}^n$  be open, $E$ be a trivial holomorphic vector bundle with a Nakano semi-positive singular Hermitian metric $h$. Let $F \in \mathcal{O}(E)(D)$,

\begin{enumerate}
    \item[\rm {(1)}]
    the complex singularity exponent associated to $h$ at $o$ is
     $$c^F_o(h):=\sup \{\ \beta \ge 0\  |\ (F, o)\in \mathcal{E} (h(\det h)^{\beta} )_o   \},$$
     the complex singularity exponent associated to $h$ on a compact set $K$ is
     $$c^F_K(h):=\sup \{\ \beta \ge 0\  |\ (F, x)\in \mathcal{E} (h(\det h)^{\beta} )_x {\rm~ for~ any ~} x\in K   \},$$
    \item[\rm {(2)}]
    for any open subset $o\in U\subset D$,
\begin{align*}
    C_{F,\beta}(U):&=\inf \left\{\int_U |\widetilde{F}|^2_{\frac{h}{\det h}}~\Big| ~\widetilde{F} \in \mathcal{O} (E)(D)\ \ \mathrm{and}\ \ (\widetilde{F}-F, o)\in \mathcal{E} (h(\det h)^\beta )_o   \right\},\\
    C_{F,\beta^+}(U):&=\inf \left\{\int_U |\widetilde{F}|^2_{\frac{h}{\det h}} ~ \Big | ~\widetilde{F} \in \mathcal{O} (E)(D)\ \ \mathrm{and}\ \ (\widetilde{F}-F, o)\in \underset{\beta' > \beta}
    {\bigcup}\mathcal{E}  (h(\det h)^{\beta'} )_o \right\}.
    \end{align*}
    \end{enumerate}
\end{defi}

\begin{thm}[Theorem~\ref{main thm1, effective }]\label{thm 1}
    Let $o\in D\subset \mathbb{C}^n$ be a Stein domain, $E$ be a trivial holomorphic vector bundle over $D$ with a Nakano semi-positive singular Hermitian metric $h$.
    Assume $\varphi=-\log \det h\ <\ 0$ on $D$ and $\varphi(o)=-\infty$. Then for any $F\in \mathcal{E}(h)$ satisfying $\int_D |F|^2_h < +\infty$, we have $F\in \mathcal{E}(h(\det h)^{\beta})_o$ if
\begin{align*}
    \theta(\beta)\ >\ \frac{\int_D |F|^2_h}{{C}_{F,c^F_o(h)^+}(D)},
    \end{align*}
    where
    $
    \theta(\beta) := 1+ \int_{0}^{+\infty} (1-e^{-g_\beta(t)}) e^t dt $
    and $
    g_\beta(t):=\int_{t}^{+\infty} \frac{ds}{se^{1+(1+\beta)s}}
    $.\\
    In particular,
    $\bigcup_{\beta>0}\mathcal{E}(h(\det h)^{\beta})=\mathcal{E}(h)$.
\end{thm}

 Recall that, Guan-Zhou \cite{GZ15} established an effectiveness result for the strong openness by estimating the volume growth of sublevel sets of a plurisubharmonic function. This estimate essentially goes back to \cite{GZ2015extension}, and its proof involves a twisted $\dpa$-equation as well as an accurate calculation of the curvature term. In our case, due to Raufi's counterexample, the Chern curvature current may not have measure coefficients. This leads us to consider deriving a volume estimate directly from the $\dpa$-equation. Although the estimate given by this approach is not sharp, it is sufficient to establish the strong openness property.

    In addition, we generalize the strong openness property obtained by Lempert \cite{Lempert17} to the following theorem.

\begin{thm}[Theorem~\ref{thm generalized Lempert}]\label{thm 2'}
    Suppose $E \rightarrow X$ is a holomorphic vector bundle with a singular Hermitian metric $h$ which is locally bounded below and $\{h_j\}$ are Nakano semi-positive  singular Hermitian metrics on $E$. If  $h_j\ge h$ and $-\log \det h_j$ converges to $-\log\det h$ locally in measure, then $\sum_j \mathcal{E}(h_j) = \cal{E}(h)$.
\end{thm}

\begin{rem}
    Since $h_j$ is Nakano semi-positive and  $h_j\ge h$, $\mathcal{E}(h_j)$ is a coherent submodule sheaf of $\cal{E}(h)$. Then $\{\sum_{k\le j}  \mathcal{E}(h_k)\}$ is a increasing sequence of coherent submodule sheaves of $\cal{E}(h)$. Therefore the strong Noetherian property of coherent analytic sheaf {\rm(\cite{Demaillybook2012}[\rm{Chapter \Rmnum{2}}-(3.22)])} implies that $\sum_j \mathcal{E}(h_j)=\bigcup_j \sum_{k\le j} \mathcal{E}(h_k) $ is also a coherent submodule sheaf of $\cal{E}(h)$. It is obvious that
    if $\{h_j\}$ is a sequence of decreasing singular Hermitian metrics, then $\sum_j \cal{E}(h_j)=\bigcup_j \cal{E}(h_j)$.
\end{rem}

Lastly, as an application of  the strong openness property, we prove the following stability theorem on holomorphic vector bundles.

\begin{thm}[Corollary~\ref{cor 1}]\label{thm 3}
    Let $E$ be a holomorphic vector bundle with a locally Nakano semi-positive singular Hermitian metric $h$ over a complex manifold $X\ni o$.
    Let $h_j$ be Griffiths semi-positive on $E$ with $h_j\ge h$ and $F_j\in \cal{O}(E)$.
    Assume that \\
    $(1)$ $\varphi_j=-\log\det h_j$ converges to $\varphi=-\log\det h$ locally in measure,\\
    $(2)$ $(F_j,o)\in \cal{E}(h)_o$ and $F_j$ compactly converges to $F$ on $X$.\\
    Then $|F_j|^2_{h_j}$ converges to $|F|^2_h$ in $L^p$ in some neighborhood of $o$ for any $0<p<1+c_o^F(h)$.
\end{thm}

	\vskip0.3cm
	 Now let's summarize some key ideas used in the proofs of this paper.

$\bullet$  For a Griffiths semi-positive singular Hermitian metric $h$ on the trivial vector bundle, we always use the following decomposition
 $$h=\frac{h}{\det h}\cdot \det h.$$
	 The advantage of such decomposition is that both $\frac{h}{\det h}$ and $\det h$ have nice properties. On the one hand, by Lemma \ref{pro h/deth}, $\frac{h}{\det h}$ is Griffiths semi-negative. Therefore, each entry of $\frac{h}{\det h}$ is a complex linear combination of nonnegative plurisubharmonic functions and hence locally bounded. It also follows from {\rm Lemma \ref{lem int of |f|^2e^phi dominates |f|}-(2)} that $\ker \dpa\cap L^2_{(n,0)}(U,\omega,E,\frac{h}{\det h})$ is a Hilbert space. On the other hand, $\det h$ is Griffiths semi-positive, $-\log\det h$ is a plurisubharmonic function and $\det h$ locally dominates $h$ (by {\rm Lemma \ref{prop lowerbound}}).

This trick essentially converts high-dimensional objects into one-dimensional objects and converts Griffiths semi-positive singular Hermitian metrics into plurisubharmonic functions. Consequently, a range of tools can be applied to study Griffiths semi-positive singular Hermitian metrics, including the pluripotential theory, the multiplier ideal sheaf and the $L^2$ method. This is why we always assume that the singular Hermitian metric is Griffiths semi-positive.

$\bullet$	
	 Guan-Zhou \cite{GuanZhouSOC2013} solved the strong openness conjecture on holomorphic line bundles by an induction on dimensions. Their method relies on the fundamental fact that the restriction of a plurisubharmonic function to a complex submanifold remains its plurisubharmonicity. However, it seems much unclear whether a Nakano semi-positive singular Hermitian metric is still (locally) Nakano semi-positive when restricts to a ``generic" complex submanifold.

 An intuitive approach is to solve the $\dpa$-equation in a tubular neighborhood of the submanifold and then restrict everything to the submanifold. The main difficulty is how to deduce the required estimate on submanifold from those estimates on tubular neighborhoods. Fortunately, a generalized Siu's lemma proved by Zhou-Zhu \cite{Siu Lemma} provides a solution. Interestingly, the proof of the generalized Siu's lemma employed the Ohsawa-Takegoshi $L^2$ extension and the strong openness property on holomorphic line bundles, while we use the generalized Siu's lemma to derive an Ohsawa-Takegoshi type $L^2$ extension theorem and the strong openness property for singular  holomorphic vector bundles.

$\bullet$
	The proof of the effectiveness of strong openness is quite detailed and tedious. To facilitate readers,
	the reader can refer to the proof of Theorem \ref{thm generalized Lempert} to get a glimpse of it.

	In order to obtain the stability of the strong openness, Lemma \ref{lem e^(phi-phi_j) bounded} is a key observation which ia a corollary of Guan-Li-Zhou's stability theorem  \cite{GLZ2016}.

\subsection{Notation}
    Let $X$ be a complex manifold, we can endow $X$ with a measure.
    Let $\{(U_\alpha,\rho_\alpha),\alpha\in I\}$ be an atlas of $X$.
    We say that $S\subset X$ is Lebesgue-measurable if $\rho_\alpha(S\cap U_\alpha)$ is Lebesgue-measurable in $\C^n$ for any $\alpha$.
    One can check that this is well-defined because Lebesgue-measurable sets in $\C^n$ are mapped to Lebesgue-measurable sets via any biholomorphic map.
    Then $X$ is a measurable space.
    Let $dV$ be the volume form on $X$ which is induced by some Hermitian metric of $X$.
    Of course, $dV$ is also a measure on $X$.
    In fact, $(\rho_\alpha)_*dV$ is equivalent to the Lebesgue measure on $\rho_\alpha(U_\alpha)$.

    We often omit the volume $dV$ in all integral representations if there is no confusion. For example, write $\int_D |f|^2 dV$ as $\int_D |f|^2$.

    For a smooth vector bundle $E$ over $X$, we say that a section $s$ of $E$ is measurable if $s=\sum s_j\sigma_j$ such that each $s_j$ is a measurable function on $X$ where $(\sigma_j)$ is a smooth frame of $E$.

    We always assume that each plurisubharmonic function appearing in all  conditions and assumptions does not identically equal to $-\infty$ in any component of $X$.

	For a matrix $A$, we denote the transpose of $A$ by $A^T$.

\subsection{Organization}
The rest sections of this paper are organized as follows.
     Section 2 is about plurisubharmonic functions and singular Hermitian metrics on holomorphic vector bundles.
     In particular, we get a useful lemma for plurisubharmonic functions in section 2.1,
     some properties of  Griffiths/Nakano semi-positive  singular Hermitian metrics on holomorphic vector bundles in section 2.2 and 2.3 respectively and recall the general Siu's lemma obtained by Zhou-Zhu in section 2.4.
     Section 3 is devoted to the proof of an Ohsawa-Takegoshi type $L^2$ extension and its applications, including the strong openness property (section 3.3) of the multiplier submodule sheaves for Nakano semi-positive singular Hermitian metrics on holomorphic vector bundles.
     Section 4  is devoted to the proof of  an effectiveness result of the strong openness, which implies the strong openness property independently.
     Section 5 is devoted to prove the stability of the complex singularity exponent related Nakano semi-positive singular Hermitian metrics.
     Section 6 is about some examples of Nakano semi-positive singular Hermitian metrics, including a generalization of Berndtsson's famous work on the direct image bundles.

\section{Preliminaries and useful results}

\subsection{Properties of plurisubharmonic functions}
Firstly, let us recall and give some properties of plurisubharmonic functions.

\begin{lem}[\cite{Skoda72}]\label{lem Skoda}
    Let $\varphi$ be a plurisubharmonic function, $\nu(\varphi,x)$ be the Lelong number of $\varphi$ at $x$. If $\nu(\varphi,x)<2$, then $e^{-\varphi}$ is integrable near $x$.
\end{lem}

\begin{lem}[{\bf Demailly's approximation theorem}, \cite{Demailly2012}]
    Let $\varphi$ be a plurisubharmonic function on a bounded pseudoconvex open set $D\subset \C^n$.
    For every $m>0$, let $\mathcal{H}_D(m\varphi)$ be the Hilbert space of holomorphic functions $f$ on $D$ such that $\int_{D}|f|^2e^{-2m\varphi}d\lambda<+\infty$.
    Let $\varphi_m=\frac{1}{2m}\log \sum|\sigma_l|^2$ where $(\sigma_l)$ is an orthonormal basis of $\mathcal{H}_D(m\varphi)$.
    Then there are constants $C_1,\ C_2>0$ independent of $m$ such that
    $$ \varphi(z)-\frac{C_1}{m}\le\varphi_m(z)\le\sup\limits_{|\zeta-z|<r}\varphi(\zeta)+\frac{1}{m}\log\frac{C_2}{r^n}$$ for every $z\in D$ and $r<d(z,\partial D)$.
    In particular, $\varphi_m$ converges to $\varphi$ pointwise and in $L^{1}_{\mathrm{loc}}$ topology on $D$ when $m\rightarrow+\infty$.
\end{lem}

\begin{lem}[\cite{Hormander2007notions},\cite{Guedj2017degenerate}]\label{prop Hartos lemma}
    Let $\{\varphi_j\}$ be a sequence of plurisubharmonic functions.
    If $\{\varphi_j\}$ is locally uniformly bounded from above, then either $\{\varphi_j\}$ compactly converges to $-\infty$, or there is a subsequence of $\{\varphi_j\}$ converging in $L^p_{\mathrm{loc}}$ for any $0<p<+\infty$.
\end{lem}

\begin{lem}\label{lem equiv of convergence}
    Let $\varphi,\varphi_j$ be plurisubharmonic functions on a domain. Assume that $\{\varphi_j\}$ is locally uniformly bounded from above.
    If $\varphi_j$ converges to $\varphi$ locally in measure, then $\varphi_j$ converges to $\varphi$ in $L^p_{\mathrm{loc}}$ for any $0<p<+\infty$.

\begin{proof}
        It is known that $\varphi\in L^p_{\mathrm{loc}}$ and $\varphi_j\in L^p_{\mathrm{loc}}$ for large $j$. If $\varphi_j$ does not converge to $\varphi$ in $L^p_{\mathrm{loc}}$, there is a compact subset $K$
        such that $ \limsup_j \int_K |\varphi_j-\varphi|^p=\lim_k \int_K |\varphi_{j_k}-\varphi|^p>0$. Since $\varphi_{j_k}$ is locally uniformly bounded from above and $\varphi_j$ converges to $\varphi$ locally
        in measure, there is a subsequence $\{\varphi_{j_{k_i}}\}$ of $\{\varphi_{j_k}\}$
        converging to a plurisubharmonic function $\widetilde{\varphi}$ in $L^p_{\mathrm{loc}}$ by Lemma~\ref{prop Hartos lemma}.
        Hence $\widetilde{\varphi}=\varphi$ almost everywhere and $\lim_i \int_K |\varphi_{j_{k_i}}-\varphi|^p=0$, which leads to a contradiction.
    \end{proof}
\end{lem}

\begin{rem}
    In {\rm\cite{GLZ2016}}, Guan-Li-Zhou gave another proof of {\rm Lemma~\ref{lem equiv of convergence}} by using their stability theorem which is independent of  {\rm Lemma~\ref{prop Hartos lemma}}.
\end{rem}

    Based on Guan-Li-Zhou's stability theorem, we can improve convergence of plurisubharmonic functions, which is a key lemma to imply the stability from the strong openness of multiplier submodule sheaves.

\begin{lem}\label{lem e^(phi-phi_j) bounded}
    Let $\varphi $ and $\varphi_j$ be plurisubharmonic functions. If $\varphi_j\le \varphi$
    converges to $\varphi$ locally in measure, then $e^{\varphi-\varphi_j}$ converges to $\rm {1}$ in $L^p_{\mathrm{loc}}$ for any $0<p<+\infty$.
\end{lem}

\begin{proof}
	
    We first consider the case of $p=\frac{1}{2}$.
    The result is local, so we only prove it on a bounded Stein neighborhood $D$ of $o$ in $\mathbb{C}^n$.
    By Demailly's approximation theorem, there exists $C_1>0$  such that for any $m\ge1$ we have $\varphi -\frac{C_1}{m} \le \widetilde{\varphi}_m$ on a compact neighborhood $K$ of $o$, where
    $\widetilde{\varphi}_m = \frac{1}{2m} \log\sum_l |\sigma_l|^2$ and $(\sigma_l)$ is an orthonormal basis of $\mathcal{H}_D(m\varphi)$.
    In addition, there exist constants $C_2$ and $N$ such that $\sum_l |\sigma_l|^2 \le C_2 \sum_{l=1}^{N} |\sigma_l|^2$ on $K$ (see \cite{DPS01}).
    Let $U\subset K$ be an another neighborhood of $o$, which is to be determined later. Then by H\"older's inequality,
\begin{align*}
    \int_U e^{\varphi-\varphi_j}
    &\le \int_U e^{\frac{C_1}{m}+\widetilde{\varphi}_m-\varphi_j}\\
    &\le  C_2e^{\frac{C_1}{m}}\int_U (\sum_{l=1}^{N}|\sigma_l|^2)^{\frac{1}{2m}} e^{-\varphi_j}\\
    &\le C_2e^{\frac{C_1}{m}}{\rm Vol} (U)^{1-\frac{1}{2m}} \left(\sum_{l=1}^{N} \int_U  |\sigma_l|^2e^{-2m\varphi_j}\right)^{\frac{1}{2m}}.
    \end{align*}
    Since $\int_D  |\sigma_l|^2e^{-2m\varphi}=1$ and  $\varphi_j\le \varphi $,
    using Guan-Li-Zhou's stability theorem, we can obtain that $|\sigma_l|^2e^{-2m\varphi_j}$ converges to $ |\sigma_l|^2e^{-2m\varphi}$ in $L^1(U)$ as
    $j\rightarrow +\infty$ for $1\le l\le N$ by choosing small $U$.
    Hence $\{e^{\varphi-\varphi_j}\}$ is bounded in $L^1(U)$.

    Since $e^{\varphi-\varphi_j}$ converges to $1$ locally in measure, $e^{\varphi-\varphi_j}$ converges to $1$ in $L^{\frac{1}{2}}(U)$ by using the result in \cite{Wade74}:
    \textit{If $\{f_j\}$ is bounded in $L^q(U)$
    for some $0<q<+\infty$ and $f_j$ converges to $f$ in measure, then  $f_j$ converges to $f$ in $L^r(U)$ for any $0<r<q$}.

    For general $0<p<+\infty$, we apply the above special result to $(2p+2)\varphi$ and $(2p+2)\varphi_j$  and obtain that $e^{\varphi-\varphi_j}$ is bounded in $L^{p+1}(U)$.
    Hence $e^{\varphi-\varphi_j}$ converges to $\rm {1}$ in $L^p(U)$.

\end{proof}

	In fact, one can weaken the condition $\vp_j\le\vp$ to be $\varphi_j$ uniformly bounded from above and  $\cal I(p_k\varphi_j)\subset\cal I(p_k\varphi)$ for some sequence $p_k$ going to $+\infty$.  	
	The following example show that Lemma \ref{lem e^(phi-phi_j) bounded}  is not valid without this assumption.
	Consider $\phi=\log|z|^2,\phi_j=\log|z-w/j|^2  $ in $\C^2$, then $\phi_j$ converges to $\phi$ but $\cal{I}(m\phi_j)_0\not\subset \cal{I}(m\phi)_0$.
	Then
	$$e^{\phi-\phi_j}=\frac{|z|^2}{|z-w/j|^2}$$ is not integrable near $0$, and clearly does not converge to $1$ in $L^1_\loc$.

\subsection{Griffiths semi-positive singular Hermitian metrics}\label{ssect GSP}
Secondly, let us recall and give some properties of Griffiths semi-positive singular Hermitian metrics.
\begin{defi}[\cite{BP08},\cite{PT18},\cite{Rau15}]\label{def GSP}
    Let $E$ be a holomorphic vector bundle over a complex manifold $X$. Let $h$ be a singular Hermitian metric on $E$.
    \begin{enumerate}
      \item $h$ is said to be Griffiths semi-negative if $~\log|u|^2_h$ is plurisubharmonic for any local holomorphic section $u$ of $E$.
      \item $h$ is said to be Griffiths semi-positive if $h^*$ is Griffiths semi-negative on $E^*$.
      \item $h$ is said to be Griffiths positive if locally there exists a smooth strictly plurisubharmonic function $\psi$ such that $he^{\psi}$ is  Griffiths semi-positive.
    \end{enumerate}
\end{defi}

\begin{lem}[\cite{PT18}]\label{prop lowerbound}
    Let $h$ be a singular Hermitian metric on $E$ over $X$ and take  $U\Subset X$ such that $E|_U$ is trivial.

    $(1)$ If $h^*$ is locally bounded from above then there exists a constant $C_U> 0$ such that  for any $x\in U$
\begin{equation*}
    0<\frac{1}{C_U} \le \lambda_1 (x) \le \lambda_2 (x) \le \cdots \le \lambda_r (x) \le C_U^{r-1}\det h(x) \le +\infty,
    \end{equation*}
    where $\lambda_i(x)$ are eigenvalues of $h(x)$.

    $(2)$ If $h$ is a Griffiths semi-positive singular Hermitian metric, then there exists a family of smooth Griffiths positive metrics $h_v$ defined on a relatively compact open subset $U_v\subset U$ such that $h_v$
    increasingly converges to $h$ on any relatively compact subset of $U$, and $-\log \det h$ is plurisubharmonic.
\end{lem}
\begin{proof}
    $(1)$ Let $e$ be a local holomorphic frame of $E$. Then $0\le h^*_{jj}=|e^*_j|^2_{h^*}$ is locally bounded from above.
    Due to the polarization identity
$$  4h^*_{jk}=\sum_{\alpha=1}^{4} \im^\alpha |e^*_j + \im^\alpha e^*_k|^2_{h^*},    $$

	we know that $|h^*_{ik}|=|h^*_{kj}|$ is locally bounded too.
    Especially, $h^*$ is finite at each point on $X$. Let $0\le\lambda^*_1\le\lambda^*_2\le\cdots\le\lambda^*_r$ be the eigenvalues of $h^*$.
    Then $\lambda^*_r$ is locally bounded and $\lambda^*_r\le C_U$ on $U$ for some $C_U>0$. Hence on $U$
\begin{align*}
    0<\frac{1}{C_U} \le \lambda_1=\frac{1}{\lambda^*_r} \le \lambda_2=\frac{1}{\lambda^*_{r-1}} \le \cdots \le \lambda_r=\frac{1}{\lambda^*_1} \le C_U^{r-1}\det h \le +\infty .
\end{align*}
    $(2)$ We may assume that $U\subset \mathbb{C}^n$. Let $\rho\ge0$ be a smooth functions such that $\Supp\rho\subset \{|z|<1\} $ and $\int \rho =1$ and $\rho_\varepsilon(z)={\varepsilon^{-2n}}\rho(\varepsilon^{-1}z)$.
    For a fixed frame $e$ of $E$, we define metrics $h^*_\varepsilon$ as $\rho_\varepsilon *h^*$ on $E^*$ over $U_\varepsilon=\{z\in U | B(z,\varepsilon)\subset U \}$, that is, for any fixed  $z\in U_\varepsilon$  and any $\xi=\sum_j \xi_je^*_j(z)
    \in E^*_z$, $|\xi|^2_{h^*_\varepsilon}$ is defined to be
    $$(|\widetilde{\xi}|^2_{h^*} * \rho_\varepsilon) (z)=\int |\widetilde{\xi}(x)|^2_{h^*(x)} \rho_\varepsilon(z-x),$$
    where $\widetilde{\xi}=\sum_j\xi_je^*_j$ is a constant holomorphic section under the frame $e^*$.

    In fact, if $h=_e (h_{jk})$,  then
    $$|\xi|^2_{h^*_\varepsilon}
    =\int \sum_{jk} \xi_j\bar{\xi}_k h^*_{jk}(x)\rho_\varepsilon(z-x).$$
    Obviously $h^*_\varepsilon$ is decreasingly convergent to $h^*$ as $\varepsilon$ tends to $0^+$ since $|\widetilde{\xi}|^2_{h^*}$ is plurisubharmonic on $U_\varepsilon$.
    For any holomorphic section
    $u=\sum_ju_je^*_j$ on any open subset $V\subset U_\varepsilon$,
    $$|u|^2_{h^*_\varepsilon}(z)=\int \sum_{jk}u_j(z)\overline{u_k(z)}{h^*_{jk}(z-x)} \rho_\varepsilon(x) =\int |u^x(z-x)|^2_{h^*(z-x)} \rho_\varepsilon(x),$$
    where $u^x(z)=\sum_ju_j(z+x)e^*_j(z)$ is a  holomorphic section on $V-x:=\{z-x|z\in V\}$ for fixed $x$ with $|x|<\varepsilon$.

    Due to the plurisubharmonicity of $|u^x|^2_{h^*}$, it follows that
\begin{align*}
    \frac{1}{2\pi} \int_0^{2\pi} |u|^2_{h^*_\varepsilon}(z+e^{\mathrm{i}\theta}\zeta) d\theta
    &=\int \left(\frac{1}{2\pi} \int_0^{2\pi} |u^x|^2_{h^*}(z-x+e^{\mathrm{i}\theta}\zeta) d\theta\right) \rho_\varepsilon(x)\\
    &\ge \int |u^x|^2_{h^*}(z-x)\rho_\varepsilon(x)
    =|u|^2_{h^*_\varepsilon}(z).
    \end{align*}
    Therefore $|u|^2_{h^*_\varepsilon}$ is a smooth plurisubharmonic function.
    So $(h^*_\varepsilon+\varepsilon\mathrm{Id}_{E^*})e^{\varepsilon|z|^2}$ is smooth Griffiths negative and decreasingly converges to $h^*$ as $\varepsilon$ tends to $0$.
    Then $\hat h_\varepsilon:=(h^*_\varepsilon+\varepsilon\mathrm{Id}_{E^*})^*e^{-\varepsilon|z|^2}$ is smooth Griffiths positive and increasingly  converges to $h$.
    Moreover, $-\log \det \hat h_\varepsilon$ is plurisubharmonic and decreasingly converges to $-\log \det h$ as $\varepsilon\rightarrow 0^+$. We conclude that  $-\log \det h$ is plurisubharmonic.
\end{proof}
\begin{rem}\label{rem DNWZ app}
  Recently, Deng-Ning-Wang-Zhou \cite[Theorem 1.1]{DNWZ23} shows that Lemma \ref{prop lowerbound} is also valid for any Stein open subset $U$ of $X$.
\end{rem}

\begin{lem}\label{pro h/deth}
     Let $h$ be a Griffiths semi-positive singular Hermitian metric on $E$, then $(E\otimes\det E^*,h\otimes\det h^*)$ is Griffiths semi-negative.
\end{lem}

\begin{proof}
    By Lemma~\ref{prop lowerbound} locally there exists smooth Griffiths positive metric $h_v$ defined on $U_v$ increasingly converging to $h$ on $E$ since $h$ is Griffiths semi-positive.
    Then $h^*_v\otimes \det h_v$ is smooth Griffiths positive on $E^*\otimes \det E$ over $U_v$ ~(see Theorem\rm{ \Rmnum{7}}-(9.2)\cite{Demaillybook2012}).
    Hence $h_v\otimes \det h^*_v$ is smooth Griffiths negative on $E\cong E\otimes \det E^*$ over $U_v$ and decreasingly converges to $h\otimes \det h^*$.
    So for any local holomorphic section $F$ of $E\otimes\det E^*$, $|F|^2_{{h_v}\otimes{\det h_v^*}}$ is plurisubharmonic and decreasingly converges to  $|F|^2_{{h}\otimes{\det h^*}}$.
  Then $|F|^2_{{h}\otimes{\det h^*}}$ is also plurisubharmonic.
\end{proof}

\begin{lem}\label{lem h_j>h}
    Let $h$ and $h'$ be Griffiths semi-positive singular Hermitian metrics on a trivial holomorphic vector bundle E over U. If $h\ge h'$, then $\frac{h}{\det h}\le \frac{h'}{\det h'}$.
\end{lem}

\begin{proof} It is obviously true for smooth metrics $h$ and $h'$. In general,
    from the proof of Lemma~\ref{prop lowerbound}, we can construct $h_v\nearrow h$ and $h_v'\nearrow h'$ on $U_v$ such that $h_v\ge h_v'$.
    Then  $\frac{h_v}{\det h_v}\le \frac{h_v'}{\det h_v'}$.
    Let $v\rightarrow +\infty$, we obtain that $\frac{h}{\det h}\le \frac{h'}{\det h'}$.
\end{proof}

\begin{lem}\label{lem int of |f|^2e^phi dominates |f|}
    Let $D$ be an open set.
\begin{enumerate}
        \item[\rm {(1)}]
        Assume $f_j\in \mathcal{O}(D), \phi \in \Psh(D)$ and $\phi \not\equiv -\infty$. If for any relatively compact subset $U$ of $D$ there exists $C_U$ independent of j such that
        $\int_{U} |f_j|^2e^\phi \le C_U < +\infty$  for every $j$, then there exists a subsequence of \{$f_j$\} compactly converging to f $\in \mathcal{O}(D)$.
        \item [\rm {(2)}]
        Let $h$ be a Griffiths semi-positive singular Hermitian metric on a trivial holomorphic vector bundle $E$ over $D$ and $F_j\in \mathcal{O}(E)(D)$. If for any relatively compact subset $U$ of $D$ there exists
        $C_U$ independent of $j$ such that $\int_{U} |F_j|^2_{\frac{h}{\det h}} \le C_U < +\infty$  $ \forall j$, then there exists a subsequence of \{$F_j$\} compactly converging to  $F_0\in \mathcal{O}(E)(D)$.
        \item[\rm {(3)}]
        Let $D,D_j$ be open sets and $E$ be a trivial holomorphic vector bundle with a  Griffiths semi-positive singular Hermitian metric $h$ over $D$. Assume $D_j\subset D_{j+1}$ and $\bigcup_j D_j=D$.
        If $F_j\in \mathcal{O}(E)(D_j) $ satisfies that $\int_{D_j} |F_j|^2_{\frac{h}{\det h}}\le C<+\infty$ for each $j$, then there exists a subsequence of \{$F_j$\} compactly converging to
        $F_0\in \mathcal{O}(E)(D)$ and $\int_{D} |F_0|^2_{\frac{h}{\det h}}\le C$.
    \end{enumerate}
\end{lem}

\begin{proof}

\begin{enumerate}[(1)]
        \setlength{\itemsep}{0pt}
        \item For any compact $K\subset D$, choose a relatively compact open subset $U_K$ of $D$ satisfying $K\subset U_K\subset D$, for any $x \in K$, there exists $p=p_x\ > 1$ such that
        $$\nu((p-1)\phi,x)=(p-1)\nu(\phi,x)< 2.$$
        Then Lemma~\ref{lem Skoda} implies that there is a neighborhood $\ U_x\subset U_K$ of $x$  such that $\int_{U_x} e^{-(p-1)\phi}< +\infty$. So
        by H\"older's inequality, we have
\begin{align*}
        \int_{U_x} |f_j|^{\frac{2}{q}}
        &=\int_{U_x} |f_j|^{\frac{2}{q}}e^{\frac{1}{q}\phi}e^{-\frac{1}{q}\phi}\\
        &\le \left(\int_{U_x}|f_j|^2e^{\phi}\right)^{\frac{1}{q}}\left(\int_{U_x}e^{-(p-1)\phi}\right)^{\frac{1}{p}}\\
        &\le \left(C_{{U_K}}\right)^{\frac{1}{q}}\left(\int_{U_x}e^{-(p-1)\phi}\right)^{\frac{1}{p}},
        \end{align*}
        where $\frac{1}{p}+\frac{1}{q} = 1$. For any relatively compact open subset $x \in \widetilde{U}_x \subset U_K$, we have $\max_{\widetilde{U}_x} |f_j|\le {C_x} $ by mean value inequality, where
        $C_x$ is independent of $j$ . Hence  $\max_K|f_j|\le \widetilde{C}_K $ for all $j$ by the finite covering theorem. So there exists a subsequence of \{$f_j$\} compactly converging to  $f\in \mathcal{O}(D)$
        by Montel's theorem.
        \item
        Note that $h$ has a lower bound locally and $-\log \det h$ is plurisubharmonic by Lemma~\ref{prop lowerbound}, then $(1)$ implies $(2)$.
        \item
        Since $\{D_j\}_j$ is an exhaustion of $D$, i.e., $D_j$ contains any given compact subset $K$ of $D$ for large $j$, using $(2)$ and the diagonal argument, there is a subsequence of $\{F_j\}$
        compactly converging to $F_0\in\mathcal{O}(E)(D)$ on $D$.
    \end{enumerate}
\end{proof}

In the end, we give a remark on the definition of complex singularity exponents (Definition \ref{def C_F}).

\begin{rem}\label{rem C monotone}

\begin{enumerate}
\item [\rm (1)]
	It is obvious that $ C_{F,\beta}(U')\le \int_{U'} |{F}|^2_{\frac{h}{\det h}} $ and ${C}_{F,\beta}(U'')\le {C}_{F,\beta}(U')$ for any $o \in U''\subset U'$.
\item [\rm (2)]
    It  follows from  {\rm{Lemma \ref{prop lowerbound}}} that  $|F|^2_h$ is locally dominated by a constant multiple of $|F|^2e^{-\varphi}$, where $\varphi=-\log\det h\in \Psh(U)$. Then by {\rm Lemma \ref{lem Skoda}} one can  show that  $c^F_o(h)=+\infty$ if  and only if $\nu(\varphi,o)=0$ for $F\not\equiv 0$.
\item [\rm (3)]
    By $\mathrm{Lemma}~\ref{lem int of |f|^2e^phi dominates |f|}$ and the closeness of the sections of coherent analytic sheaves under the  topology of compact convergence, one can prove that the set
    $$\mathcal{H}_\beta(U,o):=
    \left\{\widetilde{F}\in \mathcal{O}(E)(U) ~\Big| ~\int_U |\widetilde{F}|^2_{\frac{h}{\det h}} < +\infty\ \ \mathrm{and}\ \ (\widetilde{F}-F, o)\in \mathcal{E} (h(\det h)^\beta )_o    \right\}$$
     is a nonempty closed convex set  if $\int_U |F|^2_{\frac{h}{\det h}}<+\infty$.
\end{enumerate}
\end{rem}

\subsection{Lempert's question and Nakano semi-positive singular Hermitian metrics}\label{ssect NSP}

 Thirdly, in \cite{Lempert17}, Lempert asked on holomorphic Hilbert bundles whether a $C^2$-smooth Hermitian metric whose curvature dominates 0 is Nakano semi-positive in the usual sense.
    It follows essentially from Deng-Ning-Wang-Zhou's Theorem~\cite{DNWZ20} that one can answer Lempert's question affirmatively. In fact, by definition, there exists a sequence $\{h_j\}$ of $C^2$-smooth
    Nakano semi-positive Hermitian metrics which increasingly converges to $h$. It is well-known that $h_j$ is $L^2$ optimal, and so is $h$ by  the following {\rm Proposition~\ref{prop he^}}.
    Then $h$ is Nakano semi-positive in the usual sense since $h$ is $C^2$-smooth. In other words, {\rm Proposition~\ref{prop he^}} can be regard as an affirmative answer to Lempert's question in the setting of singular Hermitian metrics.

\begin{pro}\label{prop he^}
    Let $h$ be a singular Hermitian metric on a holomorphic vector  (resp. Hilbert) bundle $E$.
    Let $\{h_j\}$ be a sequence of  Nakano semi-positive singular Hermitian metrics on $E$.
    Assume that $\{h_j\}$ increasingly converges to $h$.
    Then $h$ is also Nakano semi-positive.
    Especially, $he^{-\psi}$ is Nakano semi-positive for any plurisubharmonic function $\psi$.
\end{pro}

\begin{proof}
Since the limit of a decreasing sequence of plurisubharmonic functions is also plurisubharmonic,
	one can check that $h$ is a Griffiths semi-positive singular Hermitian metric by definition.
	Now for any Stein open subset $U$ such that $E|_U$ is trivial, any K\"ahler form $\omega_U$ on $U$ and any $\psi \in {\rm Spsh}(U)\cap C^\infty(U)$ such that for any $\overline{\partial}$-closed
    $f\in C^\infty_{(n,1)} (U,E|_U)$ and
\begin{align*}
    \int_{U}^{} \langle B^{-1}_{\omega_U,\psi} f,f \rangle_{\omega_U,h} e^{-\psi} dV_{\omega_U} < +\infty,
    \end{align*}
    we have
\begin{align*}\label{ineq 4}
    \notag \int_{U}^{} \langle B^{-1}_{\omega_U,\psi} f,f \rangle_{\omega_U,h_j} e^{-\psi} dV_{\omega_U}
    \le
    \int_{U}^{} \langle B^{-1}_{\omega_U,\psi} f,f \rangle_{\omega_U,h} e^{-\psi} dV_{\omega_U}
    < +\infty.
    \end{align*}
    Then there exists $u_j\in  L^2_{(n,0)} (U,\omega_U,E|_{U},h_je^{-\psi})$ such that
     $\overline{\partial}u_j = f$ and
\begin{equation*}
    \int_{U}^{} |u_j|^2_{\omega_U,h_j} e^{-\psi} dV_{\omega_U}
    \le \int_{U} \langle B^{-1}_{\omega_U,\psi} f,f \rangle_{\omega_U,h_j} e^{-\psi} dV_{\omega_U}
    \le \int_{U}^{} \langle B^{-1}_{\omega_U,\psi} f,f \rangle_{\omega_U,h} e^{-\psi} dV_{\omega_U}.
    \end{equation*}

    Since $\{h_j\}$ is bounded below by a continuous Hermitian metric, we get that $\{u_j\}$ is bounded in $L^2_\loc(U)$, and especially $\{u_j-u_1\}$ is bounded in $L^2_\loc(U)$.
    Then there exists a subsequence of $\{u_j-u_1\}$ compactly converging to $u-u_1$ on $U$ since $\dpa (u_j-u_1)=0$ .
    So we obtain $\dpa u=f$ and by Fatou's lemma
\begin{equation*}
    \int_{U}^{} |u|^2_{\omega_U,h} e^{-\psi} dV_{\omega_U}
    \le \liminf_{j\to +\infty}\int_{U}^{} |u_j|^2_{\omega_U,h_j} e^{-\psi} dV_{\omega_U}
    \le \int_{U}^{} \langle B^{-1}_{\omega_U,\psi} f,f \rangle_{\omega_U,h} e^{-\psi} dV_{\omega_U}.
    \end{equation*}
\end{proof}

\begin{rem}
In \cite{In20}, Inayama proved that a singular Hermitian metric $h$ which is a limit of a increasing sequence of Nakano semi-positive  smooth Hermitian metrics $\{h_j\}$ is  Nakano semi-positive.
\end{rem}

Now let us recall the following facts:
\begin{enumerate}
  \item Any analytic set in a Stein manifold is globally defined;
  \item Stein manifold is still Stein after removing a hypersurface;
  \item For any holomorphic vector bundle on a Stein manifold, there exists an analytic subset such that the vector bundle is trivial outside the analytic subset.
\end{enumerate}

 Then by \cite[Lemma~{\rm{\Rmnum{8}}-(7.3)}]{Demaillybook2012}, it is easy to show that
\begin{lem}\label{thm NSP outside analytic set}
    Let $X$ be a complex manifold, $A\subset X$ be a proper analytic set. Let $E$ be a holomorphic vector bundle over $X$ with a  singular Hermitian metric $h$. Assume $h$ is Nakano semi-positive
    on $X\setminus A$ and Griffiths semi-positive on $X$. Then $h$ is  Nakano semi-positive on $X$.
\end{lem}

 In \cite{In20}, Inayama proved the Lemma \ref{thm NSP outside analytic set}  with an additional assumption that $X\setminus A$ is Stein.

In addition,  by \cite[Lemma~{\rm{\Rmnum{8}}-(7.3)}]{Demaillybook2012}, the condition that $E|_U$ is trivial in the definition of Nakano semi-positivity is unnecessary.  Then we can easily obtain the following result by using Remark\ref{rem DNWZ app}, \cite[Theorem \rm{\Rmnum{7}}-(9.2)]{Demaillybook2012}, \cite[Theorem 7.2 ]{LSY13} and Proposition \ref{prop he^}.

\begin{pro}\label{pro hdeth}
  Let $h$ be a Griffiths (semi-)positive singular Hermitian metrics on holomorphic vector bundle $E$ of rank $r$ over a complex manifold $X$ of dimension $n$. Then
  \begin{enumerate}
    \item[(1)] for any $s\ge \min\{n,r\}$, $(E^*\otimes\det E^s,h^*\otimes \det h^s)$ is   Nakano (semi-)positive.
    \item[(2)]  for any $m\ge0$, $({\rm Sym}^mE\otimes \det E,{\rm Sym}^mh\otimes\det h)$ is Nakano (semi-)positive.
  \end{enumerate}
 \end{pro}

 When $m=1$ and $h$ is smooth,  Nakano (semi-)positivity of $(E\otimes \det E,h\otimes\det h)$ is obtained by Demailly-Skoda \cite{Demailly1980}.  In \cite{In20}, Inayama  showed that  $(E\otimes \det E,h\otimes\det h)$ is Nakano (semi-)positive if $(E,h)$ is Griffiths (semi-)positive in the singular setting.

Let $\{\psi_j\}\subset \Psh(D)$ such that  $\sup_j \psi_{j}<+\infty$.
It is well-known that the upper semi-continuous regularization $(\sup_j \psi_{j})^\divideontimes$  is plurisubharmonic and equals to $\sup_j |u|^2_{h^*_j}$ almost everywhere. Similarly, we obtain the closeness of Nakano semi-positivity from above (Proposition \ref{prop closeness of NSP from above}).

\begin{pro}\label{prop closeness of NSP from above}
    Let $h$ be a singular
    Hermitian metric.
    Let $\{h_j\}$ be a sequence of  Griffiths (resp. Nakano) semi-positive singular Hermitian metrics.
    Assume that $h$ is bounded below by a continuous Hermitian metric and $h\le h_j\le h_1$ converges to $h$.
    Then $h$ has a unique Griffiths (resp. Nakano) semi-positive modification, i.e. $h$ equals to a Griffiths (resp. Nakano) semi-positive singular Hermitian metric almost everywhere.
\end{pro}

\begin{proof}
	
	Firstly, we show that $h$ has a Griffiths semi-positive modification.
	The key point is to show that such modification metric is globally well-defined.
	
    Let $e$ be a local holomorphic frame of $E$.
    Via the polarization identity,
    we have that for any $1\le j,k\le r$
    $$
	    4h^*_{jk}=\sum_{\alpha=1}^{4} \im^\alpha |e^*_j+\im^\alpha e^*_k|^2_{h^*}
    .$$
    Since $h_j\ge h$ converges to $h$, we get that $|u|^2_{h^*}=\sup_j |u|^2_{h^*_j}$ for any local holomorphic section $u$ of $E^*$.
    Moreover, it is well-known that the upper semi-continuous regularization $(\sup_j |u|^2_{h^*_j})^\divideontimes$ of $\sup_j |u|^2_{h^*_j}$ is plurisubharmonic and equals to $\sup_j |u|^2_{h^*_j}$ almost everywhere.
    Especially, we locally define the new metric $\widehat{h}$ on $E$ by setting
    $$
    4{\widehat{h}}^*_{jk}
    :=\sum_{\alpha=1}^{4} \im^\alpha \left(\sup_l |e^*_j+\im^\alpha e^*_k|^2_{h^*_l}\right)^\divideontimes
    .$$
    Of course, $\widehat{h}=h$ almost everywhere.

    We claim that $\widehat{h}$ is Griffiths semi-positive.
    Take a smooth function $0\le \rho\le 1$ with compact support satisfying $\int_X \rho=1$ and set $\rho_\varepsilon =\varepsilon^{-2n}\rho(\varepsilon^{-1}z)$.
    For any local holomorphic section $u=\sum u_j e^*_j$ of $E^*$,
    $$
    |u|^2_{\widehat{h}^*}*\rho_\varepsilon
    =|u|^2_{{h}^*}*\rho_\varepsilon
    =(\sup_l |u|^2_{h^*_l})*\rho_\varepsilon
    =(\sup_l |u|^2_{h^*_l})^\divideontimes*\rho_\varepsilon
    $$
    is plurisubharmonic and increasing with respect to $\varepsilon>0$.

    On the other hand,
\begin{align*}
    4\widehat{h}^*_{jk}*\rho_\varepsilon
    &=\sum_{\alpha=1}^{4} \im^\alpha
    \left(\sup_l|e^*_j+ \im^\alpha e^*_k|^2_{h^*_l}\right)^\divideontimes*\rho_\varepsilon    \\
 &\rightarrow
    \sum_{\alpha=1}^{4} \im^\alpha
    \left(\sup_l|e^*_j+ \im^\alpha e^*_k|^2_{h^*_l}\right)^\divideontimes\\
 &=4\widehat{h}^*_{jk}   \quad\quad\quad\quad\quad\quad\quad\quad\quad\quad\quad\quad {\text {as}}~ \varepsilon\rightarrow0^+.
    \end{align*}
    Let $\varepsilon\rightarrow0^+$, noticing that $u$ is continuous and $\widehat{h}^*$ is locally finite we get that
    $$
    |u|^2_{\widehat{h}^*}*\rho_\varepsilon
    =\sum_{jk} (u_j\overline{u}_k\widehat{h}^*_{jk})*\rho_\varepsilon
    \longrightarrow \sum_{jk} u_j\overline{u}_k\widehat{h}^*_{jk}
    =|u|^2_{\widehat{h}^*}.
    $$
    Hence $|u|^2_{\widehat{h}^*}$ equals to $(\sup_j |u|^2_{h^*_j})^\divideontimes$ and is also plurisubharmonic.
	Notice that the representation $(\sup_j |u|^2_{h^*_j})^\divideontimes$ is independent of the choice of holomorphic frames and so is $|u|^2_{\widehat{h}^*}$.
	Hence $\widehat{h}$ is globally well-defined.

	And the uniqueness of  Griffiths semi-positive modification is owing to the strong upper semi-continuity of plurisubharmonic functions:

    \textit{Let $\phi$ be a plurisubharmonic function on a domain $D\subset\C^n$ and $B\subset D$ be a set of Lebesgue measure zero. Then for every $z\in D$, we have
    $$\limsup\limits_{D\setminus B \ni z'\to z}\phi(z')=\phi(z).$$  }

    Secondly we prove that $h$ is $L^2$ optimal if $h_j$ is Nakano semi-positive and hence the unique Griffiths semi-positive modification is Nakano semi-positive.

    We take any Stein open subset $U$ such that $E|_U$ is trivial, any K\"ahler form $\omega_U$ on $U$ and any $\psi \in {\rm Spsh}(U)\cap C^\infty(U)$ such that for any $\overline{\partial}$-closed
    $f\in C^\infty_{(n,1)} (U,E|_U)$ and
\begin{align*}
    \int_{U}^{} \langle B^{-1}_{\omega_U,\psi} f,f \rangle_{\omega_U,h} e^{-\psi} dV_{\omega_U} < +\infty.
    \end{align*}
    Let $\varphi_j=-\log\det h_j$.
    Noticing that $A=\Supp \cal{O}_U/\cal{I}(\varphi_1)$ is a proper analytic subset of Stein manifold $U$, we have $A=\{g_1=g_2=\cdots=g_l=0 \}$ for some holomorphic functions $g_j$ on $U$.
    Then $\widetilde{U}=U\setminus \{g_1=0\}$ is also Stein open and we can find Stein open sets $U_k\Subset U_{k+1}$ such that $\bigcup_k U_k=\widetilde{U}$.
    Since
    $\frac{h_j}{\det h_j}\le \frac{h}{\det h}$ and $\varphi_j\ge \varphi_1$,  we obtain
\begin{align*}
    \notag \int_{U_k}^{} \langle B^{-1}_{\omega_U,\psi} f,f \rangle_{\omega_U,h_j} e^{-\psi} dV_{\omega_U}
    &\le
    \int_{U_k}^{} \langle B^{-1}_{\omega_U,\psi} f,f \rangle_{\omega_U,\frac{h}{\det h}} e^{-\varphi_j-\psi} dV_{\omega_U} \\
    &\le \int_{U_k}^{} \langle B^{-1}_{\omega_U,\psi} f,f \rangle_{\omega_U,\frac{h}{\det h}} e^{-\varphi_1-\psi} dV_{\omega_U}
    < +\infty,
    \end{align*}
	where the finiteness of the last integral comes from that the integrability of $e^{-\varphi_1}$ on $U_k$ and the boundedness of
	$\frac{h}{\det h}$ and $f$ on $U_k$.

    Then by Nakano semi-positivity of $h_j$, there exists $u_{jk}\in  L^2_{(n,0)} (U_k,\omega_U,E|_{U},h_je^{-\psi})$ such that $\overline{\partial}u_{jk} = f$ and
\begin{equation*}
    \int_{U_k}^{} |u_{jk}|^2_{\omega_U,h_j} e^{-\psi} dV_{\omega_U}
    \le \int_{U_k} \langle B^{-1}_{\omega_U,\psi} f,f \rangle_{\omega_U,h_j} e^{-\psi} dV_{\omega_U}
    \le \int_{U_k}^{} \langle B^{-1}_{\omega_U,\psi} f,f \rangle_{\omega_U,h} e^{\varphi-\varphi_j-\psi} dV_{\omega_U}.
    \end{equation*}

    Then
\begin{equation*}
    \int_{U_k}^{} |u_{jk}|^2_{\omega_U,h} e^{-\psi} dV_{\omega_U}
    \le \int_{U_k}^{} \langle B^{-1}_{\omega_U,\psi} f,f \rangle_{\omega_U,h} e^{\varphi-\varphi_j-\psi} dV_{\omega_U}
    \le \int_{U_k}^{} \langle B^{-1}_{\omega_U,\psi} f,f \rangle_{\omega_U,\frac{h}{\det h}} e^{-\varphi_1-\psi} dV_{\omega_U}.
    \end{equation*}

    Since $h$ is bounded below by a continuous Hermitian metric, we get that $\{u_{jk}\}_j$ is bounded in $L^2_\loc(U_k)$, and especially $\{u_{jk}-u_{1k}\}$ is bounded in $L^2_\loc(U_k)$.
    Then there exists a subsequence of $\{u_{jk}-u_{1k}\}_j$ compactly converging to $u_k-u_{1k}$ on $U$ since $\dpa (u_{jk}-u_{1k})=0$.
    So we obtain $\dpa u_k=f$ on $U_k$ and it follows from Fatou's lemma and the Lebesgue dominated convergence theorem that
\begin{equation*}
    \int_{U_k}^{} |u_k|^2_{\omega_U,h} e^{-\psi} dV_{\omega_U}
    \le \int_{U_k}^{} \langle B^{-1}_{\omega_U,\psi} f,f \rangle_{\omega_U,h} e^{-\psi} dV_{\omega_U}
    \le \int_{U}^{} \langle B^{-1}_{\omega_U,\psi} f,f \rangle_{\omega_U,h} e^{-\psi} dV_{\omega_U}.
    \end{equation*}

    Repeating the above argument, we have there exists $u\in  L^2_{(n,0)} (\widetilde{U},\omega_U,E|_{\widetilde{U}},he^{-\psi})$ such that $\overline{\partial}u= f$ on $\widetilde{U}$
    and
\begin{equation*}
    \int_{\widetilde{U}}^{} |u|^2_{\omega_U,h} e^{-\psi} dV_{\omega_U}
    \le \int_{U}^{} \langle B^{-1}_{\omega_U,\psi} f,f \rangle_{\omega_U,h} e^{-\psi} dV_{\omega_U}.
    \end{equation*}
    Due to Lemma~{\rm{\Rmnum{8}}-(7.3)} in \cite{Demaillybook2012}, $u$ can extend across the analytic set $\{g_1=0\}$ and hence $\dpa u=f$ on $U$ and
\begin{equation*}
    \int_{U}^{} |u|^2_{\omega_U,h} e^{-\psi} dV_{\omega_U}
    \le \int_{U}^{} \langle B^{-1}_{\omega_U,\psi} f,f \rangle_{\omega_U,h} e^{-\psi} dV_{\omega_U}.
    \end{equation*}

\end{proof}

    Similarly, we can prove the following lemma.
\begin{lem}\label{modifition}
    Let $h$ be  Nakano semi-positive and $\psi\in\Psh(U)$. If there exists $\{\psi_j \}_j\subset \Spsh(U)\cap C^\infty(U)$ such that $\psi_j$ converges to $\psi$ locally in measure on $U$, $\psi_j$ is
    locally uniformly bounded from above on $U$ and
\begin{align*}
    \liminf_{j\rightarrow +\infty}\int_{U}^{} \langle B^{-1}_{\omega_U,\psi_j} f,f \rangle_{\omega_U,h} e^{-\psi_j} dV_{\omega_U}= C<+\infty,
    \end{align*}
    then there exists u $\in  L^2_{(n,0)} (U,\omega_U,E|_U,he^{-\psi})$ such that $\overline{\partial}u = f$ and
\begin{align*}
    \int_{U}^{} |u|^2_{\omega_U,h} e^{-\psi} dV_{\omega_U}
    \le C.
    \end{align*}
\end{lem}
\subsection{A generalized Siu's lemma}
Lastly, in order to obtain an $L^2$ extension theorem, we need a generalized Siu's lemma proved by Zhou-Zhu \cite{Siu Lemma}.
\begin{lem}\label{lem siu lemma}
	Let $\phi(z',z'')$ be a plurisubharmonic function on $\mathbb{B}^m_r\times \mathbb{B}^{n-m}_r$
such that
	\begin{align*}
	I_\phi:=\int_{z''\in \mathbb{B}^{n-m}_r} e^{-\phi(0,z'')}d\lambda_{n-m}<+\infty~~~(I_\phi:=e^{-\phi(0)}~~if~ m=n).
	\end{align*}
	Let $h$ be a nonnegative continuous function on $\mathbb{B}^m_r\times \mathbb{B}^{n-m}_r$.
	Assume that $\varepsilon,r'\in(0,r)$.
	Then
	\begin{align*}
	\lim_{\varepsilon\rightarrow0^+}\frac{1}{\lambda(\mathbb{B}_\varepsilon^m)}
	\int_{\mathbb{B}^m_\varepsilon\times \mathbb{B}^{n-m}_{r'}}
	h(z',z'')e^{-\phi(z',z'')}d\lambda_n
	=\int_{z''\in\mathbb{B}^{n-m}_{r'}}
	h(0,z'')e^{-\phi(0,z'')}d\lambda_{n-m},
	\end{align*}
	where $\lambda(\mathbb{B}^k_r):=$ the $2k$-dimensional Lebesgue measure of $\mathbb{B}_r^k$ $($if $m=n$, $\mathbb{B}^{n-m}_{r'}$ and $z''$ will disappear and $\int_{z''\in\mathbb{B}^{n-m}_{r'}}
	h(0,z'')e^{-\phi(0,z'')}d\lambda_{n-m}$ will be replaced by $h(0)e^{-\phi(0)}~)$.
\end{lem}

\begin{rem}\label{rem 1}
	From the proof of the above lemma in \cite{Siu Lemma}, in fact, it can be strengthened to the following result:\\
	Let $D\subset \C^{n-m}$ be a pseudoconvex open set and  $\phi(z',z'')$ be a plurisubharmonic function on $\mathbb{B}^m_r\times D$ such that
	\begin{align*}
	I_\phi:=\int_{z''\in D} e^{-\phi(0,z'')}d\lambda_{n-m}<+\infty~~~(I_\phi:=e^{-\phi(0)}~~if~ m=n).
	\end{align*}
	Let $h$ be a nonnegative continuous function on $\mathbb{B}^m_r\times D$.
	Assume that $\varepsilon\in(0,r)$.
	Then for any $D'\Subset D$
	\begin{align*}
	\lim_{\varepsilon\rightarrow0^+}\frac{1}{\lambda(\mathbb{B}_\varepsilon^m)}
	\int_{\mathbb{B}^m_\varepsilon\times D'}
	h(z',z'')e^{-\phi(z',z'')}d\lambda_n
	\le\int_{z''\in D'}
	h(0,z'')e^{-\phi(0,z'')}d\lambda_{n-m},
	\end{align*}
	where $\lambda(\mathbb{B}^m_\varepsilon):=$ the $2m$-dimensional Lebesgue measure of $\mathbb{B}_\varepsilon^m($ if  $m=n$, $D'$ and $z''$ will disappear and $\int_{z''\in D'}
	h(0,z'')e^{-\phi(0,z'')}d\lambda_{n-m}$ will be replaced by $h(0)e^{-\phi(0)}~)$.
\end{rem}

Let $(X,\omega)$ be an $n$-dimensional Stein manifold possessing a K\"ahler metric $\omega$, and  $s=(s_1,\cdots,s_k)$, where $s_i(1\leq i\leq k)$ are global holomorphic functions on $X$.
Assume that $s$ is transverse to the zero section $$S:=\{z\in X : s(z)=0\}$$ such that
$\dim S=n-k$.
Let $\phi$ be a plurisubharmonic function on $X$ and  $e^{-\phi} |_S\in L^1_{\loc}(S)$.
Denote $r=|s|^2=\sum |s_i|^2$.
For such an $r$, any continuous volume form $d\mu$ on $X$ induces a volume form $d\mu_r$ on $Y$
as follows.
Suppose $\theta\in C(X)$, then $d\mu_r$ will satisfy
\begin{equation}
\int_S \theta d\mu_r=\lim_{\varepsilon\to 0}\varepsilon^{-2k}
\int_{\{z\in X\colon r(z) < \varepsilon^2\}} \theta d\mu.
\end{equation}
Locally $d\mu_r$ can be computed if we introduce local coordinates $z_1,\cdots,z_n$ on $X$ so that $s_i= z_i$.
If $d\mu=\im^n\alpha dz_1\wedge d\bar z_1\wedge\cdots \wedge dz_n\wedge d\bar z_n$ and $v_k$ is the volume of the unit ball in $\C^k$,
then $d\mu_r= \im^{n-k}v_k\alpha dz_{k+1}\wedge d\bar z_{k+1}\wedge\cdots \wedge dz_n\wedge d\bar z_n$.
Direct computation gives that $d\mu_r$ induced by $dV_{X,\omega}$ is $dV_{S,\omega}/|\Lambda^k ds|^2$.
In particular, if $S$ is a point $o$, then $\int_{\{o\}}\theta d\mu_r=\theta(o)/|\Lambda^n ds(o)|^{2}$.

Similar to Lemma \ref{lem siu lemma}, we have the following lemma:
\begin{lem}\label{lem 2}
Let $X, S, \phi$ be as above,
 $X_1\Subset X_2\Subset X$  be Stein open subsets of $X$, and $\theta$ be a nonnegative upper
 semi-continuous function on $X$, then we have
\begin{align*}
\limsup_{\varepsilon\rightarrow0^+}\frac{1}{\varepsilon^{2k}}	
\int_{X_1\cap\{r<\varepsilon^2\}}\theta e^{-\phi}dV_{X,\omega}
\le \int_{X_2\cap S}\theta e^{-\phi}d\mu_r .
\end{align*}
\end{lem}
 \begin{proof} Firstly we assume that  $\theta$ is continuous.
Since $s$ is transverse to the zero section, we  can choose $\varepsilon_0 >0$ such that $\Lambda ^k ds\neq 0$ on $X_2\cap\{r<\varepsilon_0^2\}$.
Then  there is a finite Stein open covering $\{\Omega_l\}_{l=1}^p$ of $X_1$ such that $X_1\cap\{r<\varepsilon_0^2\}\Subset\bigcup_{l=1}^p\Omega_l \Subset X_2$, and $\Omega_l\bigcap\{r<\varepsilon_0^2\}$ is biholomorphic to $(\Omega_l\cap S)\times\mathbb{B}^k_{\varepsilon_0}$ (see \cite{ZhouZhuJDG}).
By partition of unit, there exist $\{\rho_l\}_{l=1}^p\subset C^\infty(X)$
satisfying $\rho_l|_{(\Omega_l)^c}=0,\ 1\leq l\leq p$ and $\sum_{l=1}^p\rho_l=1$ on $X_1\bigcap\{r<\varepsilon_0^2\}$.
Since $e^{-\phi}\in L^1_{\loc}(S)$, by Remark~\ref{rem 1} we have
\begin{align*}
\limsup_{\varepsilon\rightarrow0^+}\frac{1}{\varepsilon^{2k}}	
\int_{X_1\cap\{r<\varepsilon^2\}}\theta e^{-\phi}dV_{X,\omega}
&=\limsup_{\varepsilon\rightarrow0^+}\sum_{l=1}^p\frac{1}{\varepsilon^{2k}}	\int_{X_1\cap\{r<\varepsilon^2\}}\rho_l \theta e^{-\phi}dV_{X,\omega}\\
&\leq\sum_{l=1}^p\limsup_{\varepsilon\rightarrow0^+}\frac{1}{\varepsilon^{2k}}	\int_{\Omega_l \cap\{r<\varepsilon^2\}}\rho_l \theta e^{-\phi}dV_{X,\omega}\\
&= \sum_{l=1}^p\int_{\Omega_l\cap S}\rho_l \theta e^{-\phi}d\mu_r\\
&\leq \int_{X_2\cap S}\theta e^{-\phi}d\mu_r.
\end{align*}

Now for an upper semi-continuous function $\theta$, there is a sequence $\{\theta_j\}$ of continuous functions decreasingly converging to $\theta$. Then for all $j$, we have
\begin{align*}
  \limsup_{\varepsilon\rightarrow0^+}\frac{1}{\varepsilon^{2k}}	
\int_{X_1\cap\{r<\varepsilon^2\}}\theta e^{-\phi}dV_{X,\omega} &\le \limsup_{\varepsilon\rightarrow0^+}\frac{1}{\varepsilon^{2k}}	
\int_{X_1\cap\{r<\varepsilon^2\}}\theta_j e^{-\phi}dV_{X,\omega}  \\
   &\le  \int_{X_2\cap S}\theta_j e^{-\phi}d\mu_r .
\end{align*}
Then by the  monotone convergence theorem we get that
\begin{align*}
  \limsup_{\varepsilon\rightarrow0^+}\frac{1}{\varepsilon^{2k}}	
\int_{X_1\cap\{r<\varepsilon^2\}}\theta e^{-\phi}dV_{X,\omega}\le  \int_{X_2\cap S}\theta e^{-\phi}d\mu_r.
\end{align*}
 \end{proof}
\section{$L^2$ Extension and Its Applications}
\subsection{An Ohsawa-Takegoshi type $L^2$ extension theorem}

\begin{thm}\label{main1}
Let $(X,\omega)$ be an $n$-dimensional Stein manifold possessing a K\"ahler metric $\omega$, and  $s=(s_1,\cdots,s_k)$, where $s_i(1\leq i\leq k)$ are global holomorphic functions on $X$.
Assume that $s$ is generically transverse to the zero section, and let
 $$S:=\{z\in X : s(z)=0\},\ \ \ \ {{\dim}}\ S_{\reg}=n-k.$$ Assume that $r:=|s|^2\leq e^{-2}$ on $X$.
Let $E\to X$ be a holomorphic vector bundle with a Nakano semi-positive  singular Hermitian metric $h$. Assume that $\varphi\not\equiv-\infty$ on every connected component of $S$, where $\varphi=- \log \det  h$ is the local weight.
	Let $f$ be a holomorphic section of $K_X\otimes E|_S$ over $S$ and
$$
\int_S|f|^2_hd\mu_r<+\infty,
$$
then there is a holomorphic section $F$ of $K_X\otimes E$ such that $F|_S=f$ on $S$ and
\begin{equation*}
\int_X  \frac{|F|^2_h}{ r^k(\log r)^2} dV_{X,\omega}\leq C \int_S|f|^2_h d\mu_r,
\end{equation*}
where $C$ is a uniform constant only depending on $k$.
	If $k=n$, that is, $S$ is a point $o$, the result states that if $\vp(o)>-\infty$, then for any $a\in (K_X\otimes E)_o$,
	there is a holomorphic section $F$ of $K_X\otimes E$ such that $F(o)=a$ and
	\begin{equation*}
	\int_X  \frac{|F|^2_h}{ r^n(\log r)^2} dV_{X,\omega}\leq C\frac{ |a|^2_{h(o)}}{ |\Lambda^n ds(o)|^{2}}.
	\end{equation*}

\end{thm}

\begin{proof}
\vskip0.3cm
\noindent {\bf Step One: Reduce to the case that $S$ is smooth, $E$ is trivial  and $\det h$ is locally integrable.}
\vskip0.5cm
	Since $s$ is generically transverse to the zero section $S$, $A_1:=\{x\in S | \Lambda^k ds= 0\}$ is an analytic subset of $X$.  And since the plurisubharmonic function $\varphi\not \equiv -\infty$ on every  $(n-k)$-dimensional irreducible component of $S$, there exists a proper analytic set $A_2\subset S$ such that $e^{-\varphi} |_{S_\reg\setminus A_2}\in L^1_{\loc}(S_\reg\setminus A_2)$. Since $X$ is Stein, $A_2$ is also an analytic set of $X$. In addition,
	since $X$ is Stein, there exists a proper analytic set $A_3\subset X$ such that $A_3$ does not contain any component of $S$ and $E|_{X\setminus A_3}$ is trivial.
	In fact, on each  $(n-k)$-dimensional irreducible component $S_j$ of $S$,
	one can take $z_j\in S_j$ and a local holomorphic frame $\{e_l^j\}$  near $z_j$,
	then by Cartan's extension theorem, there exists holomorphic sections $\{e_l\}$ such that $e_l(z_j)=e_l^j(z_j)$ for all $j$.
	Then $\{e_l\}$ is a holomorphic frame outside a proper analytic subset $A_3$ of $X$ and $z_j\notin A_3$  for all $j$. Hence $A_3\not\supset S_j$ for all $j$.

Since $A_0:=A_1\cup A_2\cup A_3$ is an analytic set of $X$ which does not contain any  $(n-k)$-dimensional irreducible component $S_j$ of $S$. Thus swe can take $z_j\in S_j\setminus A_0$ for each $j$, then by Cartan's extension theorem, there exists a holomorphic function $g$ such that $g|_{A_0}=0$ and $g(z_j)=1$ for all $j$. Take $A=\{z\in X| g(z)=0\}$, then  $\tilde X:=X\setminus A$ is also Stein and $\tilde S:=S\setminus A$ is smooth, $E|_{\tilde X}$ is trivial  and $\det h|_{\tilde S}$ is locally integrable.

If we have proven that there is a holomorphic section $F$ of $K_{X}\otimes E|_{\tilde X}$ such that $F|_{\tilde S}=f$ on $\tilde S$ and
\begin{equation*}
\int_{\tilde X}  \frac{|F|^2_h}{ r^k(\log r)^2} dV_{X,\omega}\leq C \int_{\tilde S}|f|^2_h d\mu_r,
\end{equation*}
where $C$ is a uniform constant only depending on $k$.

Then by \cite[Lemma~{$\rm{\Rmnum{8}}$-(7.3)}]{Demaillybook2012}, $F$ can be extended across $A$ with $F|_S=f$ and
	\begin{align*}
	\int_{X}
	\frac{|F|^2_h
	}{r^{k} (\log r)^2}dV_{X,\omega}
	\le C\int_{S}|f|^2_h d\mu_r.
	\end{align*}

\vskip0.3cm
\noindent {\bf Step Two: Obtain an $L^2$ extension of openness type.}
\vskip0.5cm

Now  $\tilde X$ is Stein and $\tilde S$ is smooth, $E|_{\tilde X}$ is trivial  and $e^{-\varphi}|_{\tilde S}$ is locally integrable.
Then we can choose a family $\{\tilde X_j\}$ of Stein open sets in $\tilde X$ such that $\tilde X_j\Subset \tilde X_{j+1}$ and $\bigcup_j \tilde X_j=\tilde X$.
	Since $X$ is Stein, by Cartan's extension theorem, there is a holomorphic section $\hat{f}$ of $K_X\otimes E$ over $X$ such that $\hat{f}|_S=f$.
	Due to the assumption that $e^{-\varphi}|_{ S}\in L^1_{\rm{loc}}( S)$,
	it follows from the $L^2$-extension theorem for holomorphic functions we know that for every point $z\in S$,
	there exists a neighborhood $U_z$ of $z$ in $X$, such that $\int_{U_z}e^{-\varphi}<+\infty$.
	By the relative compactness of $\tilde S_{j+1}:=\tilde{S}\cap{X}_{j+1}$, we can find a neighborhood $U_{j+1}$ of $\tilde S_{j+1}$, such that $\int_{U_{j+1}}e^{-\varphi}<+\infty$.
	It is easy to see that
	 there exists an $\varepsilon_j>0$ such that $\tilde X_j\cap \{r<\varepsilon_j\}\Subset U_{j+1}$. Therefore $\int_{ \tilde X_j\cap\{r<\varepsilon_j\}} e^{-\varphi}<+\infty$ and then $$\int_{\tilde X_j\cap \{r<\varepsilon_j\}} |\hat{f}|^2_hdV_{X,\omega}=\int_{\tilde X_j\cap\{r<\varepsilon_j\}} |\hat{f}|^2_{\frac{h}{\det h}}e^{-\varphi}dV_{X,\omega}<+\infty.$$

	Let $\phi(z)=\log r$, $\phi_\varepsilon(z)=\log (r+\varepsilon^2)$ and $\phi_{\delta,\varepsilon}(z)=\log (r+\delta^2+\varepsilon^2)=\log (e^{\phi_\delta}+\varepsilon^2)$ on $\tilde X$.	
	Take $\tau(t)=kt,\chi(t)=-\log[-t+\log (-t)]$ when $t<-1$, then
	$$\chi'=\frac{1-t^{-1}}{-t+\log(-t)}>0~~,~~ \chi''=\frac{(1-t^{-1})^2}{[-t+\log(-t)]^2}+\frac{t^{-2}}{-t+\log(-t)}>0.$$	
    Let $0\le\rho=\rho_{\varepsilon}\le1 $ be a smooth function on $\mathbb{R}$ such that $\rho=1$ on $(-\infty, a_\varepsilon)$, $\rho=0$ on $(b_\varepsilon,+\infty)$ and $|\rho'|\le \frac{1+c_\varepsilon}{b_\varepsilon-a_\varepsilon}$ on $(a_\varepsilon,b_\varepsilon)$, where $b_\varepsilon>a_\varepsilon>0$, $c_\varepsilon>0$ are to be determined later and $\lim_\varepsilon b_\varepsilon=\lim_\varepsilon c_\varepsilon=0$.
	Clearly $v=\hat{f} \overline{\partial}\rho(r)$ is a $\dpa$-closed (0,1)-form with value of $K_{X}\otimes E$ over $\tilde X$ and for $\varepsilon>0$ small enough, we have
	\begin{align*}
	\int_{\tilde X_j} |v|^2_h e^{-\tau(\phi)}dV_{X,\omega}
	= \int_{\{a_\varepsilon\le r \le b_\varepsilon\}\cap \tilde X_j} |\hat{f}|^2_h |\rho'|^2|\dpa r|^2 e^{-\tau(\log r)}dV_{X,\omega}<+\infty,
	\end{align*}
	where $\ve$ small enough such that $b_\ve< \ve_j$.
	By the assumption $h$ is Nakano semi-positive on $\tilde X$, it follows from Proposition \ref{prop he^} that $h'=he^{-\tau(\phi)}$ is also Nakano semi-positive.
	Thus there exists a solution $\hat{u}$ of $\overline{\partial}\hat u=v$ such that
	\begin{align*}
	 \int_{\tilde X_j} |\hat{u}|^2_{h'}e^{-\Psi}dV_{X,\omega}\le
	 \int_{\tilde X_j} \langle [\im \pa\dpa\Psi\otimes\mathrm{Id}_E,\Lambda_\omega]^{-1}v,v\rangle_{h'}e^{-\Psi}dV_{X,\omega}<+\infty,
	 \end{align*}
	where $\Psi$ is a smooth strictly plurisubharmonic exhaustion function on the Stein manifold $\tilde X$. Since $\tilde X_j$ is relatively compact in $\tilde X$, we know $\Psi$ is bounded on $\tilde X_j$, so $\int_{\tilde X_j} |\hat{u}|^2_{h'} < +\infty$. Therefore we can find a solution $u:=u_{\varepsilon,j}$ of $\overline{\partial}u=v$ with minimal norm in $L^2_{(n,0)}({\tilde X_j},E,h')$, i.e., $u \perp $ Ker $\overline{\partial}$ in $L^2_{(n,0)}({\tilde X_j},E,h')$.
	Since $\chi(\phi_{\varepsilon})$ is a smooth plurisubharmonic function on $\tilde X$ and hence bounded on $\tilde X_j$, we have $ue^{\chi(\phi_{\varepsilon})}\bot$ Ker$\overline{\partial}$ in $L^2_{(n,0)}({\tilde X_j}, E, h' e^{-\chi(\phi_{\varepsilon})})$.	
	
	Let  $K_{\delta}(z)=\tau(\phi_\delta)+\chi(\phi_{\delta,\varepsilon})+\delta \Psi$, we denote that
	$$B_\delta:=[\im \pa\dpa K_{\delta}(z)\otimes\mathrm{Id}_E,\Lambda_\omega]
	=[\im \pa\dpa (\tau(\phi_\delta)+\chi(\phi_{\delta,\varepsilon})+\delta \Psi)\otimes\mathrm{Id}_E,\Lambda_\omega].$$
	Since $\log r$ is plurisubharmonic, we have $\im  r\pa\dpa r \ge \im \pa r\wedge \dpa r$. Therefore,
		\begin{align*}
		 \im\pa\dpa K_\delta
		\ge &\delta\cdot \im\pa\dpa \Psi
	+ \left(  \frac{\chi''}{(r+\varepsilon^2+\delta^2)^2}+
	\frac{\tau'\delta^2}{r(r+\delta^2)^2}+\frac{(\delta^2+\epsilon^2)\chi'}{r(r+\delta^2+\varepsilon^2)^2}\right) \cdot \im\pa r\wedge \dpa r \\
	 \ge& 	\left(  \frac{\chi''}{(r+\varepsilon^2+\delta^2)^2}+
	 \frac{\tau'\delta^2}{r(r+\delta^2)^2}+\frac{(\delta^2+\epsilon^2)\chi'}{r(r+\delta^2+\varepsilon^2)^2}\right) \cdot \im\pa r\wedge \dpa r,
		\end{align*}
	where $\chi'=\chi'(\phi_{\delta,\ve}), \chi''=\chi''(\phi_{\delta,\ve})$.
	It follows that
	\begin{align*}
	 |\dpa r|^2_{\im\pa\dpa K_{\delta}}\le \frac{1}{\frac{\chi''}{(r+\varepsilon^2+\delta^2)^2}+ \frac{\tau'\delta^2}{r(r+\delta^2)^2}+\frac{(\delta^2+\epsilon^2)\chi'}{r(r+\delta^2+\varepsilon^2)^2}}.
	\end{align*}
Let $\delta\rightarrow 0^+$, then we get that
	$$\limsup_{\delta\rightarrow 0^+}|\dpa r|^2_{\im \pa\dpa K_{\delta}}
	\leq \frac{1}{\frac{\chi''}{(r+\varepsilon^2)^2}+\frac{\chi'\varepsilon^2}{r(r+\varepsilon^2)^2}}.$$

    So we have
	\begin{align*}
	& \limsup_{\delta\rightarrow 0^+}\int_{\tilde X_j} \langle B_\delta^{-1}\overline{\partial}(ue^{\chi(\phi_\varepsilon) }), \overline{\partial}(ue^{\chi(\phi_\varepsilon) })\rangle_he^{-K_{\delta}(z)}dV_{X,\omega}\\
	=&\limsup_{\delta\rightarrow 0^+}\int_{ \tilde X_j} \langle B^{-1}_\delta(\hat{f}\dpa\rho+\chi' \dpa\phi_\varepsilon\wedge u), \hat{f}\dpa\rho+\chi' \dpa\phi_\varepsilon\wedge u\rangle_{h}e^{2\chi(\phi_\varepsilon)-K_{\delta}(z)}dV_{X,\omega}\\
	=&\limsup_{\delta\rightarrow 0^+}\int_{{\tilde X_j}}|\dpa r|^2_{\im \pa\dpa K_{\delta}} |\rho' \hat{f}+(r+\varepsilon^2)^{-1}\chi'  u|^2_{h}e^{2\chi(\phi_\varepsilon)-K_{\delta}(z)}dV_{X,\omega}\\
	\le& \int_{ \tilde X_j}
	\frac{r|(r+\varepsilon^2)\rho'(r) \hat{f}+\chi'(\phi_\varepsilon) u|^2_{h}~e^{\chi(\phi_\varepsilon)-\tau(\phi)}}
	{r{\chi''(\phi_\varepsilon)}+\chi'(\phi_\varepsilon)\varepsilon^2}dV_{X,\omega}<+\infty,
 \end{align*}
 where for the  inequality we use the dominated convergence theorem for fixed $\ve>0$.

Applying Lemma~\ref{modifition} to the metric $h$ and the plurisubharmonic function ${\tau(\phi)+\chi(\phi_\varepsilon)}$ which is approximated by the smooth strictly plurisubharmonic functions $K_{\delta}(z)=\tau(\phi_\delta)+\chi(\phi_{\delta,\varepsilon})+\delta\Psi$ as $\delta\rightarrow 0^+$,
	there exists $\tilde{u}\in L^2_{(n,0)}({\tilde X_j},E, h e^{-\tau(\phi)-\chi(\phi_\varepsilon)})$ satisfying $\dpa \tilde{u}=\dpa (ue^{\chi(\phi_\varepsilon)})$ and
	\begin{align*}
	\int_{\tilde X_j} |\tilde{u}|^2_{h}e^{-\chi(\phi_\varepsilon)-\tau(\phi)}dV_{X,\omega}
	\le  \int_{\tilde X_j}
	\frac{r|(r+\varepsilon^2)\rho'(r) \hat{f}+\chi'(\phi_\varepsilon) u|^2_{h}~e^{\chi(\phi_\varepsilon)-\tau(\phi)}}
	{r{\chi''(\phi_\varepsilon)}+\chi'(\phi_\varepsilon)\varepsilon^2}dV_{X,\omega}.
	\end{align*}	

	Considering that $ue^{\chi(\phi_\varepsilon)}$ solves the equation $\overline{\partial}\tilde{u}=\overline{\partial}(ue^{\chi(\phi_\varepsilon)})$ for $\tilde{u}$ with minimal norm in $L^2_{(n,0)}({\tilde X_j},E, he^{-\chi(\phi_\varepsilon)-\tau(\phi)})$, we have the following estimate
	\begin{align*}
	\int_{\tilde X_j} |ue^{\chi(\phi_\varepsilon)}|^2_{h}e^{-\chi(\phi_\varepsilon)-\tau(\phi)}dV_{X,\omega}
	\le \int_{\tilde X_j}
	\frac{r|(r+\varepsilon^2)\rho'(r) \hat{f}+\chi'(\phi_\varepsilon) u|^2_{h}~e^{\chi(\phi_\varepsilon)-\tau(\phi)}}
	{r{\chi''(\phi_\varepsilon)}+\chi'(\phi_\varepsilon)\varepsilon^2}dV_{X,\omega}.
	\end{align*}
Since
	\begin{align*}
	|(r+\varepsilon^2)\rho'(r) \hat{f}+\chi'(\phi_\varepsilon) u|^2_{h}
	\le 2(r+\varepsilon^2)^2|\rho'|^2|\hat{f}|^2_h+(1+\mathbbm{1}_{\Supp\rho'_{\varepsilon}(r)})
|\chi'(\phi_\varepsilon)|^2|u|^2_h,
	\end{align*}
    	we obtain that {\small
	\begin{align}\label{formula right}
	\int_{\tilde X_j} \Big(1- \frac{r(1+\mathbbm{1}_{\Supp\rho'_{\varepsilon}(r)})|\chi'(\phi_\varepsilon)|^2}
	{r{\chi''(\phi_\varepsilon)}+\chi'(\phi_\varepsilon)\varepsilon^2}\Big)
	|u|^2_he^{\chi(\phi_\varepsilon)-\tau(\phi)}dV_{X,\omega}
	\le
	\int_{\tilde X_j } \frac{2r(r+\varepsilon^2)^2\rho'(r)^2}
	{\chi'(\phi_\varepsilon)\varepsilon^2}|\hat{f}|^2_h  e^{\chi(\phi_\varepsilon)-\tau(\phi)}dV_{X,\omega}.
	\end{align}}
	
	Next we need to control the LHS and RHS of (\ref{formula right}) respectively as $\ve$ goes to $0$.

\vskip0.3cm
\noindent {\bf Step Three: Apply the generalized Siu's lemma to control the  RHS of (\ref{formula right})}
\vskip0.5cm	
	
	Recall that $ \tau(t)=kt,\chi(t)=-\log[-t+\log (-t)]$,
	 and
	$\chi'=\frac{1-t^{-1}}{-t+\log(-t)},$
	then ${e^\chi}/\chi'=1/(1-t^{-1})$.
	Since $|\rho'|\le \frac{1+c_\varepsilon
	}{b_\varepsilon-a_\varepsilon}$, we get
	    	
		\begin{align*}
\int_{\tilde X_j } \frac{2r(r+\varepsilon^2)^2\rho'(r)^2}
	{\chi'(\phi_\varepsilon)\varepsilon^2}|\hat{f}|^2_h  e^{\chi(\phi_\varepsilon)-\tau(\phi)}dV_{X,\omega}
	\le \int_{\tilde X_j\cap\{a_\varepsilon\le r\le b_\varepsilon
			\}  }
		\frac{2(r+\varepsilon^2)^2r
		}{ (1-{\phi_\varepsilon}^{-1})\varepsilon^2}
		\cdot\frac{(1+c_\varepsilon)^2
		}{(b_\varepsilon-a_\varepsilon)^2}
		\cdot\frac{|\hat{f}|^2_h
		}{r^{k}  }dV_{X,\omega}.
		\end{align*}
		
    Take $a_\varepsilon=a\varepsilon^2, b_\varepsilon=b\varepsilon^2$ where $b>a>0$, then it follows immediately that when $a_\ve\le r\le b_\ve$,
	$$
	\frac{(r+\ve^2)^2r
	}{ (1-{\phi_\ve}^{-1})\ve^2}
	\cdot\frac{(1+c_\ve)^2
	}{(b_\ve-a_\ve)^2}
	\le
	\frac{(b\ve^2+\ve^2)^2  b\ve^2
	}{(b\ve^2-a\ve^2)^2\ve^2}
	\cdot\frac{(1+c_\ve)^2}{ (1-{\phi_\ve}^{-1})}
	\le C_1
	,
	$$
	since $\phi_\ve=\log(r+\ve^2)\le -1$.
	
	Thus
\begin{align}\label{111}
	\int_{\tilde X_j } \frac{2r(r+\varepsilon^2)^2\rho'(r)^2}
	{\chi'(\phi_\varepsilon)\varepsilon^2}|\hat{f}|^2_h  e^{\chi(\phi_\varepsilon)-\tau(\phi)}dV_{X,\omega}\le C_1\int_{\tilde X_j\cap\{a_\varepsilon\le r\le b_\varepsilon\}  }
	\frac{|\hat{f}|^2_h
	}{r^{k}  }dV_{X,\omega}.
\end{align}

		Since $|\hat{f}|^2_{\frac{h}{\det h}}$ is non-negative and upper semi-continuous on $\tilde X$,  applying the Lemma~\ref{lem 2} we  obtain that
	\begin{align*}
	\limsup_{\varepsilon\rightarrow0^+} \frac{1}{\varepsilon^{2k}}\int_{\tilde{{X}}_j\cap\{r\le b\varepsilon^2\}} |\hat{f}|^2_{\frac{h}{\det h}}e^{-\varphi}dV_{X,\omega}
	&\le b^{k}\int_{\tilde X_{j+1}\cap \tilde S}|\hat{f}|^2_{\frac{h}{\det h}}e^{-\varphi}d\mu_r\\
	&\le b^{k}\int_{S}|f|^2_{ h}d\mu_r.
	\end{align*}
	So
\begin{align}\label{siu lemma result}
	\limsup_{\ve\to0^+}\int_{\tilde X_{j}\cap\{a_\ve \le r\le b_\ve\}}\frac{|\hat{f}|^2_h}{r^k}dV_{X,\omega}\le \limsup_{\ve\to0^+}\int_{\tilde X_{j}\cap\{r\le b_\ve\}}\frac{|\hat{f}|^2_h}{a^k\ve^{2k}}dV_{X,\omega}
	\le \frac{b^{k}}{a^k}\int_{S}|f|^2_{ h}d\mu_r.
\end{align}
	Together with (\ref{111}), we have
\begin{align}\label{a1}
	\limsup_{\ve\to0^+}\int_{\tilde X_j } \frac{2r(r+\varepsilon^2)^2\rho'(r)^2}
	{\chi'(\phi_\varepsilon)\varepsilon^2}|\hat{f}|^2_h  e^{\chi(\phi_\varepsilon)-\tau(\phi)}dV_{X,\omega}
	\le C_2\int_{S}|f|^2_{ h}d\mu_r
\end{align}
	for some positive constant $C_2$ only depending on $k$.

\vskip0.3cm
\noindent {\bf Step Four: Estimate the  LHS of (\ref{formula right})}
\vskip0.5cm	

Since
  $\tau(t)=kt,\chi(t)=-\log[-t+\log (-t)]$ when $t<-1$, then
 	$$\chi'=\frac{1-t^{-1}}{-t+\log(-t)},~~ \chi''=\frac{(1-t^{-1})^2}{[-t+\log(-t)]^2}+\frac{t^{-2}}{-t+\log(-t)}.$$

On $ \{\mathbbm{1}_{\Supp\rho'_{\varepsilon}(r)}=0\}$, we have {\small	
	\begin{align*}
	\Big(1-(1+\mathbbm{1}_{\Supp\rho'_{\varepsilon}(r)}) \frac{r|\chi'(\phi_\ve)|^2}
	{r{\chi''(\phi_\ve)}+\chi'(\phi_\ve)\ve^2}\Big)
	e^{\chi(\phi_\ve)}
	&\ge\Big(1- \frac{|\chi'(\phi_\ve)|^2}
	{{\chi''(\phi_\ve)}}\Big)
	e^{\chi(\phi_\ve)}\\
	&=\Big(1-\frac{(1-{\phi_\ve}^{-1})^2 }{(1-{\phi_\ve}^{-1})^2+\frac{-{\phi_\ve}+\log(-{\phi_\ve})}{{\phi_\ve}^2}}\Big)
	\frac{1
	}{-{\phi_\ve}+\log(-{\phi_\ve})  }\\
	&=
	\frac{\phi_\ve^{-2} }{(1-{\phi_\ve}^{-1})^2+\frac{-{\phi_\ve}+\log(-{\phi_\ve})}{{\phi_\ve}^2}}
	\ge C_3 \phi_\ve^{-2}
	\end{align*}}
	for some positive constant $C_3$ independent of $\ve$ since $-\phi_\ve=-\log(r+\ve^2)\ge 1$.
	
    On $\{\mathbbm{1}_{\Supp\rho'_{\varepsilon}(r)}=1\}$, we have $a\ve^2\le r \le b\ve^2$. Since $$-\log({(b+1)\ve^2})\le -\phi_\ve\le -\log({(a+1)\ve^2})$$ for $\ve$ small enough, we get that
    $
    \chi'(\phi_\ve)=\frac{1-\phi_\ve^{-1}}{-\phi_\ve+\log(-\phi_\ve)} \thicksim\frac{1}{-\phi_\ve}$ and $e^{\chi(\phi_\ve)}=\frac{1}{-\phi_\ve+\log(-\phi_\ve)}\thicksim\frac{1}{-\phi_\ve}$.
    Therefore,
    \begin{align*}
	\Big(1- \frac{2r|\chi'(\phi_\ve)|^2}
	{r{\chi''(\phi_\ve)}+\chi'(\phi_\ve)\ve^2}\Big)
	e^{\chi(\phi_\ve)}
	\ge & \Big(1- \frac{2r|\chi'(\phi_\ve)|^2}
	{\chi'(\phi_\ve)\ve^2}\Big)
	e^{\chi(\phi_\ve)}\\
	=& \big(1-\frac{2r}{\ve^2} \chi'(\phi_\ve)\big)e^{\chi(\phi_\ve)}
	\ge  C_4 \phi_\ve^{-2}
	\end{align*}
	for some positive constant $C_4$ independent of $\ve$.
	
    Thus for $\varepsilon>0$ small enough, the LHS of (\ref{formula right})
	\begin{align}
	\int_{\tilde X_j} \Big(1- \frac{r(1+\mathbbm{1}_{\Supp\rho'_{\varepsilon}(r)})|\chi'(\phi_\varepsilon)|^2}
	{r{\chi''(\phi_\varepsilon)}+\chi'(\phi_\varepsilon)\varepsilon^2}\Big)
	|u|^2_he^{\chi(\phi_\varepsilon)-\tau(\phi)}dV_{X,\omega}
	\ge
	C_5\int_{\tilde X_j}
	\frac{|u|^2_h}{r^k (\log({r+\ve^2}))^2}dV_{X,\omega}.
	\end{align}
	
\vskip0.3cm
\noindent {\bf Step Five: Construct a holomorphic extension with the demanded estimate. }
\vskip0.5cm	

Since $\dpa u=\hat{f}  \dpa \rho(r)=0$ in $\{r< a_\varepsilon\}\cap \tilde X_j$,  $u$ is holomorphic near $\tilde S\cap \tilde X_j$.
	 Since $r^{-k}$ is not integrable near $\tilde S$, combining (\ref{formula right}) and (\ref{a1}), we get that $u|_{\tilde S\cap \tilde X_j}=0$ by the Fubini-Tonelli theorem.

	Let $F_{\varepsilon,j}=\hat{f}\rho(r)-u$ on $\tilde X_j$, then $\dpa F_{\varepsilon,j}=v-\dpa u=0$ and $F_{\varepsilon,j}=f$ on $\tilde S\cap \tilde X_j$. And by (\ref{formula right}) and (\ref{a1}), we obtain
that
\begin{align*}
	\int_{\tilde X_j}
	|F_{\varepsilon,j}|^2_h dV_{X,\omega}
	\le&
	2\int_{\tilde X_j}
	|\hat{f}\rho_\ve(r)|^2_hdV_{X,\omega}
	+2C_7\int_{\tilde X_j}
	\frac{|u|^2_h}{r^k (\log({r+\ve^2}))^2}dV_{X,\omega}\\
	\le &
	2\int_{\tilde X_j\cap\{r\le b_\ve\}}
	|\hat{f}|^2_hdV_{X,\omega}
	+ 2\frac{C_2C_7}{C_5}\int_S |f|_h^2 d\mu_r,
\end{align*}	
	where $C_7=\sup_{0\le r<e^{-2}} r^k(\log{r})^2<+\infty$.

	Hence $\{F_{\varepsilon,j}\}_\ve$ is $L^2$ bounded on $\tilde{X}_j$ for each fixed $j$, by Montel' theorem we can choose a subsequence of $\{F_{\varepsilon,j}\}_\ve$ compactly converging to a holomorphic section $F_j$ on $\tilde{X}_j$.
	For  convenience, we denote such subsequence by $F_{\ve,j}$.

Now we need to estimate $$\int_{\tilde X_j}
	\frac{|F_{j}|^2_h}{r^k (\log{r})^2}dV_{X,\omega}.$$

Restrict the integrand to the region  $\tilde X_j\cap\{r\ge a_\ve\}$ and we get that
	
\begin{align}\label{ineq restricted}
	\int_{\tilde X_j\cap\{r\ge a_\ve\}}
	\frac{|F_{\varepsilon,j}|^2_h}{r^k (\log{r})^2}dV_{X,\omega}	
    &\le 2\int_{\tilde X_j\cap\{r\ge a_\ve\}}
    \frac{|\hat{f}\rho_\ve(r)|^2_h}{r^k (\log{r})^2}dV_{X,\omega}+2\int_{\tilde X_j}
	\frac{|u|^2_h}{r^k (\log{r})^2}dV_{X,\omega}\nonumber\\  & \le
    2\int_{\tilde X_j\cap\{r\ge a_\ve\}}
    \frac{|\hat{f}\rho_\ve(r)|^2_h}{r^k (\log({r+\ve^2}))^2}dV_{X,\omega}
    + 2\frac{C_2}{C_5}\int_S |f|_h^2 d\mu_r,
\end{align}
	
By (\ref{siu lemma result}), we have
	$$ \lim_{\ve\to0^+}\int_{\tilde X_j\cap\{r\ge a_\ve\}}
	\frac{|\hat{f}\rho_\ve(r)|^2_h}{r^k (\log({r+\ve^2}))^2}dV_{X,\omega}
	\le
	\lim_{\ve\to0^+}
	\frac{1}{(\log({b\ve^2+\ve^2}))^2}
	\int_{\tilde X_j\cap\{ a_\ve\le r\le b_\ve\}}
	\frac{|\hat{f}|^2_h}{r^k}dV_{X,\omega}=0.$$

Then in (\ref{ineq restricted}), let $\ve$ goes to $0$, and by Fatou's lemma we obtain that
	
\begin{align*}
	\int_{\tilde X_j}
	\frac{|F_{j}|^2_h}{r^k (\log{r})^2}dV_{X,\omega}	
	\le
	2\frac{C_2}{C_5}\int_S |f|_h^2 d\mu_r,
\end{align*}
	
	By a standard limitation argument via Montel's theorem again,  there is a subsequence of$\{F_j\}$ compactly converging to an $F\in H^0(\tilde X,K_X\otimes E)$ with
\begin{align*}
	\int_{\tilde X}
	\frac{|F|^2_h}{r^k (\log{r})^2}dV_{X,\omega}	
	\le
	2\frac{C_2}{C_5}\int_S |f|_h^2 d\mu_r.
\end{align*}
 \end{proof}	

\begin{rem}
\begin{enumerate}
  \item [$(1)$]
In fact, taking $a=\frac{k}{k+2}, b=1$ and calculating more precisely, we can obtain the estimate with the constant $2\frac{C_2}{C_5}=3e^2(k+2)^2$.
	The process is boring and cumbersome, and we omit it here.
  \item [$(2)$]
we can also obtain the above $L^2$ extension theorem for quasi-Stein manifolds (for example, projective manifolds).  	
\item[$(3)$] From our proof, we can actually weaken the condition that $h$ is Nakano semi-positive to be that
$h^*$ is locally bounded above, $|u|^2_\frac{h}{\det h}$ is upper semi-continuous for any local holomorphic section $u$, $\det h$ is Griffiths semi-positive and $h$ is $L^2$ optimal.
\end{enumerate}	
\end{rem}

\cite[Theorem 3.2]{DNWZ20} (\cite[Theorem 3.5]{HI19}) showed that the {\em multiple coarse $L^2$ estimate} condition implies Griffiths semi-positivity of (locally H\"older) continuous Hermitian metrics by proving an $L^2$ extension theorem from a single point. Similarly, we can show the following result by Theorem \ref{main1} but omit the proof.

\begin{cor}
  Let $E$ be a holomorphic vector bundle endowed with a singular Hermitian metric $h$. Assume that
  \begin{enumerate}
    \item[$(i)$]  $h^*$ and $h\otimes \det h^*$ are upper semi-continuous;
    \item[$(ii)$]  $\det h$ is Griffiths semi-positive;
    \item[$(iii)$] $(E,h)$ satisfies the multiple coarse $L^2$-estimate condition $($see \cite[Definition 1.1(2)]{DNWZ20}$)$.
  \end{enumerate} Then $(E,h)$ satisfies the multiple coarse $L^2$-extension condition $($see \cite[Definition 1.2(2)]{DNWZ20}$)$. In particular, $h$ is Griffiths semi-positive.
\end{cor}

	Let  $Y$ be a closed complex submanifold f $X$ and $\iota$ the embedding map, then Griffiths positivity of $(E,h)$ implies Griffiths positivity of $(\iota^*E,\iota^*h)$ for singular Hermitian metrics and  Nakano positivity of $(E,h)$ implies Nakano positivity of $(\iota^*E,\iota^*h)$ for smooth Hermitian metrics.
	It is natural to ask whether Nakano positivity of $(E,h)$ implies Nakano positivity of $(\iota^*E,\iota^*h)$ for singular Hermitian metrics.
	Following the idea in the proof of  Theorem \ref{main1}, we can answer the question affirmatively.

\begin{pro}\label{pro nak res nak}
	 If $(E,h)$ is a locally Nakano semi-positive on $X$, then for any complex submanifold $Y\subset X$ such that $\iota^*\det h$ is finite almost everywhere on $Y$, $(\iota^*E,\iota^*h)$ is also locally  Nakano semi-positive on $Y$, where $\iota$ is the embedding map from $Y$ to $X$.
\end{pro}

\begin{proof}
	By induction, we only need to prove the case when $\dim Y=n-1$.
	It is clear that $\iota^*h$ is Griffith semi-positive on $Y$.
	Thus we only need to verify $\iota^*h$ satisfies the optimal $L^2$-estimate condition. Fixed $y\in Y$, since $y\in X$ and $h$ is locally Nakano semi-positive,
	we can find a neighborhood $U$ of $y$ in $X$, such that $E|_{U}$ is trivial and for every Stein neighborhood $\Omega\subset U$ of $y$,  $h$ satisfies the optimal $L^2$-estimate condition on $\Omega$.
	With loss of generality, we may assume $U$ is a unit polydisc $\triangle^n$ with center at $y$,
	and $U\cap Y=\{x\in U, z_n=0\}$ be the $n-1$ dimensional unit polydisc with center at $y$.
	It is obvious that $\iota^*h|_{U\cap Y}$ is trivial. We need to prove for every Stein open neighborhood $\Omega$ around $y$ which satisfies $\Omega=(\Omega\cap Y)\times \{|z_n|<1\}$, $\iota^*h$ satisfies the optimal $L^2$-estimate condition on $\Omega\cap Y$. By similarly argument we have used in the Theorem \ref{main1}, we may assume that $e^{-\varphi}\in L_\loc^1(\Omega)$.

  Choose a  K\"ahler form $\omega$ on $\Omega$, then $\iota^*\omega$ is a K\"ahler form on $\Omega\cap Y$, such that $\omega=\iota^*\omega+\im dz_n\wedge d\overline z_n$. And choose a smooth strictly plurisubharmonic function $\psi$ on $\Omega\cap Y$, and $\dpa$-closed $f\in L^2_{(n-1,1)}(\Omega\cap Y,\iota^*w,\iota^*E|_{\Omega\cap Y},\iota^*h e^{-\psi})$. Our goal is to find a $u\in L^2_{(n-1,0)}(\Omega\cap Y,\iota^*w,\iota^*E|_{\Omega\cap Y},\iota^*h e^{-\psi})$, such that $\dpa u=f$ on $\Omega\cap Y$ and
 $$
    \int_{\Omega\cap Y} |u|^2_{\omega\cap Y,\iota^*h} e^{-\psi} dV_{\iota^*\omega}
    \le \int_{\Omega\cap Y} \langle B^{-1}_{\iota^*\omega,\psi} f,f \rangle_{\iota^*\omega,\iota^*h} e^{-\psi} dV_{\iota^*\omega},
   $$
    provided that the right hand is finite, where $B_{\iota^*\omega,\psi}=[\im\partial\overline{\partial}\psi \otimes {\rm Id}_E, \Lambda_{\iota^*\omega}]$.

Since we can take a Stein exhaustion of $\Omega$ and regularize $f$ by convolutions, we may assume that  $e^{-\varphi}$ is integrable in a neighborhood of $\Omega$ and the $\dpa$-closed form $f$ is smooth in a neighborhood of $\Omega\cap Y$.

	We define $\tilde{f}=\pi^*f\wedge dz_n$, $\Psi=\pi^*\psi$ where $\pi$ is the standard projection from $\Omega$ to $\Omega\cap Y$,
	then it is easy to see $\dpa\tilde{f}=0$ on $\Omega_{\epsilon}=(\Omega\cap Y)\times \{|z_n|<\epsilon\}$ and $\tilde f\in C^{\infty}_{(n,1)} (\Omega_{\epsilon},w,E|_{\Omega_\epsilon},h e^{-\Psi})$.
By Lemma \ref{modifition}, there exists a $v_{\epsilon}\in L^2_{n,0}(\Omega_{\epsilon},w,E|_{\Omega_{\epsilon}},h e^{-\Psi})$, such that $\dpa v_{\epsilon}=\tilde f$ and
     \begin{align}
    \int_{\Omega_{\epsilon}} |v_{\epsilon}|^2_{\omega,h} e^{-\Psi} dV_\omega
    \le \int_{\Omega_{\epsilon}} \langle B^{-1}_{\omega,\Psi}\tilde f,\tilde f \rangle_{\omega,h} e^{-\Psi} dV_{\omega},
 \end{align}
     provided that the right hand is finite, where $B_{\omega,\Psi}=[\im\partial\overline{\partial}\Psi \otimes {\rm Id}_E, \Lambda_{\omega}]$.

     Additionally, by the weak regularity of $\dpa$ on $(n,0)$-forms, we can take $\tilde u$ to be smooth.
 Then for some $\epsilon$ small enough, we have
 $\int_{\Omega_{\epsilon}} \langle B^{-1}_{\omega,\Psi}\tilde f,\tilde f \rangle_{\omega,h} e^{-\Psi} dV_{\omega}<+\infty$.
And by Lemma \ref{lem 2},
 \begin{equation}\label{bbb}
  \limsup_{\epsilon\to 0}\frac{1}{\pi \epsilon^2}\int_{\Omega_{\epsilon}} \langle B^{-1}_{\omega,\Psi}\tilde f,\tilde f \rangle_{\omega,h} e^{-\Psi} dV_{\omega}\le\int_{\Omega\cap Y }\langle B^{-1}_{\iota^*\omega,\psi} f,f \rangle_{\iota^*\omega,\iota^*h} e^{-\psi} dV_{\iota^*\omega}.
 \end{equation}

Using the Fubini-Tonelli theorem, we know for every $\epsilon>0$, there exists a $\xi\in \mathbb C$, $|\xi|<\epsilon$,
and $$\int_{\Omega\cap\{z_n=\xi\} }|v_{\epsilon}(z',\xi)|^2_{\omega,h}e^{-\Psi}dV_{\omega}\le\frac{1}{\pi\epsilon^2}\int_{\Omega_{\epsilon}} |v_{\epsilon}|^2_{\omega,h} e^{-\Psi} dV_\omega.$$
So we choose a sequence $\epsilon_k\to 0$, for every $k$, there exists a $\xi_k$, $|\xi_k|<\epsilon_k$, and  $$\int_{\Omega\cap\{z_n=\xi_k\} }|v_{\epsilon_k}(z',\xi_k)|^2_{\omega,h}e^{-\Psi}dV_{\omega}\le\frac{1}{\pi\epsilon_k^2}\int_{\Omega_{\epsilon_k}} |v_{\epsilon_k}|^2_{\omega,h} e^{-\Psi} dV_\omega.$$
Let $u_{\epsilon}=v_{\epsilon}/dz_n$, then we have $\dpa u_{\epsilon}=\pi^* f$. We denote $\iota_k$
be the inclusion map from $\{z_n=\xi_k\}\cap\Omega$ to $\Omega$. We denote $u_k=\iota_k^*u_{\epsilon_k}$, $h_k=\iota_k^*h$, $\omega_k=\iota_k^*\omega$.
Then we have $\dpa u_k=f$ and
\begin{equation}\label{aaa}
  \int_{\Omega\cap Y}|u_k|^2_{\omega_k,h_k}e^{-\psi}dV_{\iota^*\omega}\le \frac{1}{\pi\epsilon_k^2}\int_{\Omega_{\epsilon_k}} |v_{\epsilon_k}|^2_{\omega,h} e^{-\Psi} dV_\omega.
\end{equation}
Noticed that $\dpa(u_k-u_1)=0$, so $\{u_k-u_1\}$ is a sequence of holomorphic section over $\Omega\cap Y$.
And we know $\int_{\Omega\cap Y}|u_k-u_1|^2dV_{\iota^*\omega}$ is locally uniformly bounded since $h$ is locally bounded above.  Therefore $\{u_k-u_1\}$ is a normal family and there exists a subsequence (for simplicity, we still denote by $\{u_k-u_1\}$) converging to a holomorphic section $u-u_1$.
Since $\dpa u=\dpa (u-u_1)+\dpa u_1=f$,  so we only need to prove that
$$\int_{\Omega\cap Y}|u|^2_{\iota^*\omega,\iota^*h}e^{-\psi}dV_{\iota^*\omega}\leq\int_{\Omega\cap Y }\langle B^{-1}_{\iota^*\omega,\psi} f,f \rangle_{\iota^*\omega,\iota^*h} e^{-\psi} dV_{\iota^*\omega}.$$

Since $h$ is Griffith semi-positive, it follows from Lemma \ref{prop lowerbound} that there exists a family of smooth Griffiths positive metrics $h^j$  such that $h^j$ increasingly converges to $h$ on $\Omega$. We denote $h^j_k=\iota^*_kh^j$, then $h^j_k$ is smooth and increasingly converges to $h_k $ and $\lim\limits_{k\to\infty}h^j_k=\iota^*h^j$.
\begin{align*}
    \int_{\Omega\cap Y}|u|^2_{\iota^*\omega,\iota^*h}e^{-\psi}dV_{\iota^*\omega}
    &=\lim_{j\to\infty}\int_{\Omega\cap Y}|u|^2_{\iota^*\omega,\iota^*h^j}e^{-\psi}dV_{\iota^*\omega}\\
    &=\lim_{j\to\infty}\int_{\Omega\cap Y}\liminf_{k\to\infty}|u_k|^2_{\iota_k^*\omega,h^j_k}e^{-\psi}dV_{\iota^*\omega}\\
    &\le\lim_{j\to\infty}\liminf_{k\to\infty}\int_{\Omega\cap Y}|u_k|^2_{\iota_k^*\omega,h^j_k}e^{-\psi}dV_{\iota^*\omega}\\
    &\le\liminf_{k\to\infty}\int_{\Omega\cap Y}|u_k|^2_{\omega_k,h_k}e^{-\psi}dV_{\iota^*\omega}\\
    &\le \liminf_{k\to\infty}\frac{1}{\pi\epsilon_k^2}\int_{\Omega_{\epsilon_k}} |v_{\epsilon_k}|^2_{\omega,h} e^{-\Psi} dV_\omega\\
    &\le\int_{\Omega\cap Y }\langle B^{-1}_{\iota^*\omega,\psi} f,f \rangle_{\iota^*\omega,\iota^*h} e^{-\psi} dV_{\iota^*\omega},
    \end{align*}
    where the first inequality is due to Fatou's lemma, the second inequality is owing to $h_k^j\le h_k$,
    the third inequality comes from $(\ref{aaa})$ and the last inequality is from $(\ref{bbb})$.
    We complete the proof.
\end{proof}
\begin{rem}
  From the proof, one can see that the open subset $U_Y$ of $Y$ such that $h|_{U_Y}$ is Nakano semi-positive  depends only on the  open subset $U_X$ of $X$ such that $h|_{U_X}$ is Nakano semi-positive and the relative positions of $Y$ and $X$ as well as the singularity set of $-\log\det h$.
\end{rem}

\subsection{Properties for multiplier submodule sheaves}
In this section, we discuss some further properties of  the multiplier submodule sheaves which are analogous to those for the multiplier ideal sheaves in \cite{DEL} and \cite{GZ20}.  In \cite{In20}, Inayama proved the coherence of multiplier submodule sheaf  $\mathcal{E}(h)$ for a (locally) Nakano semi-positive singular Hermitian metric $h$.
\begin{pro}[\cite{In20}]\label{pro coherent}
		If $h$ is  (locally) Nakano semi-positive, then $\mathcal{E}(h)$ is a coherent subsheaf of $\mathcal{O}(E)$. Moreover, if $\Omega$ is a bounded Stein open subset of $X$ such that $E|_{\Omega}=\Omega\times\C^r$, and $h|_\Omega$ is Nakano semi-positive, then the sheaf  $\mathcal{E}(h)$ is generated by any basis of the Hilbert space $\mathcal{H}^2(\Omega,h)$ of holomorphic sections $f\in \cal O(E)(\Omega)$ such that $\int_{\Omega}|f|^2_hdV<+\infty$.
	\end{pro}
	
As a direct corollary of Theorem \ref{main1}, we obtain that
\begin{cor}[Restriction Formula]\label{cor res for}
   If $(E,h)$ is a locally Nakano semi-positive metric on $X$, then for any complex submanifold $Y\subset X$ such that $\det h|_Y$ is finite almost everywhere on $Y$, $\mathcal{E}(h|_Y)\subset \mathcal{E}(h)|_Y$.
   \end{cor}

   \begin{pro}[Addition Formula]
   Let $X_1,X_2$ be complex manifolds, $\pi_i:X_1\times X_2\to X_i$, $i=1,2$ be the projections and  $E_i$ be  a holomorphic vector bundle on $X_i$ of rank $r_i$ with a singular Hermitian metric $h_i$ respectively, $i=1, 2$. Assume  that $(E_1\hat\otimes E_2, \pi_1^{*}h_1\otimes \pi_2^{*}h_2)$ is (locally) Nakano semi-positive  on $X_1\times X_2$, then both $(E_i, h_i)$ are locally Nakano semi-positive and
    $$\mathcal E(\pi_1^{*}h_1\otimes \pi_2^{*}h_2)=\pi_1^{*}\mathcal{E}({h_1})\hat \otimes\pi_2^{*}\mathcal{E}(h_2).$$
   \end{pro}
\begin{proof}
Since  $(E_1\hat\otimes E_2, \pi_1^{*}h_1\otimes \pi_2^{*}h_2)$ is (locally) Nakano semi-positive, by Proposition \ref{pro nak res nak},
we can choose a relatively compact Stein open subsets $U_2\subset X_2$ such that  $E_2|_{U_2}$ is trivial, $x_2\in U_2$ such that $(E_1\hat\otimes E_2|_{X_1\times\{x_2\}}, \pi_1^{*}h_1\otimes \pi_2^*h_2)$ is locally Nakano semi-positive.
Let $\{e_j\}$ be an orthonormal frame of  the trivial vector bundle $E_2|_{U_2}$ such that  $h_2(x_2)=I_{r_2}$. Then for any $x_1\in X_1$, there exists a Stein coordinate $\Omega$ such that $E_1|_\Omega$ is trivial and for any $U_1\subset\Omega$, any K\"ahler form $\omega_1$, smooth strictly plurisubharmonic function $\Psi$ and $\dpa$-closed $f\in C^\infty_{(n,1)}(U_1, E_1, \omega_1, h_1e^{-\Psi})$, we can solve the equation
$$\dpa u=f\otimes e_1$$
on $U_1\times \{x_2\}$ with an estimate
$$\int_{U_1\times \{x_2\}}^{}|u|^2_{\omega_1,\pi_1^{*}h_1\otimes I_{r_2}}e^{-\Psi}dV_{\omega_1}\le\int_{U_1\times  \{x_2\}}^{}\langle B^{-1}_\Psi f\otimes e_1,f\otimes e_1\rangle _{\omega_1,\pi_1^{*}h_1\otimes I_{r_2}}e^{-\Psi}dV_{\omega_1}.$$

Since $u=\sum_{j=1}^{r_2}u_j\otimes e_j$ and $\dpa u=\sum_{j=1}^{r_2}\dpa u_j\otimes e_j$, we obtain that $\dpa u_1=f$ and $\dpa u_j=0$ for $j\ge 2$. Then we can view $u_1$ as an $(n,0)$-form on $E_1|_{U_1}$ by identifying $U_1\times \{x_2\}$ with $U_1$. Then $$\dpa u_1=f$$
with an estimate
\begin{align*}
  \int_{U_1}^{}|u_1|^2_{\omega_1,h_1}e^{-\Psi}dV_{\omega_1} &\le \int_{U_1\times \{x_2\}}^{}|u|^2_{\omega_1,\pi_1^{*}h_1\otimes I_{r_2}}e^{-\Psi}dV_{\omega_1} \\
   &\le  \int_{U_1\times  \{x_2\}}^{}\langle B^{-1}_\Psi f\otimes e_1,f\otimes e_1\rangle _{\omega_1,\pi_1^{*}h_1\otimes I_{r_2}}e^{-\Psi}dV_{\omega_1}\\
   &= \int_{U_1}^{}\langle B^{-1}_\Psi f,f\rangle _{\omega_1,h_1}e^{-\Psi}dV_{\omega_1}.
\end{align*}
Therefore, $(E_1, h_1)$ is locally Nakano semi-positive and so is $(E_2, h_2)$.

 Now we choose two relatively compact Stein open subsets $U_1\subset X_1, U_2\subset X_2$ such that $E_1|_{U_1} $ and $E_2|_{U_2}$ is trivial and $h_1, h_2, \pi_1^{*}h_1\otimes\pi^*h_2$ are Nakano semi-positive respectively. Then $\mathcal{H}^2(U_1\times U_2, \pi_1^{*}h_1\otimes \pi_2^{*}h_2)$ is the Hilbert tensor product of $\mathcal{H}^2(U_1,h_1)$ and $\mathcal{H}^2(U_2,h_2)$, and admits $\{f_k\boxtimes g_k\}$ as a Hilbert basis, where $\{f_k\}$ and $\{g_k\}$ are respective Hilbert bases. By Proposition \ref{pro coherent}, $\mathcal{E}(\pi_1^{*}h_1\otimes \pi_2^{*}h_2)|_{U_1\times U_2}$ is generated as an $\mathcal{O}_{U_1\times U_2}$
 module by the $\{f_k\boxtimes g_k\}$. we obtain $$\mathcal E(\pi_1^{*}h_1\otimes \pi_2^{*}h_2)=\pi_1^{*}\mathcal{E}({h_1})\hat \otimes\pi_2^{*}\mathcal{E}(h_2).$$
\end{proof}

   \begin{cor}[Subadditivity Fomula]
    Let $X$ be a complex manifold, $\pi:X\times X\to X$ be the projection and  $E_i, i=1,2$ be  holomorphic vector bundles on $X$ of rank $r_i$ with a singular Hermitian metric $h_i$ respectively. Assume that $(E_1\hat \otimes E_2, \pi^{*}h_1\otimes \pi^{*}h_2)$ is Nakano semi-positive  on $X\times X$, then
    $$\mathcal E(h_1\otimes h_2)\subset\mathcal{E}(h_1)\hat \otimes\mathcal{E}(h_2).$$
   \end{cor}
   \begin{proof}
       We apply  the addition formula to $X_1=X_2=X$ and the restriction formula to the  diagonal $Y$ of $X\times X$.
       Then
       \begin{align*}
           \mathcal{E}(h_1\otimes h_2)&= \mathcal{E}(\pi_1^{*}h_1\otimes \pi_2^{*}h_2|_Y)\\
           &\subset \mathcal{E}(\pi_1^{*}h_1\otimes \pi_2^{*}h_2)|_Y\\
           &=\pi_1^{*}\mathcal{E}({h_1})\hat \otimes\pi_2^{*}\mathcal{E}(h_2)|_Y\\
           &=\mathcal{E}(h_1)\hat \otimes\mathcal{E}(h_2).
       \end{align*}
   \end{proof}

   \begin{cor}[Pull Back Formula]
	Let $f : X\to Y$ be a holomorphic map, and $(E,h)$ be a holomorphic vector bundle on $Y$ where $h$ is a singular Hermitian metric, $\pi_Y: X\times Y\to Y$ and $\pi_X : X\times Y\to X$ be  the standard projections.
	Assume that $(\pi_Y^*E, \pi_Y^{*}h)$ is locally Nakano semi-positive on $X\times Y$.
	Then  $(E,h)$ is locally Nakano semi-positive on $Y$ and $(f^*E,f^*h)$ is locally Nakano semi-positive on $X$.
    In addition, $\mathcal{E}(f^*h)\subset f^*\mathcal{E}(h)$.
\end{cor}
\begin{proof}
	Let $i_Y: Y\to X\times Y$ be the inclusion map, then $\pi_Y\circ i_Y={\rm id}_Y$. Therefore $h={\rm id}_Y^*h=(\pi_Y \circ i_Y)^*h=i_Y^*(\pi_Y^*h)$, and hence $h$ is locally Nakano semi-positive by Proposition \ref{pro nak res nak}.

    Let
    $$\Gamma_f=\{(x,f(x));x\in X\}\subset X\times Y$$
    be the graph of $f$, $i_f$ be the inclusion map from $\Gamma_f$ to $X\times Y$ . Noticing that $\pi_X|_{\Gamma_f}$ is a biholomorphism from $\Gamma_f$ to $X$, and $f=\pi_Y\circ i_f\circ \pi_X|_{\Gamma_f}^{-1}$, we conclude that  $$f^*h=(\pi_Y\circ i_f\circ \pi_X|_{\Gamma_f}^{-1})^*h=(\pi_X|_{\Gamma_f}^{-1})^*i_f^*(\pi_Y^*h)$$ is locally  Nakano semi-positive by Proposition \ref{pro nak res nak}. Moreover, using Corollary \ref{cor res for}, we get
\begin{align*}
\mathcal{E}(f^*h)&=\mathcal{E}((\pi_Y\circ i_f\circ \pi_X|_{\Gamma_f}^{-1})^*h)\\
&=\mathcal{E}((\pi_X|_{\Gamma_f}^{-1})^*i_f^*\pi_Y^*h)\\
&=(\pi_X|_{\Gamma_f}^{-1})^*\mathcal{E}(i_f^*\pi_Y^*h)\\
&\subset (\pi_X|_{\Gamma_f}^{-1})^*i_f^*\mathcal{E}(\pi_Y^*h)\\
&=(\pi_X|_{\Gamma_f}^{-1})^*i_f^*\pi_Y^*\mathcal{E}(h)\\
&=(\pi_Y\circ i_f\circ \pi_X|_{\Gamma_f}^{-1})^*\mathcal{E}(h)\\
&=f^*\mathcal{E}(h),
\end{align*}
where the equality in the fifth line holds is due to  $\mathcal{E}(\pi_Y^*h)=\pi_Y^*\mathcal{E}(h)$. For completeness, we give the proof in the following:

It  is obvious that $\pi_Y^*\mathcal{E}(h)\subset\mathcal{E}(\pi_Y^*h)$. Conversely,
for a point $p\in X\times Y$, assume $f_p\in\mathcal{E}(\pi^*_Y h)_p$,  by definition, there is  a neighborhood $U$ of $p$,  such that $\int_U|f|^2_h <+\infty$. We choose  a local chart $(z,w)$ near on $U$, and we may assume $p$ is the original point $o$ and $U$ is the product of two polydiscs $\Delta_x\times \Delta_y$ centered at $o$.
Since $f\in \Gamma(\Delta_x\times \Delta_y, \pi_Y^*E)$, we can write $$f(z,w)=\Sigma_{I}u_I(w)z^I,$$
where $u_I\in \Gamma(\Delta_y, E)$.  Therefore,
\begin{align*}
 &\4l \int_{\Delta_x\times \Delta_y}|f(z,w)|^2_{\pi_Y^*h} dV_z dV_w\\
 &=\int_{\Delta_x\times \Delta_y}|\Sigma_{I}u_I(w)z^I|^2_{\pi_Y^*h} dV_z dV_w\\
 &=\int_{\Delta_y}\int_{\Delta_x}|\Sigma_{I}u_I(w)z^I|^2_{\pi_Y^*h}dV_z dV_w\\
 &=\int_{\Delta_y}\Sigma_I C_I|u_I(w)|^2_hdV_w<+\infty .
\end{align*}
We get $u_I\in \mathcal{E}(h)$ for each $I$, and then $f\in \pi_Y^*\mathcal{E}(h)$ since $\pi_Y^*\mathcal{E}(h)$ is coherent. Thus we obtain  that $\mathcal{E}(\pi_Y^*h)\subset \pi_Y^*\mathcal{E}(h)$.
\end{proof}

\subsection{Strong openness property of  multiplier submodule sheaves}
In this section, we prove the strong openness property of  multiplier submodule sheaves by using the $L^2$ extension theorem movably  which is invented by Guan-Zhou in \cite{GZ15}.

  Now we deduce strong openness of multiplier submodule sheaves following Lempert's language.
Let $U\subset \C^n$ be open and $E$ is a holomorphic vector bundle over $U$.
Let $h_1\ge h_2\ge\cdots$ be locally Nakano semi-positive on $E$ and assume that $h=\lim_j h_j$ is locally bounded below by a continuous Hermitian metric.
If $H\subset \C_n$ be a complex hyperplane, $W\subset H\cap U$ is relatively open and $f$ is a measurable section of $E|_W$, Lempert in \cite{Lempert17} defined
\begin{align*}
\|f\|^2=\inf_j \int_W |f|^2_{h_j}\in [0,+\infty]
\end{align*}
and proved the following lemma:
\begin{lem}\label{lem lempert}
	The following two statements are equivalent:
\begin{enumerate}
  \item[$(1)$] The germ $F_o\in \cal{O}(E)_o$ belongs to $\bigcup_j \cal{E}(h_j)_o$;
  \item[$(2)$] For any sufficiently small neighborhood
	$V\ni o$ and any hyperplane $H_0:=\{p=0\}\subset \C^n$ there is a zero-measure subset $E\subset \Delta$ such that
	\begin{align*}
	\underset{\Delta\setminus E\ni\xi\rightarrow0}{\liminf}~~ |\xi| \cdot\|F|_{V\cap H_\xi}\|=0,
	\end{align*}
	where $H_\xi:=\{p=\xi\}$ is a hyperplane in $\C^n$.	
\end{enumerate}
\end{lem}	

\begin{thm}[Strong openness]\label{soc}
  Let  $E$ be  a holomorphic vector bundle on $U$, and $h_1\geq h_2\geq\cdots$  be Nakano semi-positive singular Hermitian metrics on $E$. Suppose that $h:=\lim_j h_j$ is bounded below by a continuous Hermitian metric, then $\bigcup_j\cal E (h_j)=\cal E (h)$.
\end{thm}
\begin{proof}
We denote $M=\bigcup_j\cal E (h_j)$, and it is enough to show $M_x=\mathcal{E}(h)_x$ for arbitrary $x$, where we will take $x$ to be the origin $o$. We prove the theorem by induction on $m=\dim X$.
It  is obvious for  the case $m=0$.
Assume the statement is right for $m-1$, then consider the $m$-dimensional case.

To apply the induction hypothesis we have to verify whether the restrictions $E|_{H_\xi}$ to hyperplanes $H_\xi\subset\C^m$ satisfy the hypothesis of the theorem. Owing to Proposition \ref{pro nak res nak}, it is true for ``generic" $H$, which is enough.

Fix a relatively compact neighborhood $V_o\subset U$  and an $f\in\cal O(E)(V_o)$ such that $\int_{V_o} |f|^2_h<+\infty$.
We will show that $ f_o\in M_o$ by using the characterization in Lemma \ref{lem lempert}.
Take a neighborhood $V\Subset V_o$ of $o$ and a hyperplane $H_0$.
Again we assume $H_0=\{z\in\C^m: z_1=0\}$ and for $\xi\in\Delta$ write $H_{\xi}=\{z\in\C^m: z_1=\xi\}$.
The Fubini-Tonelli theorem guarantees that there is  $S\subset\Delta$ of full measure such that
$$\int_{V_o\cap H_{\xi}} |f|^2_h<+\infty\quad
 \text{ for }~ \xi\in S,$$
and
\begin{equation}\label{ccc}
\liminf_{S\ni \xi\to 0} |\xi|^2 \int_{V_o\cap H_{\xi}} |f|^2_h=0.
\end{equation}
The induction hypothesis implies that for each $\xi\in S$ there is a $j_\xi$ such that $\int_{V\cap H_{\xi}} |f|^2_{h_{j_\xi}}<+\infty$,
and by the monotone convergence theorem we have
$$\| f|_{V\cap H_{\xi}}\|^2=\int_{V\cap H_{\xi}} |f|^2_h.$$
 It follows from $(\ref{ccc})$ that
 \begin{equation}
\liminf_{S\ni \xi\to o} |\xi|\cdot\| f|_{V\cap H_{\xi}}\|=0.
\end{equation}
 Then by Lemma \ref{lem lempert}, $f_o\in M_o$ indeed.
\end{proof}

By Lemma~\ref{prop lowerbound}, Proposition~\ref{prop he^} and the strong Noetherian property of coherent analytic sheaves, for a Nakano semi-positive singular Hermitian metric $h$ on a holomorphic vector bundle $E$,
	we can show that the ``upper semicontinuous regularization'' of the sheaf $\cal{E}(h)$:
	$$\mathcal{E}_+(h)=\bigcup_{\beta >0} \mathcal{E}(h(\det h)^\beta)$$
	is a coherent submodule sheaf of $\cal{E}(h)$. Then there is a natural problem: Is $\cal{E}_+(h)$ equal to $\cal{E}(h)$?
	
Of course, Theorem \ref{soc} implies that the answer is positive.
In order to obtain strong openness on holomorphic vector bundles in more general cases, we need the following lemma.

\begin{lem}\label{lem p>1}
    Let $o\in D\subset \mathbb{C}^n$ be a Stein open set, $E$ be a trivial holomorphic vector bundle over $D$ with a  Nakano semi-positive singular Hermitian metric $h$.
    Assume $F$ is a holomorphic section.
    Then the following are equivalent:
    \begin{enumerate}
      \item[(1)] $|F|^2_{h(\det h)^\beta} $ is integrable near $o$ for some $\beta> 0;$
      \item[(2)] $ (|F|^2_h)^p$ is integrable near $o$ for some $p> 1$.
    \end{enumerate}
\end{lem}

\begin{proof}
    $(1)\Rightarrow(2)$:
        since $|F|^2_{\frac{h}{\det h}}$ is plurisubharmonic and hence locally bounded from above, then $$(|F|^2_h)^p=(|F|^2_{\frac{h}{\det h}})^pe^{-p\varphi}\le C^{p-1}|F|^2_{\frac{h}{\det h}}e^{-p\varphi}.$$
        Take $p=\beta+1$ and then $(|F|^2_h)^p$ is integrable near $o$.

        $ (2)\Rightarrow (1) $: Let $U$ be a neighborhood of $o$, by H\"older's inequality,
        $$\int_U |F|^2_{h(\det h)^\beta}=\int_U |F|^2_h e^{-\beta\varphi} \le (\int_U (|F|^2_h)^p)^{1/p} (\int_U e^{-q\beta\varphi})^{1/q},$$
        where $1/p+1/q=1$.
        Taking $\beta $ small enough such that $\nu(q\beta\varphi,o) =\beta \nu(q\varphi,o)< 2$ and shrinking $U$ if necessary, we have $\int_U e^{-q\beta\varphi}<+\infty$ by Lemma~\ref{lem Skoda}.
\end{proof}

	  Together with Lemma \ref{lem e^(phi-phi_j) bounded},  we can obtain the following result.
\begin{cor}\label{main thm2,usual case}
    Let $X$ be a complex manifold, $E$ be a holomorphic vector bundle on $X$ with a  Nakano semi-positive singular Hermitian metric $h$.
    Assume that $h_j$ is  Griffiths semi-positive on $E$. If  $h_j\ge h$ and $-\log \det h_j$ converges to $-\log\det h$ locally in measure, then  $\ \underset{j}{\bigcup} \mathcal{E}(h_j)=\mathcal{E}(h)$.
\end{cor}

\begin{proof}
        Only need to prove that $\mathcal{E}(h)_o \subset \mathcal{E}(h_j)_o$ for large $j$ for any $o \in X$ since $h_j\ge h$.  For any $F \in \mathcal{E}(h)_o$, choose a Stein neighborhood
        $D$ of $o$ such that $E$ is trivial on $D$.

        Let $\varphi=-\log\det h$ and $\varphi_j=-\log\det h_j $ and we may assume $\varphi< 0$ on $D$ by shrinking $D$ and adding a constant.
        Then for a small open subset $o$ $\in$ $U$ of $D$,   $h_j\ge h$ implies $\frac{h_j}{\det h_j}\le\frac{h}{\det h}$ by Lemma~\ref{lem h_j>h}, and we obtain
\begin{align*}
        \int_U |F|^2_{h_j}
        &= \int_U |F|^2_{\frac{h_j}{\det h_j}}e^{-\varphi_j} \\
        &\le \int_U |F|^2_{\frac{h}{\det h}}e^{-\varphi_j}\\
        &=\int_U |F|^2_h e^{\varphi-\varphi_j}\\
        &\le \left(\int_U (|F|^2_h)^p\right)^{1/p}\left(\int_U e^{q\varphi-q\varphi_j}\right)^{1/q},
        \end{align*}
        where $1/p+1/q=1$.

        By Theorem~\ref{main thm1, effective }  there exists $\beta>0$ such that $F\in \mathcal{E}(h(\det h)^\beta)_o$.
        By Lemma~\ref{lem p>1}  we can choose $p> 1$ such that $(|F|^2_h)^p$ is integrable near $o$.
        Since $\varphi_j\le \varphi$ converges to $\varphi$ locally in measure, shrinking $U$ if necessary we can choose $j_0$ such that $e^{q\varphi-q\varphi_j}$ is integrable on $U$ for $j\ge j_0$ by
        Lemma \ref{lem e^(phi-phi_j) bounded}. \\
    \end{proof}

\begin{rem}
    By {\rm Remark~\ref{rmk ge lempert}}, that $\varphi_j$ converges to $\varphi$ locally in measure implies that $\varphi_j$ converges to $\varphi$ in $L^{1}_{\rm loc}$, since $\varphi_j\le \varphi$.
\end{rem}

    It follows from the above theorem that we have the following consequence.
\begin{cor}\label{cor he^{-beta varphi}}
    Let $X$ be a complex manifold, $E$ be a holomorphic vector bundle on $X$ with a  Nakano semi-positive singular Hermitian metric $h$. Let $\psi\in {\rm Psh}(X)$.
    Assume that  $\psi_j\in {\rm Psh}(X)$ such that $\psi_j\le\psi$ and $\psi_j$ converges to $\psi$ locally in measure, then $ \underset{j}{\bigcup} \mathcal{E}(he^{-\psi_j})=\mathcal{E}(he^{-\psi})$.
    In particular, $ \underset{\beta>0}{\bigcup} \mathcal{E}(he^{-\beta\psi})=\mathcal{E}(h)$.
\end{cor}

\subsection{Multiplier ideal associated to $(E,h)$  }
 It is well known that the multiplier ideal sheaf $\mathcal{I}(\varphi)$ associated to a plurisubharmonic function $\varphi$ is integrally closed.  In this section, we give an analogous result for the vector bundle case.
  \begin{defi}[\cite{Cataldo98}]
   To a singular Hermitian metric $h$ on a holomorphic vector bundle $E$, one can associate a multiplier ideal  $\mathcal{I}(h) \subset\cal{O}$ defined as follows:
	$$\mathcal{I}(h)_x :=\{ f_x\in \cal{O}_{x}:f_x\cdot\cal{O}(E)_x \subset \cal{E}(h)_x \}.$$
  \end{defi}
  Obviously, the triangle inequality implies that $\mathcal{I}(h)\otimes \cal O(E)\subset \cal E(h)$. Moreover,
we can see that $\mathcal{I}(h)$ is equal to the annihilator of the quotient sheaf $\mathcal{O}(E)/\mathcal{E}(h)$.

 Recall that for any analytic sheaf $\mathcal{F}$, the annihilator of $\mathcal F$, denoted by  ${\rm Ann}\mathcal{F}$ is an ideal sheaf defined as follows:
 $${\rm Ann} \mathcal{F}_x :=\{f_x\in \mathcal{O}_x : f_x\mathcal{F}_x=0\}.$$
In addition, if $\mathcal{F}$ is coherent, then ${\rm Ann}\mathcal{F}$ is also coherent (see \cite{Grauert2012}).

  Then $\mathcal{I}(h)={\rm Ann}(\mathcal{O}(E)/\mathcal{E}(h))$ is a coherent ideal sheaf if $h$ is a Nakano semi-positive singular Hermitian metric .

\begin{pro}
    Let $E$ be a holomorphic vector bundle on $U$, and $h_1\geq h_2\geq\cdots$  be Nakano semi-positive singular Hermitian metrics on $E$. Suppose that $h=\lim_j h_j$ is bounded below by a continuous Hermitian metric, then
$\bigcup_j\mathcal I (h_j)=\mathcal I(h)$.
\end{pro}
\begin{proof}
    By theorem \ref{soc} and the Strong Noetherian property of coherent sheaves, there exists $j_0$, such that $\mathcal E(h_j)=\mathcal E(h)$ for all $j\geq j_0$, therefore $\mathcal{I}(h_j)=\mathcal{I}(h)$ for all $j\geq j_0$, so  we have $\bigcup_j\mathcal I (h_j)=\mathcal I(h)$.
\end{proof}

\begin{defi}
 Let $A$ be integrally closed domain and $I\subset A$ be an ideal.
the integral closure $\bar{I}$ of $I$ is defined as the set of $f\in A$ satisfying a monic polynomial equation $$f^d+\sum^{d}_{i=1}a_if^{d-i}=0,$$
where $a^i\in I^i$. It is easy to prove $\overline{I}$ is also an ideal of $A$, and we say that $I$ is integrally closed if $I=\bar{I}$.
\end{defi}
  Now consider that $\cal I\subset\mathcal{O}$ is a coherent ideal sheaf on a complex manifold $X$. The integral closure of $\cal I$ is the ideal sheaf $\bar{\cal I}\subset \mathcal O$ defined by taking the integral closure stalkwise. In fact $\bar{\cal I}$ is also a coherent ideal sheaf on $X$, and there is a characterization of integral closure of coherent analytic ideal sheaves proved by Skoda and Brian?on.
 \begin{pro}[\cite{Skoda1974}]\label{integrally lemma}
     Let $\cal I$ be a coherent analytic ideal of $\mathcal{O}$ with local generators $(f_i).$ Then $f\in \mathcal{O}$ belongs to $\bar{\cal I}$ if and only if $|f|= O(\max_i|f_i|).$
 \end{pro}

 \begin{pro}
    If $(E,h)$ is locally Nakano semi-positive singular Hermitian metric, then the  multiplier ideal $\mathcal{I}(h)$ is integrally closed.
\end{pro}
\begin{proof}
    Assume that $(f_i)$ be a set of local generators of $\mathcal{I}(h)_x$, by definition, we have $|f_i e|^2_h\in L^1_{loc}$ near $x$ for any $i$ and every $e_x\in \mathcal{O}(E)_x$ . If $f\in \mathcal{O}_{x}$ is contained in the integrally closure of $\mathcal{I}(h)_x$, then by Proposition \ref{integrally lemma}, $|f|= O(\max_i|f_i|)$. Hence we get $|fe|^2_h\in L^1_{loc}$ near $x$ for every $e_x\in \mathcal{O}(E)_x$,
    i.e, $f\in \mathcal{I}(h)_x$.
\end{proof}

\section{Effectiveness result of strong openness}

	In this section,  we will approach the strong openness property on holomorphic vector bundles in a different way by establishing an effectiveness result (Theorem \ref{main thm1, effective }).
	The idea of the proof of the effectiveness result comes from Guan\cite{Guan19} and Guan-Zhou\cite{GZ15}.  In concrete,
	we will prove an extension theorem (Theorem \ref{main thm 1}) to obtain an estimate (\ref{derivative of G}) of the derivative of the integrals of extremal functions on the open set $\{ \varphi <-t\}$ where $\varphi=-\log\det h$.
	
	We divide the proof of the effectiveness result of strong openness into three steps.
\vskip0.3cm
\noindent {\bf Step One: The property of ${C}_{F,\beta}(U)$.}
\vskip0.5cm

	Roughly, we prove that there is an extremal holomorphic section $F_0\in \mathcal{O}(E)(U)$ such that $ {C}_{F,\beta}(U)=\int_U |F_0|_{\frac{h}{\det h}}$.
	The statement of the following lemma is so complicated because one can not expect $F_0$ is defined on $D$.
\begin{lem}\label{lem keda}
	Let $E$ be a trivial holomorphic vector bundle on $D$.
	Let $U\ni o$ be an open subset of $D$.
	Assume that $F\in \cal{O}(E)(D)$.
    If $\ {C}_{F,\beta}(U)$ appearing in {\rm Definition \ref{def C_F}} is finite, then
    \setlength{\parskip}{0pt}
\begin{enumerate}
	\setlength{\itemsep}{0pt}
	\setlength{\parskip}{0pt}
	
	\item[\rm {(1)}]
	we can choose a minimizing sequence   $\{f_j\}\subset \mathcal{O}(E)(D)$ compactly converging to an $F_0\in \cal{O}(E)(U)$ on $U$ such that $(f_j-F, o)\in \mathcal{E} (h(\det h)^\beta)_o $,
	${C}_{F,\beta}(U)\le\int_{U}|f_j|^2_{\frac{h}{\det h}} \le {C}_{F,\beta}(U)+\frac{1}{j}$ and $\lim_j \int_U |f_j-F_0|^2_{\frac{h}{\det h}}=0$.
	
	\item [\rm {(2)}] Such $F_0$ satisfies that $ (F_0-F, o)\in \mathcal{E} (h(\det h)^\beta)_o$ and $ \int_{U}|F_0|^2_{\frac{h}{\det h}} $ $= {C}_{F,\beta}(U)$.
	Moreover, $F_0\in \mathcal{O}(E)(U)$ is unique in the sense that for any minimizing sequence $\{\widetilde{f}_j\}$ compactly converging to some $\widetilde{F}_0\in \cal{O}(E)(U)$, $\widetilde{F}_0=F_0$.
	
	\item  [\rm {(3)}]
	For any $\hat{F}\in \mathcal{O}(E)(D)$ such that $(\hat{F}-F, o)\in \mathcal{E}(h(\det h)^\beta)_o$ and $\int_{U} |\hat{F}|^2_{\frac{h}{\det h}} <+\infty$, we have
\begin{equation*}
	\int_{U} |F_0|^2_{\frac{h}{\det h}}+
	\int_{U} |\hat{F}-F_0|^2_{\frac{h}{\det h}}=
	\int_{U} |\hat{F}|^2_{\frac{h}{\det h}}.
	\end{equation*}
\end{enumerate}
\end{lem}
\begin{proof}
	
	Existence of $F_0$.
	As $C_{F,\beta}(U)< +\infty$, there exists  $   F_j \in \mathcal{O}(E)(D)$ such that $(F_j-F, o)\in \mathcal{E} (h(\det h)^\beta)_o$ and
	${C}_{F,\beta}(U)\le\int_{U}|F_j|^2_{\frac{h}{\det h}} \le {C}_{F,\beta}(U)+\frac{1}{j}$ for each $j>0$.
	Hence there is a subsequence of $\{F_j\}$ compactly converging to some $F_0 \in \mathcal{O}(E)(U)$ on $U$ and weakly converging to $F_0$ in $\mathcal{H}_\beta(U,o)$ via Lemma \ref{lem int of |f|^2e^phi dominates |f|} and the Banach-Alaoglu-Bourbaki theorem.
	
	Moreover, by Mazur's theorem(Lemma~\ref{Mazur' theorem} below), there exists a sequence $\{f_{v}\}^\infty_{v=1}$ such that $\lim_v\int_{U} |f_{v}-F_0|^2_{\frac{h}{\det h}}=0$ where each $f_{v}$ is a convex combination of
	$\{F_j \}_{j\ge v}$. So $ f_{v} \in \mathcal{O}(E)(D)$, $(f_{v}-F, o)\in \mathcal{E} (h(\det h)^\beta)_o$ and ${C}_{F,\beta}(U)\le\int_{U}|f_{v}|^2_{\frac{h}{\det h}} \le {C}_{F,\beta}(U)+\frac{1}{v}$.
	
	Clearly, $ \int_{U}|F_0|^2_{\frac{h}{\det h}} $ $= {C}_{F,\beta}(U)$.
	By  the closeness of the sections of coherent analytic sheaves under the  topology of compact convergence, we have $ (F_0-F, o)\in \mathcal{E} (h(\det h)^\beta)_o$.
	
	Uniqueness of $F_0$.
	By assumption,  $\frac{f_{j}+\widetilde{f}_{j}}{2}\in  \mathcal{O}(E)(D)$ and $(\frac{f_{j}+\widetilde{f}_{
			j}}{2}-F, o)\in \mathcal{E} (\det h)^\beta)_o$.
	It follows from Definition~\ref{def C_F} that
	$$\int_{U}\left|\frac{f_{j}+\widetilde{f}_{j}}{2}\right|^2_{\frac{h}{\det h}} \ge {C}_{F,\beta}(U).$$
	Let $j\to +\infty$, then
	$$\int_{U}\left|\frac{F_0+\widetilde{F}_0}{2}\right|^2_{\frac{h}{\det h}} \ge {C}_{F,\beta}(U).$$
	Conversely,
	$$\int_{U}\left|\frac{F_0+\widetilde{F}_0}{2}\right|^2_{\frac{h}{\det h}}
	+\int_{U}\left|\frac{F_0-\widetilde{F}_0}{2}\right|^2_{\frac{h}{\det h}}
	=\frac{1}{2}\left(\int_{U}|F_0|^2_{\frac{h}{\det h}}
	+\int_{U}|\widetilde{F}_0|^2_{\frac{h}{\det h}}\right)
	={C}_{F,\beta}(U),$$
	we get  $\int_{U}|\frac{F_0-\widetilde{F}_0}{2}|^2_{\frac{h}{\det h}}=0$ and hence $F_0=\widetilde{F}_0$.
	
	$(3)$ For any $\hat{F}\in \mathcal{O}(E)(D)$ such that $(\hat{F}-F, o)\in \mathcal{E}(h(\det h)^\beta)_o$ and $\int_{U} |\hat{F}|^2_{\frac{h}{\det h}} <+\infty$, we have
	$f_j+\alpha (\hat{F}-f_j)\in \mathcal{O}(E)(D)$ and $(f_j+\alpha (\hat{F}-f_j)-F, o)\in \mathcal{E}(h(\det h)^\beta)_o$ ~for $\alpha\in \R$. Noticing that for $\alpha \in \mathbb{R}$,
\begin{align*}
	C_{F,\beta}(U)
	&\le \int_U |f_j+\alpha (\hat{F}-f_j)|^2_{\frac{h}{\det h}}\\
	&\le C_{F,\beta}(U)+\frac{1}{j}
	+2\alpha \Re\int_U \langle f_j,\hat{F}-f_j\rangle_{\frac{h}{\det h}}
	+\alpha^2\int_U | \hat{F}-f_j|^2_{\frac{h}{\det h}}
	\end{align*}
	and $\lim_j \int_U |f_j-F_0|^2_{\frac{h}{\det h}}=0$, we obtain that  for any $\alpha \in \mathbb{R}$
\begin{align*}
	C_{F,\beta}(U)\le C_{F,\beta}(U)
	+2\alpha \Re\int_U \langle F_0,\hat{F}-F_0\rangle_{\frac{h}{\det h}}
	+\alpha^2\int_U | \hat{F}-F_0|^2_{\frac{h}{\det h}},
	\end{align*}
	which implies that $ \Re\int_U \langle F_0,\hat{F}-F_0\rangle_{\frac{h}{\det h}}=0$.
	Hence
\begin{equation*}
	\int_{U} |F_0|^2_{\frac{h}{\det h}}+
	\int_{U} |\hat{F}-F_0|^2_{\frac{h}{\det h}}=
	\int_{U} |\hat{F}|^2_{\frac{h}{\det h}}.
	\end{equation*}
\end{proof}

\begin{lem}[{\bf Mazur's Theorem}, \cite{Conway2019course}]\label{Mazur' theorem}
    Let $X$ be a normed space. If $\{x_j\}$ weakly converges to $x$, then there is a sequence $\{x'_j \} \subset X$ such that  $ \| x'_j-x\|\rightarrow 0 $, where each $ x'_j $ is a convex combination of $ \{ x_v \} $.
\end{lem}

\begin{rem}
    In fact, we can choose the sequence $\{x'_j\}$ such that $x'_j$ is a convex combination of $\{x'_v\}_{v\ge j}$. Since $\{x_j\}$ also weakly converges to $x$, $\{x_j\}_{j\ge j_0}$ weakly converges to $x$
    for any fixed $j_0\ge 1$. Hence there exists a convex combination $x'_{j_0}$ of $\{x_j\}_{j\ge j_0}$ such that $\| x'_{j_0}-x\|\le 1/j_0$. Then the sequence $\{x'_{j_0}\}_{j_0\ge1}$ is what we expect.
\end{rem}

	The following result is a consequence of Lemma~\ref{lem keda}.
\begin{lem}\label{lem 0}
    Let $o\in U \subset D$ and let $F\in \cal{O}(E)(D)$. Then  $(F, o) \in\mathcal{E} (h(\det h)^\beta)_o $ ~if and only if
    $\ {C}_{F,\beta}(U) = 0$.
\end{lem}

    Now we fix a singular Hermitian metric $h$, and let $\varphi=-\log\det h$.
    Let $F$ be a holomorphic section over such open set $D$.
    Consider the volume of the sublevel sets of $\varphi$ with a multiplier
    $$G_\beta(t):=C_{F,\beta} (\{\varphi< -t\}),$$ where $\{\varphi< -t\} $ means $\{\varphi< -t\}\cap D $. Notice that
    $\underset{t> t_0}{\bigcup} \{\varphi< -t\}= \{\varphi< -t_0\}$ for $t_0\in [0,+\infty)$, we obtain the following lemmas:
\begin{lem}\label{lem prop of G}
    Assume $G_\beta(0)< +\infty$, then $G_\beta(t)$ is decreasing with respect to $t\in [0,+\infty)$ and  $\underset{t\rightarrow +\infty}{\lim} G_\beta(t)=0$.
\end{lem}

\begin{lem}\label{lem estimate of G}
    Let $0< G < +\infty$. For $t\in [0, +\infty)$ such that $ G_\beta(t)< G$, we have
\begin{align*}
    \int_{t}^{+\infty} \frac{\liminf\limits_{B\rightarrow 0^+}  {\left(-\frac{G_\beta(s+B)-G_\beta(s)}{B}\right)}}{G-G_\beta(s)}ds
    \le \log \frac{G}{G-G_\beta(t)} .
    \end{align*}
\end{lem}

\begin{proof}
        Because $G_\beta(s)$ and $\log \frac{1}{G-G_\beta(s)}$ are decreasing on $[t,+\infty )$, they are both differentiable outside a zero-measure subset $E_0$ of $[t,+\infty)$.
        For any $s \in [t,+\infty) \setminus E_0 $, we have
\begin{align*}
        &\4l \underset{B\rightarrow 0^+}{\lim}\ \frac{1}{B} \left(\log \frac{1}{G-G_\beta(s)}-\log \frac{1}{G-G_\beta(s+B)}\right)\\
        &=\underset{B\rightarrow 0^+}{\lim}\ \frac{1}{B} \log\ \left(1+\frac{G_\beta(s)-G_\beta(s+B)}{G-G_\beta(s)}\right)\\
        &=\lim_{B\rightarrow 0^+}\ \frac{1}{B} \frac{G_\beta(s)-G_\beta(s+B)}{G-G_\beta(s)}(1+o(1))\\
        &=\liminf_{B\rightarrow 0^+}\ \frac{1}{B} \frac{G_\beta(s)-G_\beta(s+B)}{G-G_\beta(s)}.
        \end{align*}

        By Fatou's lemma and Lemma \ref{lem prop of G}, we get
\begin{align*}
        &\4l \int_{t}^{+\infty} \frac{\liminf\limits_{B\rightarrow 0^+}\left(-\frac{G_\beta(s+B)-G_\beta(s)}{B}\right)}{G-G_\beta(s)}ds\\
        &=\int_{t}^{+\infty} \underset{B\rightarrow 0^+}{\lim}\ \frac{1}{B} \left(\log \ \frac{1}{G-G_\beta(s)}-\log \ \frac{1}{G-G_\beta(s+B)}\right)ds\\
        &\le \left(\log \frac{1}{G-G_\beta(s)}\right)  \bigg|^t_{+\infty}\\
        &= \log \frac{G}{G-G_\beta(t)}.
        \end{align*}
    \end{proof}

\noindent {\bf Step Two: An extension theorem of openness type }
\vskip0.5cm

	We recall an openness type extension theorem appearing in \cite{GZ2015extension}, \cite{GZ15} and \cite{Guan19}.
\begin{theorem}
	Let $B\in (0,1]$ be arbitrarily given.
	Let $D$ be a pseudoconvex domain in $\C^n$ containing $o$.
	Let $\phi$ be a negative plurisubharmonic function on $D$, such that $\phi(o)=-\infty$.
	Let $F$ be an $L^2$ integrable holomorphic function on
	$\{\phi < - t_0\}$, $t_0 \ge 0$.

	Then there exists a holomorphic function $\tilde{F}$ on $D$,
	such that, $( \tilde{F}-F, o)\in \mathcal{I}(\phi)_o$ and
\begin{align*}
\int_D
|\tilde{F} - (1 - b_{t_0,B}(\phi))F|^2
\le \frac{C(t_0,B)}{B}
\int_{\{-t_0-B<\phi<-t_0\} }  |F|^2e^{-\phi},
\end{align*}
with the coefficient $C(t_0,B)=1 - e^{-(t_0+B)}$,
where $b_{t_0,B}(t) = \int_{-\infty}^t \frac{1}{B} \mathbbm{1}_{\{-t_0-B<s<-t_0\}}ds$.
\end{theorem}

	Roughly, one can reformulate the similar integral on the right hand as
\begin{align*}
	\frac{1}{B}
	\int_{\{-t_0-B\le\phi<-t_0\} }  |F|^2e^{-\phi}
	=
	\frac{
		\int_{\{\phi<-t_0\} }  |F|^2e^{-\phi}-\int_{\{\phi<-t_0-B\} }  |F|^2e^{-\phi}
		}{B}.
\end{align*}
	And let $B\to 0^+$, if the limit exists, then we obtain an estimate of the derivative of $\int_{\{\phi<-t\} }  |F|^2e^{-\phi}$ with respect to $t$.
	The main results in \cite{Guan19} and our main results are all based on such precise estimate of the derivative of similar integrals.

	Imitating the above result, we prove an openness extension theorem (Theorem \ref{main thm 1}) which plays a critical role in the proof of Theorem \ref{thm 1}.
	As you will see, the statement and the proof of Theorem \ref{main thm 1} are different from those of  Guan-Zhou's result.
	Guan-Zhou used more precise calculation on the Chern curvature.
	However, in \cite{Rau15} Raufi showed that the Chern curvature does NOT exist as positive measures for a Nakano semi-positive singular Hermitian metric in general.
	Fortunately, we obtain a weak result directly derived from the optimal $L^2$-estimate which happens to be useful.

\begin{thm}\label{main thm 1}
    Let $o\in$ $D$ $\subset \mathbb{C}^n$ be a Stein open set, $E$ be a trivial holomorphic vector bundle over $D$ with a  Nakano semi-positive singular Hermitian metric h.
    Assume $\varphi=-\log\det h < 0$ and $\varphi(o)=-\infty$.
    Then for given $t_0 > 0,\ B\in (0,1],~ \beta\ge 0$, and $F_{t_0}\in \mathcal{O}(E)(D)$ with  $\int_{\{\varphi< -t_0\}}  |F_{t_0}|^2_{\frac{h}{\det h}} dV < +\infty$,
    there exists $\widetilde{F}=\widetilde{F}_{t_0, B,\beta}\in \mathcal{O}(E)(D)$ such that $ (\widetilde{F}-F_{t_0}, o) \in \mathcal{E}(h(\det h)^\beta)_o$ and
\begin{align}\label{openness extension}
     \int_{D} |\widetilde{F}-(1-b_{t_0,B}(\varphi))F_{t_0}|^2_{h(\det h)^\beta}e^{-\chi_{t_0,B}(\varphi)}
    \le \frac{C(t_0,B,\beta)}{B} \int_{\{-t_0-B<\varphi<-t_0\}}  |F_{t_0}|^2_{\frac{h}{\det h}},
    \end{align}
    with the coefficient $C(t_0,B,\beta)=(t_0+B)e^{1+(1+\beta)(t_0+B)}$,
    where $\ \chi_{t_0,B}(t)=\frac{1}{t_0+B} \int_{0}^{t} b_{t_0,B}(s) ds$, $b_{t_0,B}(t)=\int_{-\infty}^{t}\frac{1}{B}\mathbbm{1}_{\{-t_0-B< s< -t_0 \}} ds$.
\end{thm}

	We  point out a difference that $F$ in Guan-Zhou's result is only defined on $\{\phi<-t\}$, while $F_{t_0}$ in ours is defined on the whole domain $D$ instead of $\{\varphi<-t_0\}$.
	That is why Definition \ref{def C_F}-(2) and Lemma \ref{lem keda} are more complicated than those of Guan-Zhou.
	We will explain the reason in the proof.

	One of the main points, not mentioned in the proof, is the choice of $\chi_{t_0,B}(t)$.
	We tried such functions of type $\ \chi_{t_0,B,\alpha}(t)=\alpha \int_{0}^{t} b_{t_0,B}(s) ds$, where $\alpha$ is a fixed positive constant.
	Theorem \ref{main thm 1} is also valid for $\chi_{t_0,B,\alpha}(t)$ with the coefficient $$C(t_0,B,\beta,\alpha)=\frac{1}{\alpha }e^{\alpha(t_0+B)+(1+\beta)(t_0+B)}.$$
	However, by computation, if we denote the function $\theta(\beta)$ appearing in Theorem \ref{main thm1, effective } associated with $\chi_{t_0,B,\alpha}(t)$ by $\theta_\alpha(\beta)$,
	then we have
	$$\lim_{\beta\to 0} \theta_\alpha(\beta)
	=1+\int_0^\infty \big(1-\exp(-\frac{\alpha}{1+\alpha}e^{-(1+\alpha)t})\big)e^tdt
	<+\infty$$
	for each fixed $\alpha>0$, which can not imply the strong openness property.
	In fact, the value $\alpha =(t_0+B)^{-1}$ exactly minimizes the coefficient $C(t_0,B,\beta,\alpha)$ for fixed $t_0,B,\beta$.

\begin{proof}
    	The proof is divided into three parts.
    	\vskip0.3cm
    	\noindent{\bf Part one: a modified version by using the optimal $L^2$ estimate.}
    	\vskip0.5cm
        Since $D$ is a Stein open set, we can take a sequence $\{D_j\}$ of Stein open subsets such that $D_j\Subset D_{j+1}$ and $\bigcup_j D_j =D$.
        For any $0< \varepsilon < \frac{1}{8}B$, let $0\le \mathbbm{1}_\varepsilon (s)\le 1$ be a smooth function on $\mathbb{R}$ such that
        $\Supp\mathbbm{1}_\varepsilon \subset [-t_0-B+\varepsilon, -t_0-\varepsilon]$ and $\mathbbm{1}_\varepsilon =1$ on $[-t_0-B+2\varepsilon, -t_0-2\varepsilon]$, then $\mathbbm{1}_\varepsilon$ converges to
        $\mathbbm{1}_{t_0,B}:=\mathbbm{1}_{\{-t_0-B< s< -t_0\}}$.

        Let $$b_{\varepsilon}(t)=\int_{-\infty}^{t}\frac{1}{B}\mathbbm{1}_\varepsilon (s)ds,~~~ \chi_{\varepsilon}(t)=\frac{1}{t_0+B} \int_{0}^{t} b_{\varepsilon}(s) ds.$$
        Then  $\chi_\varepsilon$ is a smooth convex increasing function on $\mathbb{R}$ and for $t\le 0$,
\begin{align}\label{ineq chi varepsilon}
        -1 \le \frac{1}{t_0+B} \max\{t,-t_0-B\}\le \chi_\varepsilon(t) \le \frac{1}{t_0+B} \max\{t,-t_0\} \le 0.
        \end{align}

        Let $0\le \rho\in C^\infty_0(\mathbb{C}^n) $ be a radial function such that $\int_{\mathbb{C}^n} \rho=1$  and  $\Supp \rho \subset \{|z|<1 \}$.
        Then $$\varphi_\delta(z)=\int_{\{|\xi|< 1\}} \varphi(z-\delta\xi)\rho(\xi)dV(\xi)$$
        is a well-defined smooth non-positive plurisubharmonic function on $D_j$ when $0< \delta < d(D_j, D^c)$.
        In addition, $\varphi_\delta$ decreasingly converges to $\varphi$ as $\delta \searrow 0$ on $D_j$.

        Let $\eta=dz_1\wedge dz_2\wedge\dots\wedge dz_n $ and $\omega= \sum \im dz_j\wedge d\overline{z}_j$, then $\overline{\partial}(1-b_\varepsilon(\varphi_\delta))\wedge F_{t_0}\eta$ is a smooth
        $\overline{\partial}$-closed $(n,1)$ form on $D_j$ with  $$\Supp\{b'_\varepsilon(\varphi_\delta) \} \subset \{-t_0-B+\varepsilon \le \varphi_\delta \le -t_0-\varepsilon\} \subset \{\varphi< -t_0\}.$$

        Since $h$ is  Nakano semi-positive on $D$ and so is $h(\det h)^\beta=he^{-\beta\varphi}$ by Proposition \ref{prop he^}, then there exists $(n,0)$-form $u_{\varepsilon,\delta,j}$ on $D_j$ such that
        $$\overline{\partial} u_{\varepsilon,\delta,j} = -\overline{\partial}(1-b_\varepsilon(\varphi_\delta))\wedge F_{t_0}\eta$$
        and
\begin{align}
        & \int_{D_j} |u_{\varepsilon,\delta,j}|^2_{h(\det h)^\beta}e^{-\chi_\varepsilon(\varphi_\delta)-\varepsilon |z|^2} \nonumber\\
        \le &\int_{D_j} \langle B^{-1}_{\varepsilon,\delta}(\overline{\partial}(1-b_\varepsilon(\varphi_\delta))\wedge F_{t_0}\eta), \overline{\partial}
        (1-b_\varepsilon(\varphi_\delta))\wedge F_{t_0}\eta\rangle_{\omega,h(\det h)^\beta}e^{-\chi_\varepsilon(\varphi_\delta)-\varepsilon|z|^2} \label{ineq main1 },
        \end{align}
        where $B_{\varepsilon,\delta}=B_{\omega,\chi_\varepsilon(\varphi_\delta)+\varepsilon|z|^2}$.
		Set $F_{\varepsilon,\delta,j}\eta:=u_{\varepsilon,\delta,j}+(1-b_\varepsilon (\varphi_\delta))F_{t_0}\eta$ on $D_j$, and then $F_{\varepsilon,\delta,j}$ is a modified openness extension.
		
        \vskip0.3cm
        \noindent{\bf Part two: Finiteness of the RHS of (\ref{ineq main1 }).}
        \vskip0.5cm

        We need to prove the integral (\ref{ineq main1 }) is finite.
        Note that
\begin{align*}
         B_{\varepsilon,\delta}
        &=B_{\omega,\chi_\varepsilon(\varphi_\delta)+\varepsilon|z|^2}\\
        &=[(\im\chi'_\varepsilon\partial\overline{\partial}\varphi_\delta
        +\im\chi''_\varepsilon\partial\varphi_\delta  \wedge \overline{\partial}\varphi_\delta
        +\varepsilon \im \partial\overline{\partial}|z|^2)\otimes {\rm Id}_E, \Lambda_\omega]\\
        &\ge[(\im \chi''_\varepsilon\partial\varphi_\delta  \wedge \overline{\partial}\varphi_\delta)\otimes {\rm Id}_E, \Lambda_\omega]\\
        &=\chi''_\varepsilon (\overline{\partial}\varphi_\delta)(\overline{\partial}\varphi_\delta)^* \otimes {\rm Id}_E
        \end{align*}
        is an operator on $(n,1)$-forms, where $(\overline{\partial}\varphi_\delta)^*$ is the adjoint operator of $(\overline{\partial}\varphi_\delta)\wedge\cdot$. For any $(n,0)$-form $u$ and $(n,1)$-form $v$,
        on $\Supp\chi''_\varepsilon$ we have
\begin{align*}
        |\langle\overline{\partial}\varphi_\delta \wedge u, v\rangle|^2
        &=|\langle u, (\overline{\partial}\varphi_\delta)^* v\rangle|^2\\
        &\le | u|^2|(\overline{\partial}\varphi_\delta)^* v|^2\\
        &= | u|^2\langle v, (\overline{\partial}\varphi_\delta)(\overline{\partial}\varphi_\delta)^* v\rangle\\
        &\le \frac{1}{\chi''_\varepsilon}| u|^2\langle B_{\varepsilon,\delta} v, v\rangle.
        \end{align*}

        Let $v=B^{-1}_{\varepsilon,\delta}(\overline{\partial}\varphi_\varepsilon\wedge u)$, we get
\begin{equation*}
       \langle B^{-1}_{\varepsilon,\delta}(\overline{\partial}\varphi_\delta\wedge u),(\overline{\partial}\varphi_\delta\wedge u)\rangle
        \le \frac{1}{\chi''_\varepsilon}|u|^2
        = \frac{t_0+B}{b'_\varepsilon} |u|^2
        \end{equation*}
        on $\Supp b'_\varepsilon$.
        Therefore by taking $u=F_{t_0}\eta$, we obtain
\begin{align}
        & \int_{D_j} \langle B^{-1}_{\varepsilon,\delta}(\overline{\partial}(1-b_\varepsilon(\varphi_\delta))\wedge F_{t_0}\eta), \overline{\partial}(1-b_\varepsilon(\varphi_\delta))
        \wedge F_{t_0}\eta\rangle_{(\omega,h(\det h)^\beta)}e^{-\chi_\varepsilon(\varphi_\delta)-\varepsilon|z|^2} \nonumber\\
        =~& \int_{D_j \cap \Supp b'_\varepsilon(\varphi_\delta)} \langle B^{-1}_{\varepsilon,\delta}(\overline{\partial}\varphi_\delta\wedge F_{t_0}\eta), \overline{\partial}\varphi_\delta
        \wedge F_{t_0}\eta\rangle_{(\omega,h(\det h)^\beta)}e^{-\chi_\varepsilon(\varphi_\delta)-\varepsilon|z|^2}|b'_\varepsilon (\varphi_\delta) |^2 \nonumber   \\
        \le~& \int_{D_j \cap \Supp b'_\varepsilon(\varphi_\delta)}  |F_{t_0}|^2_{h(\det h)^\beta}|\eta|^2_\omega e^{-\chi_\varepsilon(\varphi_\delta)}
        |b'_\varepsilon (\varphi_\delta) |(t_0+B) \nonumber\\
        =~& (t_0+B)\int_{D_j \cap \Supp b'_\varepsilon(\varphi_\delta)}  |F_{t_0}|^2_{h(\det h)^\beta} e^{-\chi_\varepsilon(\varphi_\delta)} |b'_\varepsilon (\varphi_\delta) | \nonumber\\
        \le~& \frac{t_0+B}{B}e
        \int_{D_j \cap \{-t_0-B+\varepsilon\le\varphi_\delta\le -t_0-\varepsilon\}}  |F_{t_0}|^2_{\frac{h}{\det h}} e^{-(1+\beta)\varphi},
        \label{ineq main2 }
        \end{align}
        where the last inequality comes from (\ref{ineq chi varepsilon}) and the fact $|b'_\varepsilon|\le 1/B$.
		
        Now we claim that for given $j, \varepsilon $,  the last integral (\ref{ineq main2 })is finite when $\delta$ is small enough.

        According to the assumption $\int_{\{\varphi< -t_0\}}  |F_{t_0}|^2_{\frac{h}{\det h}} <+\infty$,
        we know that $|F_{t_0}|^2_{\frac{h}{\det h}}$ is integrable on $\{-t_0-B+\varepsilon\le\varphi_\delta\le -t_0-\varepsilon\}\subset \{\varphi<-t_0 \}$,
        however, $e^{-(1+\beta)\varphi}$ may be not bounded on $\{-t_0-B+\varepsilon\le\varphi_\delta\le -t_0-\varepsilon\}$.
        So for the finiteness of the integral we have to add an extra assumption that $F_{t_0}$ is defined on $D$.
        And hence $F_{t_0}$ is bounded on $D_j$.
        Under such assumption, it is easy to see that $|F_{t_0}|^2_{\frac{h}{\det h}}$ is bounded, and we only need to prove $e^{-(1+\beta)\varphi}$ is integrable.

        Since $\mathcal{I}((1+\beta)\varphi)$ is coherent,
        $$A:=\Supp\mathcal{O}_D/\mathcal{I}((1+\beta)\varphi)=\{x \in D \big| e^{-(1+\beta)\varphi} \textrm{ is not integrable near}\ x \}$$
        is an analytic subset and hence closed in $D$.
        Noting that for given $j, \varepsilon $, $$ \{-t_0-B+\varepsilon\le\varphi_\delta \} \cap \overline{D_j \cap \{\varphi< -t_0\}} \cap A $$    is a decreasing sequence of compact subsets in $D$ as $\delta \searrow 0$ and
\begin{align*}
        & \underset{\delta > 0}{\bigcap}{ \{-t_0-B+\varepsilon\le\varphi_\delta\}}\cap \overline{D_j \cap \{\varphi< -t_0\}} \cap A \\
        \subset~& \{-t_0-B+\varepsilon\le\varphi\}\cap \overline{D_j \cap \{\varphi< -t_0\}} \cap \{\varphi=-\infty \} =\emptyset    ,
        \end{align*}
        we can find $0<\delta_j < d(D_j,D^c)$ such that
\begin{align*}
        { \{-t_0-B+\varepsilon\le\varphi_{\delta_j}\}}\cap \overline{D_j \cap \{\varphi< -t_0\}} \cap A     =\emptyset.
        \end{align*}
 Since  $ e^{-(1+\beta)\varphi}$ is locally integrable in $A^c$, we get that
        $e^{-(1+\beta)\varphi}$ is integrable on $$\{-t_0-B+\varepsilon\le\varphi_{\delta_j}\}\cap \overline{D_j \cap \{\varphi< -t_0\}}.$$

        In addition, $F_{t_0}$ is holomorphic on $D$ and by Lemma~\ref{prop lowerbound}, $\frac{h}{\det h}$ has an upper bound on $\overline{D}_j$, so  $|F_{t_0}|^2_{\frac{h}{\det h}}$ is also bounded on $D_j$.
        Therefore $|F_{t_0}|^2_{\frac{h}{\det h}} e^{-(1+\beta)\varphi}$ is integrable on $$\{-t_0-B+\varepsilon\le\varphi_{\delta_j}\}\cap \overline{D_j \cap \{\varphi< -t_0\}}.$$

        Thus for $0< \delta < \min \{\frac{1}{8}B, \delta_j\}$, we obtain {\small{
\begin{align}
        & \int_{D_j \cap \{-t_0-B+\varepsilon\le\varphi_\delta\le -t_0-\varepsilon\}}  |F_{t_0}|^2_{\frac{h}{\det h}} e^{-(1+\beta)\varphi}
        \le \int_{   \{-t_0-B+\varepsilon\le\varphi_{\delta_j}\} \cap \overline{D_j \cap \{\varphi< -t_0\}}}  |F_{t_0}|^2_{\frac{h}{\det h}} e^{-(1+\beta)\varphi} < +\infty .\label{ineq main3 }
        \end{align}}}

        By definition of Nakano semi-positivity of $h$, such $u_{\varepsilon,\delta,j}$ exists for $0< \delta < \min \{\frac{1}{8}B, \delta_j\}$.

        \vskip0.3cm
		\noindent{\bf Part three: take limits to get the target extension.}
		\vskip0.5cm
		
        Recall that $F_{\varepsilon,\delta,j}\eta=u_{\varepsilon,\delta,j}+(1-b_\varepsilon (\varphi_\delta))F_{t_0}\eta$ on $D_j$, then $$\overline{\partial}F_{\varepsilon,\delta,j}\wedge\eta
        =\overline{\partial}u_{\varepsilon,\delta,j}       +\overline{\partial}(1-b_\varepsilon (\varphi_\delta))\wedge F_{t_0}\eta=0,$$
        so $F_{\varepsilon,\delta,j}\in\mathcal{O}(E)({D_j})$.
        We will let $\delta\to 0^+, \varepsilon\to 0^+, j\to +\infty$ to get $\tilde{F}$ by standard limit arguments.

        Considering that $D_j\Subset D$, we have $|z|^2\le C_j$ on $D_j$ for some $C_j>0$. Using the assumption that $\varphi<0$ on $D$, i.e. $\det h >1 $ on $D$ and combining   (\ref{ineq main1 }),
        (\ref{ineq main2 }) and (\ref{ineq main3 }), we have
\begin{align*}
        \int_{D_j} |u_{\varepsilon,\delta,j}|^2_{\frac{h}{\det h}}
        &\le e^{\varepsilon C_j} \int_{D_j} |u_{\varepsilon,\delta,j}|^2_{h(\det h)^\beta}e^{-\chi_\varepsilon(\varphi_\delta)-\varepsilon |z|^2}\\
        &\le e^{1+\varepsilon C_j}  \frac{t_0+B}{B} \int_{   \{-t_0-B+\varepsilon\le\varphi_{\delta_j}\} \cap \overline{D_j \cap \{\varphi< -t_0\}}}  |F_{t_0}|^2_{\frac{h}{\det h}} e^{-(1+\beta)\varphi} < +\infty,
        \end{align*}
        and then  for any $0<\delta\le\min\{\frac{1}{8}B, \delta_j\}$, we obtain that {\small{
\begin{align*}
            &\left(\int_{D_j} |F_{\varepsilon,\delta,j}|^2_{\frac{h}{\det h}}\right)^{\frac{1}{2}}\\
            \le~& \left(\int_{D_j} |u_{\varepsilon,\delta,j}|^2_{\frac{h}{\det h}}\right)^{\frac{1}{2}}+
                \left(\int_{D_j}|(1-b_\varepsilon(\varphi_\delta))F_{t_0}|^2_{\frac{h}{\det h}}\right)^{\frac{1}{2}}\\
            \le~& \left(e^{ 1+\varepsilon C_j}  \frac{t_0+B}{B} \int_{   \{-t_0-B+\varepsilon\le\varphi_{\delta_j}\} \cap \overline{D_j \cap \{\varphi< -t_0\}}}  |F_{t_0}|^2_{\frac{h}{\det h}} e^{-(1+\beta)\varphi} \right)^{\frac{1}{2}}
           +\left(\int_{\{\varphi<-t_0 \}} |F_{t_0}|^2_{\frac{h}{\det h}}\right)^{\frac{1}{2}},
        \end{align*}}}
       where the last integrals are  finite and independent of $\delta$.

        Hence there exists a subsequence of $\{F_{\varepsilon,\delta, j}\}_\delta$ compactly converging to $F_{\varepsilon,j}\in \mathcal{O}(E)(D_j)$ on $D_j$ by $\mathrm{Lemma}~\ref{lem int of |f|^2e^phi dominates |f|} (2)$.
        By Fatou's lemma, the monotone convergence theorem and (\ref{ineq main3 }), we get that
\begin{align}
        &  \int_{D_j} |F_{\varepsilon,j}-(1-b_\varepsilon(\varphi))F_{t_0}|^2_{h(\det h)^\beta}e^{-\chi_\varepsilon(\varphi)-\varepsilon|z|^2}\nonumber\\
        =~& \int_{D_j} \liminf\limits_{\delta\rightarrow 0^+}|F_{\varepsilon,\delta,j}-(1-b_\varepsilon(\varphi_\delta))F_{t_0}|^2_{h(\det h)^\beta}e^{-\chi_{\varepsilon}
        (\varphi_{\delta})-{\varepsilon} |z|^2}\nonumber\\
        \le~& \liminf\limits_{\delta\rightarrow 0^+}  \int_{D_j}|F_{\varepsilon,\delta,j}-(1-b_\varepsilon(\varphi_\delta))F_{t_0}|^2_{h(\det h)^\beta}e^{-\chi_{\varepsilon}
        (\varphi_{\delta})-{\varepsilon} |z|^2}\nonumber\\
        =~&\liminf\limits_{\delta\rightarrow 0^+}
          \int_{D_j} |u_{{\varepsilon},\delta, j}|^2_{h(\det h)^\beta}e^{-\chi_{\varepsilon}(\varphi_{\delta})-{\varepsilon} |z|^2} \nonumber\\
        \le~& \liminf\limits_{\delta\rightarrow 0^+} \frac{t_0+B}{B}e
        \int_{D_j \cap \{-t_0-B+\varepsilon\le\varphi_\delta\le -t_0-\varepsilon\}}  |F_{t_0}|^2_{\frac{h}{\det h}} e^{-(1+\beta)\varphi-\varepsilon|z|^2} \nonumber\\
        \le~&
        \frac{t_0+B}{B}e^{1+(1+\beta)(t_0+B-\varepsilon)}
        \int_{D_j \cap \{-t_0-B+\varepsilon\le\varphi< -t_0-\varepsilon\}}  |F_{t_0}|^2_{\frac{h}{\det h}} .\label{for delta}
        \end{align}
        And then for $0<\varepsilon<\frac{1}{8}B$, we have{\small{
\begin{align*}
        & \left(\int_{D_j} |F_{\varepsilon,j}|^2_{\frac{h}{\det h}}\right)^{\frac{1}{2}}\\
        \le~& \left(\int_{D_j} |F_{\varepsilon,j}-(1-b_\varepsilon(\varphi))F_{t_0}|^2_{\frac{h}{\det h}}\right)^{\frac{1}{2}}
        +\left(\int_{D_j}|(1-b_\varepsilon(\varphi))F_{t_0}|^2_{\frac{h}{\det h}}\right)^{\frac{1}{2}}\\
        \le~& \left(\int_{D_j} |F_{\varepsilon,j}-(1-b_\varepsilon(\varphi))F_{t_0}|^2_{{h}{(\det h)}^\beta}e^{-\chi_{\varepsilon}(\varphi_{\delta})-{\varepsilon} |z|^2+\varepsilon C_j}\right)^{\frac{1}{2}}
        +\left(\int_{D_j}|(1-b_\varepsilon(\varphi))F_{t_0}|^2_{\frac{h}{\det h}}\right)^{\frac{1}{2}}\\
        \le~& \left(\frac{t_0+B}{B}e^{1+(1+\beta)(t_0+B)+\frac{1}{8}BC_j}
        \int_{\{ \varphi < -t_0\}}  |F_{t_0}|^2_{\frac{h}{\det h}}  \right)^{\frac{1}{2}}+\left(\int_{\{\varphi < -t_0\}}  |F_{t_0}|^2_{\frac{h}{\det h}}  \right)^{\frac{1}{2}}.
        \end{align*}}}

        Since the last integrals are  finite and independent of $\varepsilon$, it follows from  Lemma~\ref{lem int of |f|^2e^phi dominates |f|}-(3) that there exists a subsequence of $\{F_{\varepsilon,j}\}_\varepsilon$ compactly converging to $F_j\in  \mathcal{O}(E)(D_j)$ as $\varepsilon\rightarrow 0$. Then by Fatou's lemma and (\ref{for delta}) we get that
\begin{align}
        &  \int_{D_j} |F_{j}-(1-b_{t_0,B}(\varphi))F_{t_0}|^2_{h(\det h)^\beta}e^{-\chi_{t_0,B}(\varphi)}\nonumber\\
        =~& \int_{D_j} \liminf\limits_{\varepsilon\rightarrow 0^+}|F_{\varepsilon,j}-(1-b_\varepsilon(\varphi))F_{t_0}|^2_{h(\det h)^\beta}e^{-\chi_{\varepsilon}
        (\varphi)-{\varepsilon} |z|^2}\nonumber\\
        \le~& \liminf\limits_{\varepsilon\rightarrow 0^+}  \int_{D_j}|F_{\varepsilon,j}-(1-b_\varepsilon(\varphi))F_{t_0}|^2_{h(\det h)^\beta}e^{-\chi_{\varepsilon}
        (\varphi)-{\varepsilon} |z|^2}\nonumber\\
        \le~&
        \frac{t_0+B}{B}e^{1+(1+\beta)(t_0+B)}
        \int_{D_j \cap \{-t_0-B<\varphi< -t_0\}}  |F_{t_0}|^2_{\frac{h}{\det h}} .\label{for eps}
        \end{align}
        Hence for any $j$,
\begin{align*}
        & \left(\int_{D_j} |F_{j}|^2_{\frac{h}{\det h}}\right)^{\frac{1}{2}}\\
        \le~& \left(\int_{D_j} |F_{j}-(1-b_{t_0,B}(\varphi))F_{t_0}|^2_{\frac{h}{\det h}}\right)^{\frac{1}{2}}
        +\left(\int_{D_j}|(1-b_{t_0,B}(\varphi))F_{t_0}|^2_{\frac{h}{\det h}}\right)^{\frac{1}{2}}\\
        \le~& \left(\frac{t_0+B}{B}e^{1+(1+\beta)(t_0+B)}
        \int_{\{ \varphi < -t_0\}}  |F_{t_0}|^2_{\frac{h}{\det h}}  \right)^{\frac{1}{2}}+\left(\int_{\{\varphi < -t_0\}}  |F_{t_0}|^2_{\frac{h}{\det h}}  \right)^{\frac{1}{2}},
        \end{align*}
         where the last integrals are  finite and independent of $j$.

        Let  $j\rightarrow +\infty$,  then there exists a subsequence of $\{F_{j}\}$ compactly converging to
        $\widetilde{F}\in \mathcal{O}(E)(D)$ due to Lemma~\ref{lem int of |f|^2e^phi dominates |f|}-(3).
        By Fatou's lemma and (\ref{for eps}), we have
\begin{align*}
         \int_{D} |\widetilde{F}-(1-b_{t_0,B}(\varphi))F_{t_0}|^2_{h(\det h)^\beta}e^{-\chi_{t_0,B}(\varphi)}\le  \frac{t_0+B}{B}e^{1+(1+\beta)(t_0+B)} \int_{\{-t_0-B<\varphi< -t_0\}}  |F_{t_0}|^2_{\frac{h}{\det h}}.
        \end{align*}

        Since $\varphi(o)=-\infty$, $\{\varphi< -t_0-B-1\}$ is a nonempty neighborhood of $o$.
        Noting that $b_{t_0,B}(\varphi)=0$ on $\{\varphi< -t_0-B-1\}$ and
         $$\int_{\{\varphi< -t_0-B-1\}} |\widetilde{F}-F_{t_0}|^2_{h(\det h)^\beta}e^{-\chi_{t_0,B}(\varphi)}  < +\infty,$$
        we obtain that $(\widetilde{F}-F_{t_0}, o) \in \mathcal{E}(h(\det h)^\beta)_o$.
    \end{proof}
\vskip0.3cm
\noindent{\bf Step Three: Effectiveness result of strong openness}
\vskip0.5cm
	We present the following lower bound of $G_\beta(t)$, which may not be sharp but is enough to prove Theorem \ref{thm 1}.

\begin{pro}\label{prop convex estimate}
    Let $o\in$ $D$ $\subset \mathbb{C}^n$ be a Stein open set, $E$ be a trivial holomorphic vector bundle over $D$ with a Nakano semi-positive singular Hermitian metric h. Assume $\varphi=-\log \det h < 0$ and
    $\varphi(o)=-\infty$. If $\int_D |F|^2_h < +\infty$, then we have
\begin{align*}
    G_\beta(t)\ge  G_\beta(0) (1-e^{-g_\beta(t)})
    \end{align*}
    where
\begin{equation*}
    g_\beta(t):=\int_{t}^{+\infty} \frac{ds}{se^{1+(1+\beta)s}}.
    \end{equation*}
\end{pro}

\begin{proof}
    	Recall that $G_\beta(t):=C_{F,\beta} (\{\varphi< -t\})$.
    	The reader can go back to Definition \ref{def C_F} for the definition of $C_{F,\beta}$.
    	
    	By assumption,
        $${C}_{F,\beta}(D) = G_\beta(0)\le \int_D |F|^2_{\frac{h}{\det h}} \le\int_D |F|^2_h < +\infty.$$

        By Lemma \ref{lem keda}, for any fixed $t\ge 0$, there exists an  $  F_{t} \in \mathcal{O}(E)(\{\varphi< -t \} )$ such that
        $ (F_{t}-F, o)\in \mathcal{E} (h(\det h)^\beta)_o  $
        and
\begin{align}\label{ineq keda}
        \int_{\{\varphi < -t\}}|F_{t}|^2_{\frac{h}{\det h}} = G_\beta(t),
        \end{align}
        and  there exists $  F_{t,j} \in \mathcal{O}(E)(D)$ compactly converging to $F_t$ on $\{\varphi<-t\}$ such that $ (F_{t,j}-F, o)\in \mathcal{E} (h(\det h)^\beta)_o$,
\begin{align}
        &\lim_{j\to +\infty}\int_{\{\varphi<-t\} } |F_{t,j}-F_t|^2_{\frac{h}{\det h}}=0\label{equ**1}.
        \end{align}

        By Theorem \ref{main thm 1}, there is $\widetilde{F}=\widetilde{F}_{t,j,B,\beta}\in \mathcal{O}(E)(D)$ such that $(\widetilde{F}-F_{t,j}, o) \in \mathcal{E}(h(\det h)^\beta)_o$ and
\begin{align}\label{*}
        \4l \int_{D} |\widetilde{F}-(1-b_{t,B}(\varphi))F_{t,j}|^2_{h(\det h)^\beta}e^{-\chi_{t,B}(\varphi)}
        \le  \frac{t+B}{B}e^{1+(1+\beta)(t+B)} \int_{\{-t-B\le\varphi< -t\}}  |F_{t,j}|^2_{\frac{h}{\det h}}.
        \end{align}
		We will take limits by letting $j\to +\infty, B\to 0^+$.

        For the right hand side of (\ref{*}), using
        $\int_{\{\varphi<-t-B\} } |F_{t,j}|^2_{\frac{h}{\det h}}\ge G_\beta(t+B)$ and (\ref{equ**1}),
        we have
\begin{align}\label{ineq 4*}
        &\liminf_{B\to 0^+}\lim_{j\to+\infty}\frac{1}{B}
        \int_{\{-t-B\le\varphi< -t\}}  |F_{t,j}|^2_{\frac{h}{\det h}} \nonumber\\
        \le &\liminf_{B\to 0^+}\lim_{j\to+\infty}\frac{1}{B}
        \big(\int_{\{\varphi< -t\}}  |F_{t,j}|^2_{\frac{h}{\det h}} -G_\beta(t+B)\big) \nonumber \\
        = &\liminf_{B\to 0^+}\frac{1}{B}
        \big(\int_{\{\varphi< -t\}}  |F_{t}|^2_{\frac{h}{\det h}} -G_\beta(t+B)\big) \nonumber \\
        = &\liminf_{B\to 0^+} \frac{G_\beta(t)-G_\beta(t+B)}{B}.
        \end{align}

        For the left hand side of (\ref{*}), by Lemma~\ref{lem keda}-(3),
\begin{align}\label{equ *}
        \int_{\{\varphi < -t\}} |\widetilde{F}|^2_{\frac{h}{\det h}}=\int_{\{\varphi < -t\}} |F_t|^2_{\frac{h}{\det h}}+
        \int_{\{\varphi < -t\}} |\widetilde{F}-F_t|^2_{\frac{h}{\det h}}.
        \end{align}

        And using the assumption $\det h>1$, the fact $\chi_{t,B}(s)<0$ for $s<0$ and (\ref{equ *}), we have {\small{
\begin{align}\label{ineq big1}
        &\4l \int_{D} |\widetilde{F}-(1-b_{t,B}(\varphi))F_{t,j}|^2_{h(\det h)^\beta}e^{-\chi_{t,B}(\varphi)}  \nonumber\\
        &\ge \int_{D} |\widetilde{F}-(1-b_{t,B}(\varphi))F_{t,j}|^2_{\frac{h}{\det h}} \nonumber\\
        &= \int_{\{\varphi\ge-t\}} |\widetilde{F}|^2_{\frac{h}{\det h}} +\int_{\{\varphi<-t\}} |\widetilde{F}-F_{t,j}+b_{t,B}(\varphi)F_{t,j}|^2_{\frac{h}{\det h}} \nonumber\\
        &= \int_{D} |\widetilde{F}|^2_{\frac{h}{\det h}} -\int_{\{\varphi<-t\}} |\widetilde{F}|^2_{\frac{h}{\det h}} +\int_{\{\varphi<-t\}} |\widetilde{F}-F_{t,j}+b_{t,B}
        (\varphi)F_{t,j}|^2_{\frac{h}{\det h}} \nonumber\\
        &= \int_{D} |\widetilde{F}|^2_{\frac{h}{\det h}}  -\left(\int_{\{\varphi<-t\}} |F_{t}|^2_{\frac{h}{\det h}}+\int_{\{\varphi<-t\}} |\widetilde{F}-F_{t}|^2_{\frac{h}{\det h}}
        \right)
        +\int_{\{\varphi<-t\}} |\widetilde{F}-F_{t,j}+b_{t,B}(\varphi)F_{t,j}|^2_{\frac{h}{\det h}} \nonumber\\
        &\ge G_\beta(0)-G_\beta(t)-\int_{\{\varphi<-t\}} |\widetilde{F}-F_{t}|^2_{\frac{h}{\det h}} +\int_{\{\varphi<-t\}} |\widetilde{F}-F_{t,j}+b_{t,B}(\varphi)F_{t,j}|^2_{\frac{h}{\det h}}.
        \end{align}
		}}
		For the last two integrals of (\ref{ineq big1}), by (\ref{equ**1})
\begin{align}\label{ineq 0}
		&\lim_{B\to 0^+}\lim_{j\to +\infty}
		\big|\big(\int_{\{\varphi<-t\}} |\widetilde{F}-F_{t}|^2_{\frac{h}{\det h}}\big)^{1/2} -\big(\int_{\{\varphi<-t\}} |\widetilde{F}-F_{t,j}+b_{t,B}(\varphi)F_{t,j}|^2_{\frac{h}{\det h}}\big)^{1/2}\big|   \nonumber\\
		\le &\lim_{B\to 0^+}\lim_{j\to +\infty}
		\big(\int_{\{\varphi<-t\}} |F_{t}-F_{t,j}+b_{t,B}(\varphi)F_{t,j}|^2_{\frac{h}{\det h}}\big)^{1/2} \nonumber\\
		=&\lim_{B\to 0^+}
		\big(\int_{\{\varphi<-t\}} |b_{t,B}(\varphi)F_{t}|^2_{\frac{h}{\det h}}\big)^{1/2}=0.
		\end{align}

        Combining (\ref{*}), (\ref{ineq 4*}), (\ref{ineq big1}) and (\ref{ineq 0}), we obtain that
\begin{align}\label{derivative of G}
        G_\beta(0)-G_\beta(t)\le te^{1+(1+\beta)t} \underset{B\rightarrow 0^+}{\underline{\lim}}\frac{ G_\beta(t)-G_\beta(t+B)}{B}.
        \end{align}

        Applying Lemma \ref{lem estimate of G} to $ G= G_\beta(0)$, then for $t \ge 0$ such that  $G_\beta(t)< G_\beta(0)$
\begin{align*}
        \log \frac{G_\beta(0)}{G_\beta(0)-G_\beta(t)}
        &\ge \int_{t}^{+\infty} \frac{{\underset{B\rightarrow 0^+}{\underline{\lim}}}  {\left(-\frac{G_\beta(s+B)-G_\beta(s)}{B}\right)}}{G_\beta(0)-G_\beta(s)}ds\\
        &\ge  \int_{t}^{+\infty} \frac{ds}{se^{1+(1+\beta)s}}=g_\beta(t),
        \end{align*}
        that is,
\begin{align*}
        G_\beta(t)\ge  G_\beta(0) (1-e^{-g_\beta(t)}).
        \end{align*}
        If $t$ satisfies $G_\beta(t)\ge  G_\beta(0) $, the result is obvious.
    \end{proof}

In the end, we obtain  an  effectiveness result on the strong openness of the multiplier submodule sheaves for Nakano semi-positive singular Hermitian metrics on holomorphic vector bundles.

\begin{thm}\label{main thm1, effective }
    Let $o\in D\subset \mathbb{C}^n$ be a Stein open set, $E$ be a trivial holomorphic vector bundle over $D$ with a  Nakano semi-positive singular Hermitian metric h.
    Assume $\varphi=-\log \det h < 0$ on  $D$ and $\varphi(o)=-\infty$.
    Then for any $F\in \mathcal{E}(h)(D)$ satisfying  $\int_D |F|^2_h < +\infty$, we have $F\in \mathcal{E}(h(\det h)^\beta)_o$ if
\begin{equation}
    \theta(\beta) > \frac{\int_D |F|^2_h}{{C}_{F,c^F_o(h)^+}(D)},
    \end{equation}
    where
\begin{equation}
    \theta(\beta) = 1+ \int_{0}^{+\infty} (1-e^{-g_\beta(t)}) e^t dt=1+ \frac{1}{1+\beta}\int_{0}^{+\infty} (1-e^{-g_0(t)}) e^{\frac{t}{1+\beta}} dt.
    \end{equation}
    In particular, $\mathcal{E}_+(h)=\mathcal{E}(h)$.
\end{thm}

\begin{proof}
	We denote $c^F_o(h)$ by $c_o$. If $c_o=+\infty$, there is nothing to prove. We only need to consider the case that $c_o<+\infty$.
	Notice that $\int_D |F|^2_{\frac{h}{\det h}}\le\int_D |F|^2_h < +\infty $, so $C_{F,c_o^+}(D)$ and $C_{F,\beta}(D)$ are well-defined and finite.
	Firstly, we need to prove that ${C}_{F,c_o^+}(D) > 0$.
	
	By the strong Noetherian property of coherent sheaves of modules, $$\underset{\beta > c_o}{\bigcup}
	{\mathcal{E}(h(\det h)^\beta)_o} = \mathcal{E}(h(\det h)^{\beta'})_o $$ for some $\beta' > c_o$, thus $F\notin   \mathcal{E}(h(\det h)^{\beta'})_o  $  by the definition of $c_o$.
	Then ${C}_{F,c_o^+}(D) > 0$ by Lemma \ref{lem 0}.
	
	Note that ${C}_{F,\beta}(D)\ge {C}_{F,c_o^+}(D)> 0$ for $\beta > c_o$, then by the Fubini-Tonelli theorem and Proposition \ref{prop convex estimate},
\begin{align*}
	\int_D |F|^2_h
	&= \int_D |F|^2_{\frac{h}{\det h}}e^{-\varphi}\\
	&=\int_D |F|^2_{\frac{h}{\det h}}+ \int_{0}^{+\infty} \left(\int_{\{\varphi< -t\}} |F|^2_{\frac{h}{\det h}}\right) e^t dt\\
	&\ge {C}_{F,\beta}(D)+\int_{0}^{+\infty} G_\beta(t) e^t dt\\
	&\ge {C}_{F,\beta}(D)+\int_{0}^{+\infty} G_\beta(0)(1-e^{-g_\beta(t)}) e^t dt\\
	&= {C}_{F,\beta}(D)\left(1+ \int_{0}^{+\infty} (1-e^{-g_\beta(t)}) e^t dt  \right)\\
	&\ge  {C}_{F,c_o^+}(D)  \left(1+ \int_{0}^{+\infty} (1-e^{-g_\beta(t)}) e^t dt  \right).
	\end{align*}
	
	Obviously, $g_\beta(t)=\int_{t}^{+\infty} \frac{ds}{se^{1+(1+\beta)s}}$ is decreasing in $\beta$.
	By the monotone convergence theorem,
\begin{align*}
	\theta(c_o)
	&=1+ \int_{0}^{+\infty} (1-e^{-g_{c_o}(t)}) e^t dt\\
	&=1+\underset{\beta \rightarrow c^+_o}{{\lim}}  \int_{0}^{+\infty} (1-e^{-g_{\beta}(t)}) e^t dt~\le \frac{\int_D |F|^2_h}{{C}_{F,c_o^+}(D)}.
	\end{align*}
	Hence for $\beta \ge c_o$,
\begin{align*}
	\theta(\beta) \le \frac{\int_D |F|^2_h}{{C}_{F,c_o^+}(D)}.
	\end{align*}
	
	So if
\begin{align*}
	\theta(\beta) > \frac{\int_D |F|^2_h}{{C}_{F,c_o^+}(D)},
	\end{align*}
	then $\beta < c_o$, that is, $F\in \mathcal{E}(h(\det h)^\beta)_o$.
	
	In addition, it follows from the following lemma that we can always find $\beta > 0$, such that
\begin{align*}
	\theta(\beta) > \frac{\int_D |F|^2_h}{{C}_{F,c_o^+}(D)}.
	\end{align*}
	Especially, $\mathcal{E}_+(h)=\mathcal{E}(h)$.
\end{proof}

\begin{lem}
        $\  \theta(\beta)< +\infty$ for any $\beta > 0$ and $\underset{\beta \rightarrow 0^+}{\lim} \theta(\beta) =+\infty $.
\end{lem}

\begin{proof}
        For $\beta > 0$, noticing that $1-e^{-x}\le x$ for $x$ $\ge 0$, we have
\begin{align*}
        \int_{0}^{+\infty} (1-e^{-g_\beta(t)}) e^t dt
        &= \int_{0}^{1} (1-e^{-g_\beta(t)}) e^t dt+\int_{1}^{+\infty} (1-e^{-g_\beta(t)}) e^t dt\\
        &\le \int_{0}^{1}  e^t dt+\int_{1}^{+\infty} g_\beta(t) e^t dt\\
        &= e-1+\int_{1}^{+\infty} \int_{t}^{+\infty} \frac{ds}{se^{1+(1+\beta)s}} e^t dt\\
        &= e-1+\int_{1}^{+\infty}  \frac{(e^s-e)ds}{se^{1+(1+\beta)s}} \\
        &\le  e-1+\int_{1}^{+\infty}  \frac{ds}{se^{1+\beta s}} < +\infty.
        \end{align*}

        Since
\begin{align*}
        g_\beta(t)=\int_{t}^{+\infty} \frac{ds}{se^{1+(1+\beta)s}}=\frac{1+o(1)}{(1+\beta)te^{1+(1+\beta)t}} \4l as\ t\rightarrow +\infty,
        \end{align*}
        we can choose $N$ large enough such that
\begin{align*}
        \frac{1}{2te^{1+t}} \le g_0(t) \le \frac{2}{te^{1+t}} \4l \mathrm{when}\ t \ge N.
        \end{align*}

        And using Fatou's lemma and the fact that $1-e^{-x}\ge \frac{x}{1+x}$ for $x$ $\ge 0$, we obtain
\begin{align*}
        \underset{\beta\rightarrow 0^+}{\underline{\lim}} \int_{0}^{+\infty} (1-e^{-g_\beta(t)}) e^t dt
        &\ge \int_{N}^{+\infty} \underset{\beta\rightarrow 0^+}{\underline{\lim}} (1-e^{-g_\beta(t)}) e^t dt\\
        &\ge \int_{N}^{+\infty}  (1-e^{-g_0(t)}) e^t dt
        \ge \int_{N}^{+\infty}  \frac{g_0(t)}{1+g_0(t)} e^t dt\\
        &\ge \int_{N}^{+\infty}  \frac{\frac{1}{2te^{1+t}}}{1+\frac{2}{te^{1+t}}} e^t dt
        = \int_{N}^{+\infty}  \frac{1}{2et+4e^{-t} }  dt =+\infty.
        \end{align*}
\end{proof}

In addition, we can deduce the strong openness (Theorem \ref{soc}) quickly from Theorem \ref{main thm 1} by contradiction.

\begin{thm}\label{thm generalized Lempert}
	Suppose $E \rightarrow X$ is a holomorphic vector bundle with a singular Hermitian metric $h$ which is  bounded below by a continuous Hermitian metric and $\{h_j\}$ is a sequence of  Nakano semi-positive
	metrics on $E$. If  $h_j\ge h$ and $-\log \det h_j$ converges to $-\log\det h$ locally in measure, then $\sum_j \mathcal{E}(h_j) = \cal{E}(h)$.
\end{thm}

\begin{proof}
	It suffices to show that $\sum_j \mathcal{E}(h_j)_o \supseteq \cal{E}(h)_o$ for any $o\in X$ since $h_j\geq h$.
	And by Remark \ref{rmk ge lempert}, we may assume that $h_j$ converges to $h$ almost everywhere.
	
	Suppose that $(F, o)\in \cal{E}(h)_o\setminus\sum_j \mathcal{E}(h_j)_o$. We may assume that  $\int_D |{F}|^2_h<+\infty$ for some Stein neighborhood $D$ of $o$ and $\varphi=-\log \det h<0$
	since $h$ is bounded below by a continuous Hermitian metric.
	
	Since $h_j\ge h$ is Griffiths semi-positive and  converges to $h$ almost everywhere, we have $\frac{h_j}{\det h_j}\le \frac{h}{\det h}$ and $\frac{h}{\det h}=\lim_j\frac{h_j}{\det h_j}$ almost everywhere.
	
	We set
	$$  C_{F,h^+}(D):=\inf \bigg\{\int_D |\widetilde{F}|^2_{\frac{h}{\det h}}  \Big | \widetilde{F} \in \mathcal{O} (E)(D)\  \mathrm{and}\  (\widetilde{F}-F, o)\in \sum_j \mathcal{E}(h_j)_o\bigg\}.$$
	
	Since $\sum_j \mathcal{E}(h_j)$ is coherent and $(F, o)\notin\sum_j \mathcal{E}(h_j)_o$, similar to Lemma~\ref{lem 0},  we have $C_{F,h^+}(D)>0$.
	
	In Theorem~\ref{main thm 1}, take $F_{t_0} =F$, $h=h_j$ and $\beta=0$, then there exists $\widetilde{F}=\widetilde{F}_{t, B,j}\in \mathcal{O}(E)(D)$ such that $ (\widetilde{F}-F) \in \sum_j \mathcal{E}(h_j)_o$ and
\begin{align*}
	\int_{D} |\widetilde{F}-(1-b_{t,B}(\varphi_j))F|^2_{h_j}e^{-\chi_{t,B}(\varphi_j) }
	\le  \frac{t+B}{B}e^{1+t+B} \int_{D\cap\{-t-B<\varphi_j\le -t\}}  |F|^2_{\frac{h_j}{\det h_j}},
	\end{align*}
	where $\varphi_j=-\log \det h_j$.
	
	Then by the Cauchy-Schwarz inequality, $\chi_{t,B}<0$ and $\frac{h_j}{\det h_j}\le\frac{h}{\det h}\le h_j$,
\begin{align*}
	\frac{1}{2}\int_{D} |\widetilde{F}|^2_{\frac{h}{\det h} }-\int_{D\cap\{\varphi_j\le -t\}}|F|^2_{\frac{h}{\det h} }
	\le  \frac{t+B}{B}e^{1+t+B} \int_{D\cap\{-t-B<\varphi_j\le -t\}}  |F|^2_{\frac{h}{\det h}}.
	\end{align*}
	And then
\begin{align*}
	\frac{1}{2}C_{F,h^+}(D)-\int_{D\cap\{\varphi_j\le -t\}}|F|^2_{\frac{h}{\det h} }
	\le  \frac{t+B}{B}e^{1+t+B} \int_{D\cap\{-t-B<\varphi_j\le -t\}}  |F|^2_{\frac{h}{\det h}}.
	\end{align*}
	
	Let $G(t)=\int_{D\cap\{\varphi\le-t\}}^{}|F|^2_{\frac{h}{\det h}}$.
	Since $\varphi_j$ converges to $\varphi$ almost everywhere, we have
	$$\limsup_{j\to +\infty} \mathbbm{1}_{\{\varphi_j\le-t\}}\le \mathbbm{1}_{\{\varphi\le-t\}}$$
	 almost everywhere and
	hence
	$$\lim_{j\to +\infty} \mathbbm{1}_{\{\varphi_j\le-t\}}= \mathbbm{1}_{\{\varphi\le-t\}}$$
	almost everywhere
	since $\varphi_j\le \varphi$.
	
	Since $\int_D |{F}|^2_{\frac{h}{\det h}}\le\int_D |{F}|^2_h<+\infty$, by the Lebesgue dominated convergence theorem, letting $j\rightarrow+\infty$, we obtain that
	$$ \frac{1}{2}C_{F,h^+}(D)-G(t)\le (t+B)e^{1+t+B}\frac{G(t)-G(t+B)}{B}.$$ Hence
	$$ \frac{1}{2}C_{F,h^+}(D)-G(t)\le te^{1+t}\liminf\limits_{B\rightarrow0^+}\frac{G(t)-G(t+B)}{B}.$$
	
	Then by Lemma~\ref{lem estimate of G}, we get
\begin{align*}
	G(t)\ge \frac{1}{2}C_{F,h^+}(D)(1-e^{-g(t)}),
	\end{align*}
	where
	$g(t)=\int_{t}^{+\infty}\frac{ds}{se^{1+s}}$.
	
	Then by the Fubini-Tonelli theorem,
\begin{align*}
	+\infty> \int_D |F|^2_h
	&= \int_D |F|^2_{\frac{h}{\det h}}e^{-\varphi}\\
	&=\int_D |F|^2_{\frac{h}{\det h}}+ \int_{0}^{+\infty} \left(\int_{D\cap \{\varphi< -t\}} |F|^2_{\frac{h}{\det h}}\right) e^t dt\\
	&\ge\int_D |F|^2_{\frac{h}{\det h}}+ \int_{0}^{+\infty} \left(\int_{D\cap\{\varphi\le -t-1\}} |F|^2_{\frac{h}{\det h}}\right) e^t dt\\
	&\ge \frac{1}{2}{C}_{F,h^+}(D)\left(1+ \int_{0}^{+\infty} (1-e^{-g(t+1)}) e^t dt  \right)=+\infty,
	\end{align*}
	which is a contradiction.
\end{proof}

\begin{rem}\label{rem not NSP h}
 For such $h$ in the {\rm Theorem~\ref{thm generalized Lempert}}, we have that $\bigcup_{\beta>0} \cal{E}(he^{-\beta\varphi})=\cal{E}(h)$ although $h$ may be not  Nakano semi-positive.
    In fact, for any $(f,o)\in \cal{E}(h)_o$, we have that $(f,o)\in \sum_j \cal{E}(h_j)_o$. Due to the coherence of $\cal{E}(h_j)_o$, we can find $j_o>0$ such that $\sum_{j\le j_o} \cal{E}(h_j)_o=\sum_j \cal{E}(h_j)_o$.
    Hence $(f,o)\in\sum_{j\le j_o} \cal{E}(h_j)_o$.
    Applying {\rm Corollary~\ref{cor he^{-beta varphi}}}
    to these Nakano semi-positive singular Hermitian metrics $\{h_j\}_{j\le j_o}$,
    we obtain that $\cal{E}(h_j)_o=\cal{E}(h_je^{-\beta\varphi_j})_o$ for $j\le j_o$ for some $\beta>0$.
    Then $$(f,o)\in\sum_{j\le j_o} \cal{E}(h_j)_o
    \subset \sum_{j\le j_o}\cal{E}(h_je^{-\beta\varphi_j})_o
    \subset \cal{E}(he^{-\beta\varphi})_o.$$
\end{rem}

\section{Stability theorem}

	In this section we will prove a stability theorem for multiplier submodule sheaves and  the idea comes from \cite{GLZ2016}.
	
    Let $E$ be a trivial holomorphic vector bundle on an open set $o\in\Omega\subset\C^n$, $h$ be a singular Hermitian metric on $E$ satisfying:
\begin{itemize}
    \item[(1)] $h$ is bounded below by a continuous Hermitian metric;
    \item[(2)]  $\varphi:=-\log\det h$ equals to some plurisubharmonic function almost everywhere;
    \item[(3)] $\cal{E}(h)$ is coherent;
    \item[(4)] $\bigcup_{\beta>0}\cal{E}(he^{-\beta\varphi})=\cal{E}(h)$.
    \end{itemize}

    Assume that $\varphi(o)=-\infty$.
    Let $\{e_k\}_{k=1}^{k_0}$ be a minimal generating set of $\cal{E}(h)_o$, which means that no proper set of $\{e_k\}_{k=1}^{k_0}$ generates $\cal{E}(h)_o$.
    Then by the above condition $(4)$, there exists $\varepsilon_0>0$ and a neighborhood $U$ of $o$ such that  $\int_U |e_k|^2_{h}e^{-\varepsilon_0\varphi}<+\infty$ for $1\le k\le k_0$.

    We can find a Stein open set $o\in D\Subset U$ such that $D=\{\bar{z}\ \big|\ z\in D\} $ and  $\varphi<0$ on $D$.
    We consider the inner space
    $$\mathcal{H}(D,o)=\left\{f\in \mathcal{O}(E)(D)\ \Big|\ \int_D |f|^2_{\frac{h}{\det h}}<+\infty\ \mathrm{and}\ (f,o)\in \cal{E}(h)_o \right\}$$
    with the inner product  $\llangle\cdot,\cdot\rrangle_{\frac{h}{\det h}}
    = \int_D \langle\cdot,\cdot\rangle_{\frac{h}{\det h}}$.
    It follows from (1) (2) and Lemma~\ref{lem int of |f|^2e^phi dominates |f|} that $\mathcal{H}(D,o)$ is a Hilbert space.

    Then $\{e_k\}_{k=1}^{k_0}\subset \cal{H}(D,o)$ and we may assume that $\{e_k\}_{k=1}^{k_0}$ is orthonormal in $\cal{H}(D,o)$ by the Schmidt orthogonalization  since $\{e_k\}_{k=1}^{k_0}$ is linearly independent.
    So $\{e_k\}_{k=1}^{k_0}$ can be extended to an orthonormal basis $\{e_k\}_{k=1}^{\infty}$ of $\cal{H}(D,o)$.

    Using the same argument in the case of line bundles, one can prove that for any $V\Subset D$ there exists $k_V\ge k_0$ and $C_V>0$ such that
    $$\sum_{k=1}^\infty |e_k|^2_{\frac{h}{\det h}}\le C_V\sum_{k=1}^{ k_V}|e_k|^2_{\frac{h}{\det h}}\ {\rm {on}}\ V.$$
    In fact, by taking $V$ small enough, we can request that $k_V=k_0$, which is sharp.

\begin{lem}
     For any $V\Subset D$ there exists $k_V\ge k_0$ and $C_V>0$ such that
     $$\sum_{k=1}^\infty |e_k|^2_{\frac{h}{\det h}}\le C_V\sum_{k=1}^{ k_V}|e_k|^2_{\frac{h}{\det h}}\ {\rm {on}}\ V.$$
     Moreover, there is a neighborhood $W$ of $o$ and $C>0$ such that
    $$\sum_{k=1}^\infty |e_k|^2_{\frac{h}{\det h}}\le C\sum_{k=1}^{ k_0}|e_k|^2_{\frac{h}{\det h}}\ {\rm {on}}\ W.$$
\end{lem}

\begin{proof}
    Let Eva$^j_z$ be the evaluation map on $\cal{H}(D,o)$: $f\longmapsto f^j(z)$ where $f=(f^1,f^2,\cdots,f^r)$.
    It follows from the proof of Lemma~$\ref{lem int of |f|^2e^phi dominates |f|}$-(1) that Eva$^j_z$ is a bounded linear form on $\cal{H}(D,o)$ for each $z\in D$.

    By the Riesz representation theorem, there exists $K_j(\cdot,z)\in \cal{H}(D,o)$ such that $${\rm Eva}^j_z(f)=\llangle f(\cdot),K_j(\cdot,z)\rrangle_{\frac{h}{\det h}}.$$
    Especially,
\begin{align*}
    K_i^j(z,w)=\mathrm{Eva}^j_z(K_i(\cdot,w))=\llangle K_i(\cdot,w),K_j(\cdot,z)\rrangle_{\frac{h}{\det h}}
    \end{align*}
    and
\begin{align*}
    e_k^j(z)=\mathrm{Eva}^j_z(e_k)=\llangle e_k(\cdot),K_j(\cdot,z)\rrangle_{\frac{h}{\det h}}.
    \end{align*}

    Using the representation $$K_j(z,w)=\sum_{k=1}^\infty\llangle K_j(\cdot,w),e_k(\cdot)\rrangle_{\frac{h}{\det h}}e_k(z)=\sum_{k=1}^\infty \overline{e_k^j(w)}e_k(z)$$ and taking $w=z$, we obtain that
    $$ \sum_{k=1}^\infty |e_k^j(z)|^2=K_j^j(z,z)= \|K_j(\cdot,z)\|^2_{\frac{h}{\det h}}.$$

    On the other hand,
    $$\|K_j(\cdot,z)\|_{\frac{h}{\det h}}
    =\|\mathrm{Eva}_z^j\|
    =\sup \{|f^j(z)|\ :\  f\in \cal{H}(D,o) \ \mathrm{and}\ \|f\|_{\frac{h}{\det h}}\le 1  \}$$
     is bounded on any compact subset.
    So $\sum_{k=1}^\infty|e_k(z)|^2$ is uniformly convergent on any compact subset.

    Due to the strong Noetherian property of coherent sheaves of modules, we know that the family of the subsheaves $\{\cal{E}_N\}$ of $(\cal{O}_{D\times D})^{r^2} $ generated by
    $$\left\{ \Big(e_k^i(z)\overline{e_k^j(\overline{w})} \Big)_{1\le i,j\le r}\ \ \Big|\  1\le k\le N \right\}$$
    on $D\times D$ is stable on any relative compact subset of $D\times D$.

    Given $V\Subset D$, since $D=\{\bar{z}\ \big|\ z\in D\} $, there is $ V_0\Subset D$ such that
    $V\Subset V_0$, $V_0=\{\bar{z}\ \big|\ z\in V_0 \}$ and
    $$\mathrm{Gen}\left\{ \Big(e_k^i(z)\overline{e_k^j(\overline{w})} \Big)_{1\le i,j\le r}\ \ \Big|\  k\ge 1\ \right\}
    =\cal{E}_{N_0}$$
    on ${V_0\times V_0}$ for some $N_0\ge 1$.

    We claim that $N_0\ge k_0$.
    Otherwise, we may assume that $k_0\ge 2$ and
    $$ \Big(e_{k_0}^i(z)\overline{e_{k_0}^j(\overline{w})} \Big)_{1\le i,j\le r}
    =\sum_{k=1}^{k_0-1} c_{k}(z,w)\Big(e_k^i(z)\overline{e_k^j(\overline{w})} \Big)_{1\le i,j\le r}$$
    on some neighborhood of $o\times o$ where $c_{k}(z,w)$ is holomorphic near $o\times o$.

    It follows from the minimality of $\{e_k\}_{k=1}^{k_0}$ that $e_{k_0}^{j_0}(\overline{w}_0)\not=0$ for some $j_0$ and some point $w_0$ near $o$.
    Then
    $$ e_{k_0}
    =\sum_{k=1}^{k_0-1} \frac{c_{k}(\cdot,w_0)\overline{e_k^{j_0}(\overline{w}_0)}}{ \overline{e_{k_0}^{j_0}(\overline{w}_0)}} e_k
    $$
    holds on some neighborhood of $o$, which is  contradictory to the minimality of $\{e_k\}_{k=1}^{k_0}$.

    In addition,
    $$\Big|\sum_{k=m_1}^{m_2} e_k^i(z)\overline{e_k^j(\overline{w})}\Big|^2
    \le \sum_{k=m_1}^{m_2} \Big|e_k^i(z)\Big|^2
    \sum_{k=m_1}^{m_2} \Big|e_k^j(\overline{w})\Big|^2$$
    for any $m_2\ge m_1$ by the Cauchy-Schwartz inequality.
    Therefore $\sum_{k=1}^{\infty} e_k^i(z)\overline{e_k^j(\overline{w})}$ is uniformly convergent on any compact subset of $D\times D$.
    Hence $\left(\sum_{k=1}^{\infty} e_k^i(z)\overline{e_k^j(\overline{w})}\right)_{1\le i,j\le r}$ is a holomorphic section of the sheaf
    $\mathrm{Gen}\left\{ \Big(e_k^i(z)\overline{e_k^j(\overline{w})} \Big)_{1\le i,j\le r}\ \ \Big|\  k\ge1\ \right\}$ on $D\times D$.
    Then
    $$\sum_{k=1}^{\infty} \Big(e_k^i(z)\overline{e_k^j(\overline{w})} \Big)_{1\le i,j\le r}=\sum_{k=1}^{N_0} a_{k}(z,w)\Big(e_k^i(z)\overline{e_k^j(\overline{w})} \Big)_{1\le i,j\le r}$$
    on $V_0\times V_0$ where $a_{k}(z,w)$ is holomorphic on $ V_0\times V_0$.

    Taking $w=\bar{z}$ we obtain that the following holds on $V$
\begin{align*}
    \sum_{k=1}^{\infty} |e_k(z)|^2_{\frac{h}{\det h}(z)}
    &=\sum_{k=1}^{\infty} \sum_{i,j}e_k^i(z)\overline{e_k^j(z)}\frac{h_{ij}}{\det h}(z)\\
    &=\sum_{k=1}^{N_0} \sum_{i,j}a_k(z,\bar{z})e_k^i(z)\overline{e_k^j(z)}\frac{h_{ij}}{\det h}(z)\\
    &=\sum_{k=1}^{N_0} a_k(z,\bar{z})|e_k(z)|^2_{\frac{h}{\det h}(z)}\\
    &\le C_V\sum_{k=1}^{N_0} |e_k(z)|^2_{\frac{h}{\det h}(z)},
    \end{align*}
    where $C_V=\sup_{z\in V}\sum_{k=1}^{N_0} |a_k(z,\bar{z})| $.
    So taking $k_V=N_0$ we get that $$\sum_{k=1}^{\infty} |e_k|^2_{\frac{h}{\det h}}
    \le C_V\sum_{k=1}^{k_V} |e_k|^2_{\frac{h}{\det h}}
    ~{\rm on} ~V,
    $$
    which completes the proof of the first part.

    For the second part, we only need to consider the case of $k_V>k_0$. Recalling the choice of $\{e_k\}_{k=1}^{k_0}$, since $(e_k,o)\in\cal{E} (h)_o$,
    we get that $e_k=\sum_{l=1}^{k_0}g_{k,l}e_l$ for $k_0<k\le k_V$ on some neighborhood $W\subseteq V$ of $o$, where $g_{k,l}(z)$ is holomorphic on some neighborhood of $\overline W$.
    Then $$|e_k|^2_{\frac{h}{\det h}}
    \le k_0C_{0}^{2}\sum_{k=1}^{k_0} |e_k|^2_{\frac{h}{\det h}} \ {\rm on}\ W,$$
    where ${C_0}=\sup \{|g_{k,l}(z)|\ \big| z\in W, \ k_0<k\le k_V, 1\le l \le k_0\}$.
    Therefore,
    $$\sum_{k=1}^{\infty} |e_k|^2_{\frac{h}{\det h}}
    \le C\sum_{k=1}^{k_0} |e_k|^2_{\frac{h}{\det h}} \ {\rm on} \ W,$$
    where $C=k_0k_VC_VC_{0}^{2}$.
\end{proof}

\begin{thm}\label{main thm3 more general}
    Let $h$ satisfy $(1)$-$(4)$ and $h_j$ be Griffiths semi-positive on $E$ with $h_j\ge h$ and $F_j\in \cal{O}(E)(\Omega)$. Assume that \\
    $(1)$ $\varphi_j=-\log\det h_j$ converges to $\varphi=-\log\det h$ locally in measure,\\
    $(2)$ $(F_j,o)\in \cal{E}(h)_o$ and $F_j$ compactly converges to $F\in \cal{O}(E)(\Omega)$.\\
    Then $|F_j|^2_{h_j}$ converges to $|F|^2_h$ in $L^p$ in some neighborhood of $o$ for some $p>1$.
\end{thm}

\begin{proof}
    Due to the assumption that $F_j$ compactly converges to $F$ on $\Omega$, we have that $\{F_j\}$ is uniformly bounded on $D$ which means $ |F_j|^2_{\frac{h}{\det h}}<M$ on $D$ for some $M>0$  independent of $j$.
    Moreover, $F_j\in \cal{H}(D,o)$ since $(F_j,o)\in \cal{E}(h)_o$.

    Then $F_j=\sum_k a_j^ke_k$ where $a_j^k\in \mathbb{C}$ and hence for any $j$
    $$\sum_{k=1}^\infty |a^k_j|^2=\int_D|F_j|^2_{\frac{h}{\det h}}<\mathrm{Vol}(D)M.$$
    It follows from the Cauchy-Schwartz inequality that
    $\sum_k|a^k_je_k|^2_{\frac{h}{\det h}}\le\sum_k|a^k_j|^2\sum_k|e_k|^2_{\frac{h}{\det h}}$.

    Then we get the following uniform bound:
\begin{align*}
    \int_{ W} |F_j|^2_h e^{-\varepsilon_0\varphi}
    &\le\int_{W} \sum_{k=1}^\infty|a^k_j|^2\sum_{k=1}^\infty |e_k|^2_h e^{-\varepsilon_0\varphi}\\
    &\le C\sum_{k=1}^\infty|a^k_j|^2\int_{W} \sum_{k=1}^{k_0}|e_k|^2_h e^{-\varepsilon_0\varphi}\\
    &\le CM\mathrm{Vol}(D)\int_{ W} \sum_{k=1}^{k_0}|e_k|^2_h e^{-\varepsilon_0\varphi}<+\infty.
    \end{align*}

    Hence
\begin{align*}
    \int_{ W} (|F_j|^2_h)^{1+\varepsilon_0}
    &=\int_{ W} (|F_j|^2_{\frac{h}{\det h}})^{1+\varepsilon_0}e^{-(1+\varepsilon_0)\varphi}\\
    &\le M^{\varepsilon_0} \int_{W} |F_j|^2_he^{-\varepsilon_0\varphi}\\
    &\le CM^{1+\varepsilon_0}\mathrm{Vol}(D)\int_{W} \sum_{k=1}^{k_0}|e_k|^2_h e^{-\varepsilon_0\varphi}<+\infty.
    \end{align*}

    Since $h_j\ge h$ implies $\frac{h_j}{\det h_j}\le \frac{h}{\det h}$ by Lemma~\ref{lem h_j>h}, we get
    $$  |F_j|^2_h\le|F_j|^2_{h_j}\le|F_j|^2_he^{\varphi-\varphi_j}.
    $$

    We claim that for any $1<p<1+\varepsilon_0$, $|F_j|^2_he^{\varphi-\varphi_j}-|F_j|^2_h$ converges to 0 in $L^{p}_{\rm loc}(W)$ and hence $|F_j|^2_{h_j}-|F_j|^2_h$ converges to 0 in $L^{p}_{\rm loc}(W)$.

    In fact, taking $p'$ with $p<p'<1+\varepsilon_0$, for any $\widetilde{W}\Subset W$, by H\"older's inequality, we have
\begin{align*}
    \int_{\widetilde{W}} \left(|F_j|^2_{h}\left(e^{\varphi-\varphi_j}-1\right)\right)^{p'}
    \le\Big(\int_{\widetilde{W}} (|F_j|^2_h)^{p_0p'}\Big)^{1/p_0}
    \Big(\int_{\widetilde{W}} e^{q_0p\varphi-q_0p'\varphi_j}\Big)^{1/q_0},
    \end{align*}
    where $p_0=(1+\varepsilon_0)/p'>1$ and $1/p_0+1/q_0=1$.

    It follows from Lemma~\ref{lem e^(phi-phi_j) bounded} that $\int_{\widetilde{W}} e^{q_0p'\varphi-q_0p'\varphi_j}$ is bounded in $j$, since $\varphi_j\le \varphi$ converges to $\varphi$ locally in measure
    and $\varphi$ equals to a plurisubharmonic function almost everywhere.

    Moreover, we can prove that $|F_j|^2_{h}\left(e^{\varphi-\varphi_j}-1\right)$ converges to $0$ locally in measure in $W$ due to the assumption, then $|F_j|^2_he^{\varphi-\varphi_j}-|F_j|^2_h$
    converges to 0 in $L^{p}(\widetilde{W})$ by the result in \cite{Wade74}.

    Similarly, $|F_j|^2_h$ converges to $|F|^2_h$ in $L^p_{\rm loc}(W)$.
    Therefore we obtain that $|F_j|^2_{h_j}$ converges to $|F|^2_h$ in $L^p_{\rm loc}(W)$ for $1<p<1+\varepsilon_0$.
\end{proof}

\begin{rem}
    From the proof of the above theorem, we can see that for any $0<p<1+c_o^F(h)$, $|F_j|^2_{h_j}$ converges to $|F|^2_h$ in $L^p$ in some neighborhood of $o$.
\end{rem}

    From Theorem~\ref{main thm2,usual case} and Remark~\ref{rem not NSP h}, we immediately obtain the following corollaries.
\begin{cor}\label{cor 1}
    Let $h$ be  locally Nakano semi-positive and $h_j$ Griffiths semi-positive on $E$ with $h_j\ge h$ and $F_j\in \cal{O}(E)(\Omega)$. Assume that \\
    $(1)$ $\varphi_j=-\log\det h_j$ converges to $\varphi=-\log\det h$ locally in measure,\\
    $(2)$ $(F_j,o)\in \cal{E}(h)_o$ and $F_j$ compactly converges to $F\in \cal{O}(E)(\Omega)$.\\
    Then $|F_j|^2_{h_j}$ converges to $|F|^2_h$ in $L^p$ in some neighborhood of $o$ for any $0<p<1+c_o^F(h)$.
\end{cor}

\begin{cor}\label{cor 2}
    Let $h$ be singular Hermitian metric bounded below by a continuous Hermitian metric.
    Assume that there is a sequence of Nakano semi-positive singular Hermitian metrics $\widetilde{h}_k\ge h$ converging to $h$ locally in measure.
    If $h_j$ is Griffiths semi-positive on $E$ with $h_j\ge h$ and $F_j\in \cal{O}(E)(\Omega)$ satisfying that  \\
    $(1)$ $\varphi_j=-\log\det h_j$ converges to $\varphi=-\log\det h$ locally in measure,\\
    $(2)$ $(F_j,o)\in \cal{E}(h)_o$ and $F_j$ compactly converges to $F\in \cal{O}(E)(\Omega)$.\\
    Then $|F_j|^2_{h_j}$ converges to $|F|^2_h$ in $L^p$ in some neighborhood of $o$ for any $0<p<1+c_o^F(h)$.
\end{cor}

\section{Examples}
In this section, firstly we provide a standard way of  constructing Griffiths/Nakano semi-positive  metrics by holomorphic sections. Secondly, we generalize Berndtsson's outstanding work \cite{Bern09} on direct image bundles to the singular setting. Finally, we discuss some failed approaches to  approximating a Nakano semi-positive singular Hermitian metric by Nakano positive smooth Hermitian metrics.

\subsection{Metrics induced by sections}

\begin{exa}
    Let $E$ be a  holomorphic vector bundle of rank $r$ over an n-dimensional complex manifold $X$. Let $\{F_\alpha\}_{\alpha\in I}\subset \mathcal{O}(E)(X)$ be a family of global holomorphic sections of $E$ satisfying that \\
    $(1)$ $I$ is at most countable and $\sum_{\alpha} |F_{\alpha}(z)|^2$ compactly converges  on $X$; \\
    $(2)$ $\dim \mathrm{Span}_{\alpha\in I}\{F_\alpha(z)\}=r$ outside a proper analytic subset $A\subset X$. \\
    Then we can define a Hermitian metric $h$ on $E$ such that for any local holomorphic section $G$ of $E^*$,
     $$|G|^2_{h^*}:=\sum_{\alpha}|G(F_\alpha)|^2.$$

     In fact, let $e=\{e_j\}^r_{j=1}$ be a holomorphic frame of $E$, write $F_\alpha=\sum_j F_{\alpha,j}e_j$, then $h^*=_{e^*}H$ where
    $$
    H=\sum_{\alpha} (F_{\alpha,j}\overline{F}_{\alpha,k})_{r\times r}
    =
\begin{pmatrix}
    \sum_{\alpha} F_{\alpha,1}\overline{F}_{\alpha,1}
    & \cdots
    &\sum_{\alpha} F_{\alpha,1}\overline{F}_{\alpha,r}\\
    \vdots &\ddots & \vdots\\
    \sum_{\alpha} F_{\alpha,r}\overline{F}_{\alpha,1}
    & \cdots
    &\sum_{\alpha} F_{\alpha,r}\overline{F}_{\alpha,r}\\
    \end{pmatrix}.
    $$

    It follows from the condition $(1)(2)$ that $\det h<+\infty$ on $X$ and $\det h>0$ on $X\setminus A$. Then $h$ is a singular Hermitian metric on $E$.
    Since $G(F_\alpha)$ is holomorphic function, then
     $h^*$ is Griffiths semi-negative and smooth Griffiths semi-negative on $X\setminus A$.

    Moreover, $(E^*\otimes (\det E)^{s}, {h^*}\otimes {(\det h)^s})$, for $s\ge \min\{n,r\}$,  and $({\rm Sym}^mE\otimes \det E,{\rm Sym}^mh\otimes\det h)$, for $m\ge 1$,   are  Nakano semi-positive  on $X$ due to {\rm Proposition \ref{pro hdeth}}.
\end{exa}

\subsection{Direct image bundles}
Let us recall some notations in \cite[\S 2.1]{Bern15} and \cite[\S 4.1]{DNWZ20} .		
Consider a projective fibration $\pi:Y^{n+m}\rightarrow X^n$ between two complex manifolds and a holomorphic vector bundle $E$ on $Y$ with a singular Hermitian metric $h$. Assume that the direct image sheaf $F:=\pi_{*}(K_{Y/X}\otimes E)$ is locally free,
then  $F_t=H^0(Y_t,K_{Y_t}\otimes E|_{Y_t})$ and  an $L^2$-metric $h_{F}$ on $F$ is defined as follows,
\begin{equation*}
	\langle u,v\rangle_{h_F}
	:=\int_{Y_t}	\langle u,v\rangle_{\omega_{Y_t},h}dV_{\omega_{Y_t}}=\int_{Y_t}i^{m^2}\sum_{\alpha,\beta}h_{\alpha\beta}u_\alpha\wedge \bar v_\beta,
\end{equation*}
where $u,v\in F_t$ and $Y_t=\pi^{-1}(t)$ and $\omega_{Y_t}$ be arbitrary K\"ahler metric on $Y_t$.
The $L^2$-metric is independent of the choice of $\omega_{Y_t}$.
Essentially following the idea of \cite[Theorem 4.3]{DNWZ20}, we obtain that
\begin{thm}
	Assume that $(E,h_E)$ is Nakano semi-positive and the $L^2$-metric $h_{F}$ on $F:=\pi_{*}(K_{Y/X}\otimes E)$ defined
	above is Griffiths semi-positive, then  $h_{F}$ is locally Nakano semi-positive.
\end{thm}
\begin{proof}For any $t\in X$, there exists a Stein coordinate neighbourhood $U$ of $t$ such that
$F|_U$ is  trivial and  there is a hypersurface $H$ in $\pi^{-1}(U)$ such that
$\pi^{-1}(U)\setminus H$ is Stein and
  \item  $E|_{\pi^{-1}(U)\setminus H}$ is trivial.

  Let $\omega_U$ and $\omega_{\pi^{-1}(U)}$ be  arbitrary K\"ahler metrics on $U$ and $\pi^{-1}(U)$ respectively  and $\psi$ be any
	smooth strictly plurisubharmonic function on $U$.
In order to prove that $h_{F}$ is locally Nakano semi-positive, it suffices to show that  for any
	$\dpa$-closed $f\in C^{\infty}_{(n,1)}(U,F)$ with $$\int_{U} \langle B^{-1}_\psi f,f\rangle_{\omega_U,h_{F}}e^{-\psi}dV_{\omega_U}<+\infty,$$ there is a $u\in L^2_{(n,0)}(U,F,h_{F})$ satisfying that
$$\int_{U} |u|^2_{\omega_U,h_{F}}e^{-\psi}dV_{\omega_U}\le\int_{U} \langle B^{-1}_\psi f,f\rangle_{\omega_U,h_{F}}e^{-\psi}dV_{\omega_U}.$$

	We can write
	$f(t) = dt \wedge \left(f_1(t)d\bar t_1 +\cdots+ f_n(t)d\bar t_n\right)$,
	with $f_j(t) \in  F_t = H^0(Y_t, K_{Y_t} \otimes E|_{Y_t})$.
	One can	identify $f$ as a smooth $(n+m, 1)$-form
	$\tilde{f}(t,z) := dt\wedge (f_1(t,z)d\bar t_1+\cdots + f_n(t,z)d\bar t_n)$	on $Y$,
	with $f_j(t,z)$ being holomorphic section of $K_{Y_t} \otimes E|_{Y_t}$ for fixed $t\in U$.
	We have the following observations:
\begin{itemize}
	\item[(a)] $\dpa_z f_j(t,z) = 0$ for any fixed $t \in U$,
	since $f_j(t,z)$ are holomorphic sections  $K_{Y_t} \otimes  E |_{Y_t}$.
	\item[(b)] $\dpa_t f = 0$, since $f$ is a $\dpa $-closed form on $U$.
\end{itemize}
	It follows that $\tilde{f}$ is a $\dpa$ -closed smooth $(n + m, 1)$-form on $\pi^{-1}(U)$ with values in $E$.
	We want to solve the equation $\dpa \tilde u = \tilde{f} $ on $\pi^{-1}(U)$.

On Stein open subset $\pi^{-1}(U)\setminus H$, $(E,h)$ is Nakano semi-positive and $\pi^*\psi$ is smooth plurisubharmonic.
Since $f$ is $F$-valued $(n,1)$-form, $\langle B^{-1}_\psi  f,  f\rangle_{\omega_{U},  h_{F}}dV_{\omega_{U}}$ is independent of the choice of $\omega_U$ and thus
               $$\langle B^{-1}_\psi  f,  f\rangle_{\omega_{U},  h_{F}}dV_{\omega_{U}}=
              \sum_{j,k=1}^{n} \psi^{jk}\langle f_j,f_k\rangle_{h_{F}}i^{n^2}dt\wedge d\bar t,$$
              where $(\psi^{jk}) = (\frac{\pa^2\psi}{\pa t_j\pa\bar t_k})^{-1}$.

And  since each $f_j|_{Y_t}$ is holomorphic  $E|_{Y_t}$-valued $(m,0)$-form for any fixed $t\in U$,
\begin{align*}
   \int_{Y_t}\langle B^{-1}_{\pi^*\psi} \tilde f,  \tilde f\rangle_{\omega_{\pi^{-1}(U)}|_{Y_t},  h}dV_{\omega_{\pi^{-1}(U)}|_{Y_t}}= & \sum_{j,k=1}^{n}\psi^{jk}(t)\left(\int_{Y_t}\langle f_j,f_k\rangle_{\omega_{\pi^{-1}(U)}|_{Y_t},h}dV_{\omega_{\pi^{-1}(U)}|_{Y_t}}\right)i^{n^2}dt\wedge d\bar t \\
  = &\sum_{j,k=1}^{n}\psi^{jk}(t)\langle f_j,f_k\rangle_{h_{F}}i^{n^2}dt\wedge d\bar t.
\end{align*}

Then by the Fubini-Tonelli theorem,
\begin{align*}
            \int_{\pi^{-1}(U)}^{}\langle B^{-1}_{\pi^*\psi}\tilde f,\tilde f\rangle_{\omega_{\pi^{-1}(U)},  h}e^{-\pi^*{\psi}}dV_{\omega_{\pi^{-1}(U)}}  =  \int_{U}^{}\langle B^{-1}_\psi  f,  f\rangle_{\omega_{U},  h_{F}}e^{-\psi}dV_{\omega_{U}}<+\infty.
\end{align*}
By the Nakano semi-positivity of $(E,h)$ and Lemma \ref{modifition}, there is a $\tilde u\in L^2_{(n+m,0)}(\pi^{-1}(U)\setminus H, E, h)$ such that $\dpa\tilde u=\tilde f$ and
\begin{equation}\label{aab}
   \int_{\pi^{-1}(U)\setminus H}^{}|\tilde u|^2_{\omega_{\pi^{-1}(U)},h}e^{-\pi^*{\psi}}dV_{\omega_{\pi^{-1}(U)}}\leq  \int_{U} \langle B^{-1}_\psi f,f\rangle_{\omega_U,h_{F}}e^{-\psi}dV_{\omega_U}.
\end{equation}
	By \cite[Lemma~$\rm{\Rmnum{8}}$-$(7.3)$]{Demaillybook2012}, $\dpa \tilde u=\tilde f$  can be extended trivially on $\pi^{-1}(U)$.
In addition,  by the weak regularity of $\dpa$ on $(n+m,0)$-forms, we can take $\tilde u$ to be smooth. Write $\tilde u=dt\wedge u(t,z)$
where $u(t,\cdot)$ is a section of $K_{Y_t}\otimes E$ for fixed $t$. Since $\dpa \tilde u = \tilde{f} $, we get that
 $\dpa_z \tilde u |_{Y_t} = 0$ for any fixed $t \in X$. Hence we have
	 $ u(t, \cdot) \in F_t$.
	Therefore we may view $\tilde u $ as a section $u$ of $K_X\otimes F$. And it is obvious that $\dpa u = f$ .
Moreover, by the  Fubini-Tonelli theorem and $(\ref{aab})$, we obtain that
$$\int_{U}^{}|u|^2_{\omega_U,h_{h_E}}e^{-\psi}dV_{\omega_U}\leq  \int_{U} \langle B^{-1}_\psi f,f\rangle_{\omega_U,h_{F}}e^{-\psi}dV_{\omega_U}.$$
\end{proof}
\begin{rem}
  \begin{enumerate}
    \item[$(1)$] Assume that $\det h_E|_{Y_t}\not\equiv+\infty$ for every $t\in X$, then  {\rm Theorem \ref{thm l2ext}} implies that the direct image sheaf $\pi_{*}(K_{Y/X}\otimes E)$ is in fact a vector bundle;
    \item[$(2)$]  When $E$ is a line bundle, \cite{HPS18} shows that the $L^2$ metric is Griffiths semi-positive and \cite{watanabe} shows that the  $L^2$ metric is locally Nakano semi-positive.
  \end{enumerate}
\end{rem}

\subsection{Failed approximating approaches}
	Given a Nakano semi-positive singular Hermitian metric $h$ on a holomorphic vector bundle $E$,
	Inayama asked whether there exists locally a sequence of Nakano semi-positive smooth Hermitian metrics increasingly converging to $h$ (\cite[Question 7.1]{In20}).

	Hosono \cite[Example 4.4]{Hosono17} constructed an example of Griffiths semi-positive singular metric $h$ such that the
	standard approximation defined by convolution of $h^*$ does not have uniformly bounded curvature from below in the sense of Nakano.
	Let $X=\C^2$, $\omega$ be the standard metric on $X$ and $E=X\times \C^2$  a trivial $2$-bundle over $X$.
	Two sections $F_1=(1, 0)$ and $F_2=(z_1, z_2)$ of $E$ induce a Griffiths semi-negative metric $h^*$ on $E^*=X\times \C^2$,
	$$h^*=
	\begin{pmatrix}
	1+|z_1|^2    &z_1\bar z_2\\
	\bar z_1 z_2 &|z_2|^2
	\end{pmatrix}$$
	and its dual metric
	$$h=
	\frac{1}{|z_2|^2}
	\begin{pmatrix}
	|z_2|^2   &-\bar z_1 z_2 \\
	-z_1\bar z_2  &1+|z_1|^2
	\end{pmatrix}.$$
	
	Consider the following  approximation of $h^*$:
\begin{align*}
	h'^*_\varepsilon
	&=\begin{pmatrix}
	1+|z_1|^2+\varepsilon    &z_1\bar z_2\\
	\bar z_1 z_2 &|z_2|^2+\varepsilon
	\end{pmatrix}
	=h^*+\varepsilon
	\begin{pmatrix}
	1 &0\\
	0 &1
	\end{pmatrix}.
\end{align*}
	Let $h'_\varepsilon$ be the dual metric of $h'^*_\varepsilon$.
	Then $h'_\varepsilon$ is a Griffiths semi-positive smooth metric and increasingly converges to $h$ as $\varepsilon \to 0^+$.
In fact, $h$ is Nakano semi-positive.
	Let's go to the details.
	The Chern curvature $\Theta(h)$ of a smooth Hermitian metric $h$ on a rank $r$ bundle $E$ over manifold $X^n$ is written as
	$$\Theta(h)=\dpa (\bar h^{-1}\pa  \bar h)=\sum \Theta_{jk}dz_j\wedge d\bar z_k.$$
	Then $\Theta(h)\ge_\Nak 0$  if and only if
	$$\im \Theta(h)^Th=\sum \Theta_{jk}^Th ~\im dz_j\wedge d\bar z_k$$ is semi-positive, that is,
	in the language of linear algebra, the curvature matrix
	\begin{align*}
	\widetilde{\Theta}(h)=
	\begin{pmatrix}
	\Theta_{11}^Th & \Theta_{12}^Th & \cdots & \Theta_{1n}^Th\\
	\Theta_{21}^Th & \Theta_{22}^Th & \cdots & \Theta_{2n}^Th\\
	\vdots & \vdots & \ddots & \vdots\\
	\Theta_{n1}^Th & \Theta_{n2}^Th & \cdots & \Theta_{nn}^Th\\
	\end{pmatrix}
	\end{align*}
	is semi-positive.
	A computation gives that
	$$ \Theta_{jk}^Th=h\frac{\pa^2 h^{-1}}{\pa z_j \pa \bar z_k}h
	-h\frac{\pa h^{-1}}{\pa \bar z_k}h\frac{\pa h^{-1}}{\pa z_j}h.$$
	Noticing that $h$ is nondegenerate everywhere,
	such observation gives that
	$\im \Theta(h)^Th$ is semi-positive if and only if
	$\im h^{-1}\Theta(h)^T$ is semi-positive, or equivalently the block matrix
	\begin{align}\label{formula reduced curvature}
	\hat{\Theta}(h)=\left(\frac{\pa^2 h^{-1}}{\pa z_j \pa \bar z_k}
	-\frac{\pa h^{-1}}{\pa \bar z_k}h\frac{\pa h^{-1}}{\pa z_j}\right)_{j,k}
	\end{align}
	is semi-positive.
	
	Now go back to Hosono's example, he computed  the curvature matrix
	$$\widetilde\Theta(h'_\ve)=-\frac{\ve(\ve+1)}{(\ve^2+\ve+\ve |z_1|^2+\ve |z_2|^2+|z_2|^2)^3}M',$$
	where $M'$ is
\begin{align*}	
	\left(
	\begin{array}{cccc}
		-(\ve+|z_2|^2)^2 & z_2 \bar z_1 (\ve+|z_2|^2) & z_2 \bar z_1 (\ve+|z_2|^2) & -z_2^2 \bar z_1^2 \\
		\bar z_2 z_1(\ve+|z_2|^2) & -|z_2|^2 |z_1|^2 & -(\ve+|z_2|^2) (\ve+|z_1|^2+1) & z_2 \bar z_1 (\ve+|z_1|^2+1) \\
		\bar z_2 z_1(\ve+|z_2|^2) & -(\ve+|z_2|^2) (\ve+|z_1|^2+1) & -|z_2|^2 |z_1|^2 & z_2 \bar z_1 (\ve+|z_1|^2+1) \\
		-\bar z_2^2 z_1^2 & \bar z_2 z_1 (\ve+|z_1|^2+1) & \bar z_2 z_1(\ve+|z_1|^2+1) & -(\ve+|z_1|^2+1)^2 \\
	\end{array}
	\right).
\end{align*}	
	Using his result,  we notice that $h$ is smooth outside the analytic set $\{z_2=0\}$, and by setting $\ve=0$ we get that the curvature matrix
	$\widetilde\Theta(h)=0$ and hence $h$ is Nakano semi-positive outside the analytic set $\{z_2=0\}$.
	Therefore $h$ is Nakano semi-positive on $X$ by Lemma \ref{thm NSP outside analytic set}.

	In addition, Hosono showed that $h'^*_\varepsilon$ is obtained by convolution of $h^*$ with a smooth kernel function $\rho$ (concretely, there is a positive constant $C_\rho$ such that $\rho_\epsilon* h^*=h^*+C_\rho\epsilon^2 I_2$).
	The value of the curvature matrix $\widetilde\Theta(h'_\ve)$ at the origin is
	\begin{align}\label{matrix curvature}	
	\frac{1}{\ve^2(\ve+1)^2}\left(
	\begin{array}{cccc}
	\ve^2 & 0 & 0 & 0 \\
	0 & 0 & \ve (\ve+1) & 0 \\
	0 & \ve(\ve+1) & 0 & 0 \\
	0 & 0 &  & (\ve+1)^2 \\
	\end{array}
	\right),
	\end{align}
	so for any fixed $C>0$,   $\im \Theta(h'_\varepsilon)+C\omega\otimes {\rm Id}_E$ is not semi-positive in the sense of Nakano for $\varepsilon>0$ small enough.

Now let us recall the standard convolution process for line bundles.  A positive metric $h$ is locally written as $h=e^{-\phi}$ with $\phi\in\Psh$.  Let $\phi_\varepsilon=\phi*\rho_\varepsilon$ for some appropriate approximation kernel $\rho_\varepsilon$,
	then $\phi_\varepsilon\in C^\infty\cap \Psh$ decreasingly converges to $\phi$ as $\varepsilon\to0$. Thus the new metric $h=e^{-\phi_\varepsilon}$ is smooth positive and increasingly converges to $h$.
	In fact, it is an approximation for the weight.

	Hosono's approach can be viewed as a convolution for  the metric $h^*$ but not for the weight.
	It suggests that, in the case of vector bundles, we should also convolute the weight of the metric.
	By the polarization identity,
	$$   4h^*_{jk}=\sum_{\alpha=1}^{4} \im^\alpha |e^*_j+\im^\alpha e^*_k|^2_{h^*},$$
	where $e_1=(1,0), e_2=(0,1)$ form a frame.
	Due to the Griffiths positivity of $h$,
	$$\log (|e^*_j+\im^\alpha e^*_k|^2_{h^*})^,s$$ are plurisubharmonic, which can be regarded as weights of $h$.
	Naturally, we try to construct $h^*_\varepsilon$ by convoluting  $\log (|e^*_j+\im^\alpha e^*_k|^2_{h^*})$, that is,
	\begin{align*}
	4h^*_{jk,\varepsilon}
	&:=\sum_{\alpha=1}^{4} \im^\alpha \left(|e^*_j+\im^\alpha e^*_k|^2_{h^*}\right)_\varepsilon\\
	&:=\sum_{\alpha=1}^{4} \im^\alpha {\rm Exp}\big(\rho_\varepsilon*\log(|e^*_j+\im^\alpha e^*_k|^2_{h^*})\big),
	\end{align*}
	where $\rho_\varepsilon$ is a smooth radial approximation kernel and $\Exp(t)=e^t$.
	Unfortunately, such natural smooth approximation  sequences are not Nakano semi-positive in any neighborhood of the
	origin as well.
	Concretely, the simplified curvature matrix $\hat{\Theta}(h_\varepsilon)$ has negative eigenvalues at the origin as $\varepsilon\rightarrow0^+$.
	Suggested by the expression (\ref{matrix curvature}) in the above example, we shall show that the corresponding submatrix of  $\hat{\Theta}(h_\varepsilon)$ is not semi-positive.
	Since 	\begin{align*}
		4h^{-1}_{jk,\varepsilon}
		&=\sum_{\alpha} \im^{-\alpha} \left(|e_j^*+\im^\alpha e_k^*|^2_{h^*}\right)_\varepsilon\\
		&=\sum_{\alpha} \im^{-\alpha} \Exp\left(\rho_\varepsilon*\log(|e_j^*+\im^\alpha e_k^*|^2_{h^*})\right)
		\end{align*}
and
		\begin{align*}
		|e^*_1+\im^\alpha e^*_1|^2_{h^*}&=|1+\im^\alpha|^2(1+|z_1|^2),\\
		|e^*_1+\im^\alpha e^*_2|^2_{h^*}&=1+|z_1+\im^\alpha z_2|^2,\\
		|e^*_2+\im^\alpha e^*_2|^2_{h^*}&=|1+\im^\alpha|^2|z_2|^2,
		\end{align*}
		we get that
		\begin{align*}
		h^{-1}_{11,\varepsilon}=\Exp\left(\rho_\varepsilon* \log(1+|z_1|^2)\right),
		\end{align*}
		due to
		\begin{align}\label{formula i}
		\sum_{\alpha=1}^{4} \im^{m\alpha}
		=\left\{
		\begin{array}{rcl}
		4 & &{\rm if}~ ~4 | m,\\
		0 & &{\rm if}~ ~4 \nmid m.
		\end{array}
		\right.
		\end{align}
		
		Similarly,
		\begin{align*}
		h^{-1}_{22,\varepsilon}&=\Exp\left(\rho_\varepsilon* \log(|z_2|^2)\right),\\
		4h^{-1}_{12,\varepsilon}(0)
		&=\sum_{\alpha} \im^{-\alpha} \Exp\left(\rho_\varepsilon*\log(1+|z_1+\im^\alpha z_2|^2)\right) (0)=0.
		\end{align*}
		For simplicity, we take $\rho_\varepsilon(z_1,z_2)=\hat \rho_\varepsilon(z_1)\hat \rho_\varepsilon(z_2)$, where $ \hat{\rho}_\varepsilon(z_1)\ge 0$ is a smooth radial approximation kernel on $\C$.
		Then
		\begin{align*}
		\log h^{-1}_{22,\varepsilon}
		&=\rho_\varepsilon* \log(|z_2|^2)\\
		&=\int_{0}^{\infty} \hat \rho_\varepsilon(r) \big(\int_{0}^{2\pi} \log|z_2-re^{\im \theta}|^2d\theta\big) rdr\\
		&=2\pi \int_{0}^{\infty}   \hat\rho_\varepsilon(r) \log\max\{|z_2|^2,r^2\} rdr.
		\end{align*}	
Noticing that $h^{-1}_{22,\varepsilon}(0)=\Exp(2\pi\int_0^\infty \hat \rho_\varepsilon(r)r\log r^2dr)>0$,
one can show that
$$\lim_{z_2\to0}\frac{h^{-1}_{22,\varepsilon}(z)-h^{-1}_{22,\varepsilon}(0)}{|z_2|^2}=\pi\hat
\rho_\varepsilon(0)h^{-1}_{22,\varepsilon}(0).$$	

Thus $$h^{-1}_{22,\varepsilon}=\left(1+\pi\hat \rho_\varepsilon(0)|z_2|^2+o(|z_2|^2)\right)
		\Exp(2\pi\int_0^\infty \hat \rho_\varepsilon(r)r\log r^2dr).$$
		Since $\det h^{-1}_\varepsilon(0)>0$, by continuity there is a neighborhood of $0$, which possibly depends on $\varepsilon$, such that $h_\varepsilon$ are smooth metrics in this neighborhood.

		Next we get
		\begin{align*}
		4\frac{\pa h^{-1}_{12,\varepsilon}}{\pa \bar z_1}(0)
		&=\sum_{\alpha} \im^{-\alpha} \Exp\left(\rho_\varepsilon*\log(1+|z_1+\im^\alpha z_2|^2)\right)
		\cdot \rho_\varepsilon *\left( \frac{z_1+\im^{\alpha}  z_2 }{1+|z_1+\im^\alpha z_2|^2}   \right) (0)=0,
		\end{align*}
		where we use (\ref{formula i}) and the fact that the two convolutions are independent of $\alpha$.
		
		Similarly, we have
		$\frac{\pa h^{-1}_{12,\varepsilon}}{\pa z_1}(0)
		=\frac{\pa h^{-1}_{12,\varepsilon}}{\pa z_2}(0)
		=\frac{\pa h^{-1}_{12,\varepsilon}}{\pa \bar z_2}(0)=0$
		and
		the first derivatives of $h^{-1}_{11,\varepsilon}$ at $0$ equal to $0$.
		
		To summarize, all first derivatives of $h^{-1}_\varepsilon$ at $0$ are $0$, so
		\begin{align*}
		\hat{\Theta}_{22}(h_\varepsilon)(0)
		=\frac{\pa^2 h^{-1}_{22,\varepsilon}}{\pa z_1 \pa \bar z_1}(0)=0,
		\end{align*}	
		
		\begin{align*}
		\hat{\Theta}_{33}(h_\varepsilon)(0)
		=\frac{\pa^2 h^{-1}_{11,\varepsilon}}{\pa z_2 \pa \bar z_2}(0)=0
		\end{align*}
		and {\small{
		\begin{align*}
		 4\frac{\pa^2 h^{-1}_{12,\varepsilon}}{\pa z_2\pa \bar z_1}(0)
		=&\sum_{\alpha} \im^{-\alpha} \Exp\left(\rho_\varepsilon*\log(1+|z_1+\im^\alpha z_2|^2)\right)
		\cdot \rho_\varepsilon *\left( \frac{\im^\alpha }{1+|z_1+\im^\alpha z_2|^2}
		- \frac{\im^\alpha|z_1+\im^\alpha z_2|^2 }{(1+|z_1+\im^\alpha z_2|^2)^2}\right) (0)\\
		=&\sum_{\alpha} \Exp\left(\rho_\varepsilon*\log(1+|z_1+\im^\alpha z_2|^2)\right)
		\cdot \rho_\varepsilon *\left( \frac{1 }{(1+|z_1+\im^\alpha z_2|^2)^2}   \right) (0)
		>0.
		\end{align*}}}
		Thus
		\begin{align*}
		C:=\hat{\Theta}_{23}(h_\varepsilon)(0)
		=\hat{\Theta}_{32}(h_\varepsilon)(0)
		=\frac{\pa^2 h^{-1}_{12,\varepsilon}}{\pa z_2 \pa \bar z_1}(0)>0.
		\end{align*}
		Hence
		\begin{align*}
		\hat{\Theta}(h_\varepsilon)(0)=
		\begin{pmatrix}
		* & * & * & *\\
		* & 0 & C & *\\
		* & C & 0 & *\\
		* & * & * & *\\
		\end{pmatrix}
		\end{align*}
		is not semi-positive.
		So $h_\varepsilon$ is not Nakano semi-positive near $0$.

In summary, for such a Nakano semi-positive singular Hermitian metric $h$, the above three approximating sequences given by
\begin{enumerate}
  \item  adding a small constant multiple of identity matrix to the metric $h^*$,
  \item  convoluting the metric $h^*$ with an appropriate smooth kernel function,
  \item  convoluting  the weights $\log (|e^*_j+\im^\alpha e^*_k|^2_{h^*})$ with an appropriate smooth kernel function
\end{enumerate}
are no longer Nakano semi-positive.

\end{document}